\pdfoutput=1
\documentclass[11pt]{article} 
\usepackage{xspace}

\newcommand{\mathprog}[1]{foo}
\newcommand{\arxiv}[1]{ba}
\newcommand{\nips}[1]{bar}
\ifdefined \SubmitToMathprog
\renewcommand{\mathprog}[1]{#1}%
\renewcommand{\arxiv}[1]{}%
\renewcommand{\nips}[1]{}%
\else%
\ifdefined \NipsWokshopFTW
	\renewcommand{\mathprog}[1]{}%
	\renewcommand{\arxiv}[1]{}%
	\renewcommand{\nips}[1]{#1}%
\else
	\renewcommand{\mathprog}[1]{}%
	\renewcommand{\arxiv}[1]{#1}%
	\renewcommand{\nips}[1]{}%
\fi%
\fi%

\mathprog{\RequirePackage{fix-cm}}
\mathprog{\usepackage{newtxtext}} %
\usepackage{latexsym}

\usepackage{enumerate}
\usepackage{fancyhdr,amsmath,amssymb,url}
\arxiv{\usepackage{amsthm}}
\nips{\usepackage{amsthm}}

\usepackage{booktabs}
\usepackage{tabularx}

\usepackage{float}
\usepackage{multirow}
\usepackage[normalem]{ulem}
\arxiv{\usepackage{accents}}
\usepackage[]{graphicx}				%
\usepackage{amssymb}
\usepackage{color}
\usepackage{xspace}
\usepackage[numbers,square]{natbib}
\usepackage{array}

\usepackage{mathtools}
\usepackage{thmtools}
\usepackage{thm-restate}

\usepackage{algorithm}
\usepackage[noend]{algpseudocode}

\usepackage{etoolbox}

\arxiv{\usepackage[top=1in, right=1in, left=1in, bottom=1in]{geometry}}
\mathprog{\usepackage[top=1.25in, right=1.25in, left=1.25in, 
bottom=1.25in]{geometry}}
\arxiv{\usepackage[hidelinks]{hyperref}}
\mathprog{\usepackage[colorlinks = true,
	linkcolor = blue,
	urlcolor  = blue,
	citecolor = blue,
	anchorcolor = blue]{hyperref}}

\usepackage{placeins}

\mathprog{\smartqed}
\mathprog{\usepackage{geometry}}

\usepackage{tikz}

\usepackage{xfrac}
\usepackage[nice]{nicefrac}

\mathprog{
\let\oldproof\proof
\let\oldendproof\endproof
\def\proof{\begingroup \oldproof}
\def\endproof{\qed \oldendproof \endgroup}
}

\ifdefined\SubmitToMathprog
\makeatletter
\newcommand\subparagraph{%
	\@startsection{subparagraph}{5}
	{\parindent}
	{3.25ex \@plus 1ex \@minus .2ex}
	{-1em}
	{\normalfont\normalsize\bfseries}}
\makeatother
\usepackage{titlesec}
\let\subparagraph\relax %
\titleformat{\section}{\fontsize{12.5pt}{\baselineskip}\bfseries}{\thesection}{0.5em}{}
\titleformat{\subsection}{\fontsize{11pt}{\baselineskip}\bfseries}{\thesubsection}{0.5em}{}
\titleformat{\paragraph}[runin]{\normalfont\bfseries}{\theparagraph}{}{}[~~~]
\fi

\arxiv{
\theoremstyle{plain}
\newtheorem{conjecture}{Conjecture}
\newtheorem{theorem}{Theorem}
\newtheorem{lemma}{Lemma}
\newtheorem{claim}{Claim}
\newtheorem{proposition}{Proposition}
\newtheorem{corollary}{Corollary}
\newtheorem{definition}{Definition}
\newtheorem{assumption}{Assumption}

\newcounter{remark}

\newenvironment{remark}[1][]{
	\refstepcounter{remark}
	\ifthenelse{\isempty{#1}}{%
		\noindent \textbf{Remark \theremark:}\hspace*{.05em}
	}{%
		\noindent \textbf{Remark \theremark} ({#1})\textbf{:}\hspace*{.05em}
	}
}{%
	$\clubsuit$ \bigskip
}

\theoremstyle{remark}
\newtheorem{example}{Example}
}

\nips{
	\theoremstyle{plain}
	
	\newtheorem{theorem}{Theorem}
	\newtheorem{lemma}{Lemma}

	\newtheorem{definition}{Definition}

	\newcounter{remark}

	\theoremstyle{remark}

}

\newtheorem{observation}{Observation}
\newenvironment{proof-of-lemma}[1][{}]{%
  \noindent\emph{Proof of Lemma 
  {#1}.}\hspace*{.5em}}{\arxiv{\qed\medskip\\}\mathprog{\qed}}

\newtheorem{innercustomlemma}{Lemma}
\newenvironment{customlemma}[1]
{\renewcommand\theinnercustomlemma{#1}\innercustomlemma}
{\endinnercustomlemma}

\newcommand{\mc}[1]{\mathcal{#1}}

\newcommand{\wt}[1]{\widetilde{#1}}  %
\newcommand{\norm}[1]{\left\|{#1}\right\|} %
\newcommand{\normbig}[1]{\big\|{#1}\big\|} %
\newcommand{\norms}[1]{\|{#1}\|} %

\newcommand{\inner}[2]{\langle #1, #2\rangle}
\newcommand{\innerbig}[2]{\left\langle #1, #2\right\rangle}

\newcommand{\absinner}[2]{|\inner{#1}{#2}|}
\newcommand{\absinnerbig}[2]{\left|\innerbig{#1}{#2}\right|}

\newcommand{\R}{\mathbb{R}}
\newcommand{\N}{\mathbb{N}}

\newcommand{\minimize}{\mathop{\rm minimize}}
\newcommand{\argmin}{\mathop{\rm arg\hspace{.1em}min}}

\newcommand{\support}[1]{\mathop{\mathrm{supp}} \left\{#1\right\}}
\newcommand{\supports}[1]{\mathop{\mathrm{supp}} \{#1\}}

\newcommand{\half}{\frac{1}{2}}

\newcommand{\defeq}{:=}

\newcommand{\openright}[2]{\left[{#1}, {#2}\right)}

\def\ie{{i.e.}\ }
\def\eg{{e.g.}\ }

\newcommand{\unitvec}[1]{\frac{#1}{\norm{#1}}}
\newcommand{\unitvecsmall}[1]{{#1}/{\norm{#1}}}

\newcommand{\Sm}[1]{L_{#1}}
\newcommand{\SmGrad}{\Sm{1}}
\newcommand{\SmHess}{\Sm{2}}

\newcommand{\Smp}{\Sm{p}}

\newcommand{\smC}[1]{\ell_{#1}}
\newcommand{\smChat}[1]{\hat{\ell}_{#1}}
\newcommand{\smCvar}[1]{\tilde{\ell}_{#1}}

\newcommand{\ind}[1]{^{(#1)}}
\newcommand{\indic}[1]{1_{\left({#1}\right)}}

\newcommand{\DeltaF}{\Delta}

\newcommand{\FclassBlank}{\mathcal{F}}
\newcommand{\Fclass}[1]{\mathcal{F}_{#1}(\DeltaF, \Sm{#1})}

\newcommand{\FclassD}[1]{\mc{F}^{\rm dist}_{#1}(D, \Sm{#1})}
\newcommand{\FclassDTwo}[2]{\mc{F}^{\rm dist}_{#1, #2}(D, \Sm{#1},  \Sm{#2})}

\newcommand{\ConvClassCustom}[1]{\mc{K}_1\left(\DeltaF, \SmGrad\right)}

\newcommand{\tensordim}[2]{\otimes^{#1} {#2}}
\newcommand{\opnorm}[1]{\norm{#1}_{\rm op}}
\newcommand{\opnorms}[1]{\norms{#1}_{\rm op}}
\newcommand{\opt}{^\star}
\newcommand{\<}{\langle}
\renewcommand{\>}{\rangle}
\newcommand{\deriv}[1]{\nabla^{{#1}}}

\newcommand{\alg}{\mathsf{A}}
\newcommand{\algzr}{\mathsf{Z}_{\alg}}

\newcommand{\AlgZR}{\mathcal{A}_{\textnormal{\textsf{zr}}}}
\newcommand{\AlgDet}{\mathcal{A}_{\textnormal{\textsf{det}}}}
\newcommand{\AlgRand}{\mathcal{A}_{\textnormal{\textsf{rand}}}}

\newcommand{\TimeEps}[3][\epsilon]{\mathsf{T}_{#1}\big(#2, #3\big)}
\newcommand{\CompEps}[2]{\mathcal{T}_{\epsilon}\big(#1, #2\big)}

\newcommand{\compactfunc}{\Psi}
\newcommand{\gausscdf}{\Phi}
\newcommand{\gausspdf}{\phi}

\newcommand{\orthogonalgroup}{\mathsf{O}}

\newcommand{\uniform}{\mathsf{Uni}}

\newcommand{\Otil}[1]{\wt{O}\left(#1\right)}

\newcommand{\wb}[1]{\overline{#1}}

\newcommand{\Dim}{d}

\newcommand{\T}{T}
\newcommand{\fhard}{\bar{f}_T}
\newcommand{\hhard}{\bar{h}_T}
\newcommand{\hhardVar}{{h}}

\newcommand{\fhardRot}{\tilde{f}_{T;U}}
\newcommand{\fhardBound}{\hat{f}_{T;U}}
\newcommand{\fhardBoundI}{\hat{f}_{T;I}}

\newcommand{\countset}{\mathsf{C}}

\newcommand{\sign}{\mathop \mathrm{sign}}
\newcommand{\grad}{\nabla}
\newcommand{\hess}{\nabla^2}

\newcommand{\floor}[1]{\left\lfloor #1 \right\rfloor}
\newcommand{\ceil}[1]{\left\lceil #1 \right\rceil}

\newcommand{\hide}[1]{ }

\newcommand{\E}{\mathbb{E}}
\newcommand{\del}{\partial}
\newcommand{\tr}{\mathrm{tr}}
\renewcommand{\P}{\mathbb{P}}

\arxiv{
\theoremstyle{plain}

\theoremstyle{plain}

\theoremstyle{definition}

\theoremstyle{plain}

}

\makeatletter
\newcommand{\hightarget}[1]{\Hy@raisedlink{\hypertarget{#1}{}}}
\makeatother

\usepackage{xr}
\usepackage{multirow}

\mathprog{\journalname{Mathematical Programming}}

\begin{document}

\arxiv{
\title{Lower Bounds for Finding Stationary Points I}

\author{Yair Carmon ~~~ John C.\ Duchi ~~~ Oliver Hinder ~~~ Aaron Sidford \\
  \texttt{\{\href{mailto:yairc@stanford.edu}{yairc},%
    \href{mailto:jduchi@stanford.edu}{jduchi},%
    \href{mailto:ohinder@stanford.edu}{ohinder},%
    \href{mailto:sidford@stanford.edu}{sidford}\}@stanford.edu}}
\date{}
}

\mathprog{
\title{Lower Bounds for Finding Stationary Points I
	\thanks{OH was supported by the PACCAR INC fellowship. YC and JCD 
	were partially supported by the SAIL-Toyota Center for AI Research, 
	NSF-CAREER award 1553086, and a Sloan Foundation Fellowship in 
	Mathematics. YC was partially supported by the Stanford Graduate 
	Fellowship and the Numerical Technologies Fellowship.}
}

\author{Yair Carmon  \and
            John C.\ Duchi   \and
            Oliver Hinder   \and
            Aaron Sidford }

\authorrunning{Carmon, Duchi, Hinder and Sidford} %

\institute{Y. Carmon \at
	Department of Electrical Engineering, Stanford University, Stanford, CA 
	94305, USA \\
	\email{yairc@stanford.edu}           %
	\and
	J.C. Duchi \at
	Departments of Statistics and Electrical Engineering, Stanford University, 
	Stanford, CA 94305, USA
	\email{jduchi@stanford.edu}
	\and
	O. Hinder \at
	Department of Management Science and Engineering, Stanford University, 
	Stanford, CA 94305, USA, Stanford University, Stanford, CA 94305, USA
	\email{ohinder@stanford.edu}
	\and
	A. Sidford \at
	Department of Management Science and Engineering, Stanford University, 
	Stanford, CA 94305, USA, Stanford University, Stanford, CA 94305, USA
	\email{sidford@stanford.edu}
}

\date{Received: December 13, 2017}
}

\maketitle

\begin{abstract}
  We prove lower bounds on the complexity of finding $\epsilon$-stationary
  points (points $x$ such that $\|\nabla f(x)\| \le \epsilon$) of smooth,
  high-dimensional, and potentially non-convex functions $f$.  We consider
  oracle-based complexity measures, where an algorithm is given access to
  the value and all derivatives of $f$ at a query point $x$. We
  show that for any (potentially randomized) algorithm $\mathsf{A}$, there
  exists a function $f$ with Lipschitz $p$th order derivatives such that
  $\mathsf{A}$ requires at least $\epsilon^{-(p+1)/p}$ queries to find an
  $\epsilon$-stationary point.
   Our lower bounds are sharp to within
   constants, and they show that gradient descent,
  cubic-regularized Newton's method, and generalized $p$th order
  regularization are worst-case optimal within their natural function classes.
  \mathprog{
   	}
\end{abstract}

\mathprog{\newpage}

\section{Introduction}

Consider the optimization problem
\begin{equation*}
  \minimize_{x \in \R^{\Dim}} ~ f(x)
\end{equation*}
where $f : \R^{\Dim} \rightarrow \R$ is smooth, but possibly
non-convex. In general, it is intractable to even approximately
minimize such $f$~\cite{NemirovskiYu83,MurtyKa87}, so---following an 
established line of research---we consider the problem of
finding an $\epsilon$-stationary point of $f$,
meaning some $x \in \R^d$ such that
\begin{equation}
  \label{eqn:what-we-want}
  \norm{\nabla f(x)} \le \epsilon.
\end{equation}
We prove lower bounds on the number of function and derivative evaluations  required for algorithms to find a point
$x$ satisfying inequality~\eqref{eqn:what-we-want}.  While for arbitrary
smooth $f$, a near-stationary point~\eqref{eqn:what-we-want} is certainly
insufficient for any type of optimality, there are a number of reasons to
study algorithms and complexity for finding stationary points.  In several statistical and engineering problems, including regression models with
non-convex penalties and objectives~\cite{LohWa12, LohWa13}, phase
retrieval~\cite{CandesLiSo15, SunQuWr18}, and non-convex (low-rank)
reformulations of semidefinite programs and matrix
completion~\cite{BurerMo03, KeshavanMoOh10a, BoumalVoBa16}, it is possible
to show that all first- or second-order stationary points are (near) global
minima.  The strong empirical success of local search strategies for such
problems, as well as for neural networks~\cite{LeCunBeHi15}, motivates a
growing body of work on algorithms with strong complexity guarantees for
finding stationary points~\cite{NesterovPo06, BirginGaMaSaTo17,
  CarmonDuHiSi18, AgarwalAlBuHaMa17, CarmonDuHiSi17}.  In contrast to this
algorithmic progress, algorithm-independent lower bounds %
for finding stationary points are largely unexplored.

Even for non-convex functions $f$, it is possible to find
$\epsilon$-stationary points for which the number of function and
derivative evaluations is polynomial in $1/\epsilon$ and the dimension $d$ of $\mathrm{dom}f$. Of
particular interest are methods for which the number of function and
derivative evaluations does not depend on $d$, but instead depends on measures of $f$'s regularity. The
best-known method with such a \emph{dimension-free} convergence guarantee is
classical gradient descent: for every (non-convex) function $f$ with
$\SmGrad$-Lipschitz gradient satisfying $f(x\ind{0}) - \inf_x f(x) \le
\DeltaF$ at the initial point $x\ind{0}$, gradient descent finds an
$\epsilon$-stationary point in at most $2 \SmGrad\DeltaF \epsilon^{-2}$
iterations~\cite{Nesterov04}.  Under the additional assumption that $f$ has
Lipschitz continuous Hessian, our work~\cite{CarmonDuHiSi18} and
\citet{AgarwalAlBuHaMa17} exhibit randomized first-order methods that
find an $\epsilon$-stationary point in time scaling as $\epsilon^{-{7}/{4}}
\log \frac{d}{\epsilon}$ (igoring other problem-dependent constants).
In subsequent work~\cite{CarmonDuHiSi17}, we show a different
deterministic accelerated gradient method that
achieves dimension-free complexity $\epsilon^{-{7}/{4}} \log
\frac{1}{\epsilon}$, and if
$f$ additionally has Lipschitz third derivatives,
then $\epsilon^{-{5}/{3}} \log \frac{1}{\epsilon}$ iterations
suffice to find an $\epsilon$-stationary point.

By evaluation of higher order derivatives, such as the Hessian, it is
possible to achieve better $\epsilon$
dependence. \citeauthor{NesterovPo06}'s cubic regularization of Newton's
method~\cite{NesterovPo06, CartisGoTo10} guarantees
$\epsilon$-stationarity~\eqref{eqn:what-we-want} in $\epsilon^{-3/2}$ iterations,
but each iteration may be expensive
when the dimension $d$ is large. More generally, $p$th-order regularization
methods iterate by sequentially minimizing models of $f$ based on order $p$
Taylor approximations, and \citet{BirginGaMaSaTo17} show that these 
methods
converge in $\epsilon^{-(p+1)/p}$ iterations. Each iteration requires
finding an approximate stationary point of a high-dimensional, potentially non-convex, degree 
$p + 1$
polynomial, which suggests that the methods will be practically challenging
for $p > 2$. The methods nonetheless provide fundamental upper complexity
bounds.

In this paper and its companion~\cite{NclbPartII}, we focus on the converse
problem: providing dimension-free complexity lower bounds for finding
$\epsilon$-stationary points. We show fundamental limits on the best
achievable $\epsilon$ dependence, as well as dependence on other problem
parameters. Together with known upper bounds, our results shed light on
the optimal rates of convergence for finding stationary points.

\subsection{Related lower bounds}

In the case of \emph{convex} optimization, we have a deep understanding of
the complexity of finding $\epsilon$-suboptimal points, that is, $x$ satisfying 
$f(x)
\le f(x\opt) + \epsilon$ for some $\epsilon > 0$, where $x\opt \in \argmin_x
f(x)$. Here we review only the dimension-free optimal rates, as those are
most relevant for our results. Given a point $x\ind{0}$ satisfying
$\norms{x\ind{0} - x\opt} \le D < \infty$, if $f$ is convex with
$\SmGrad$-Lipschitz gradient, Nesterov's accelerated gradient method finds
an $\epsilon$-suboptimal point in $\sqrt{\SmGrad} D \epsilon^{-1/2}$
gradient evaluations, which is optimal even among randomized, higher-order
algorithms~\cite{Nesterov83, NemirovskiYu83, Nesterov04,
  WoodworthSr16}.\footnote{Higher order methods can yield improvements under
  additional smoothness: if in addition $f$ has $\SmHess$-Lipschitz Hessian
  and $\epsilon \le \SmGrad^{7/3}\SmHess^{-4/3}D^{2/3}$, an accelerated
  Newton method achieves the (optimal) rate $(\SmHess
  D^3/\epsilon)^{2/7}$~\cite{ArjevaniShSh17,MonteiroBe13}.} For 
  non-smooth
problems, that is, when $f$ is $\Sm{0}$-Lipschitz, subgradient methods
achieve the optimal rate of $\Sm{0}^2 D^2 / \epsilon^2$ subgradient
evaluations (cf.~\cite{BraunGuPo17, NemirovskiYu83, Nesterov04}).  In Part
II of this paper~\cite{NclbPartII}, we consider the impact of convexity on
the difficulty of finding stationary points using first-order methods.

Globally optimizing smooth non-convex functions is of course intractable:
\citet[\S 1.6]{NemirovskiYu83} show that for functions $f : \R^d \to \R$
with Lipschitz $1$st through $p$th derivatives, and algorithms receiving all
derivatives of $f$ at the query point $x$, the worst case complexity of
finding $\epsilon$-suboptimal points scales at least as $(1 /
\epsilon)^{d/p}$.  This exponential scaling in $d$ shows that dimension-free
guarantees for achieving near-optimality in smooth non-convex functions are
impossible to obtain.

Less is known about lower bounds for finding stationary points for $f : \R^d
\to \R$. \citet{Nesterov12b} proposes lower bounds for finding stationary
points under a box constraint, but his construction does not extend to the
unconstrained case when $f(x\ind{0})-\inf_x f(x)$ is bounded.
\citet{Vavasis93} considers the complexity of finding $\epsilon$-stationary
points of functions with Lipschitz derivatives in a first-order (gradient
and function-value) oracle model. For such problems, he proves a lower bound
of $\epsilon^{-1/2}$ oracle queries that applies to any deterministic
algorithm operating on certain two-dimensional functions. This appears to be
the first algorithm-independent lower bound for approximating stationary
points of non-convex functions, but it is unclear if the bound is tight, even
for functions on $\R^2$.

A related line of work considers algorithm-dependent lower bounds,
describing functions that are challenging for common classes of algorithms,
such as Newton's method and gradient descent. In this vein, \citet{Jarre11}
shows that the Chebyshev-Rosenbrock function is difficult to optimize, and
that any algorithm that employs line search to determine the step size will
require an exponential (in $\epsilon$) number of iterations to find an
$\epsilon$-suboptimal point, even though the Chebyshev-Rosenbrock function has only a
single stationary point. While this appears to contradict the polynomial
complexity guarantees mentioned above, \citet{CartisGoTo13} explain this by
showing that the difficult Chebyshev-Rosenbrock instances have
$\epsilon$-stationary point with function value that is
$\omega(\epsilon)$-suboptimal.  Cartis et al.\ also develop
algorithm-specific lower bounds on the iteration complexity of finding
approximate stationary points. Their works~\cite{CartisGoTo10,CartisGoTo12}
show that the performance guarantees for gradient descent and cubic
regularization of Newton's method are tight for two-dimensional functions
they construct, and they also extend these results to certain structured
classes of methods~\cite{CartisGoTo12b,CartisGoTo17}.

\subsection{The importance of high-dimensional 
constructions}\label{sec:intro-importance-of-high-d}

To tightly characterize the algorithm- and dimension-independent
complexity of finding $\epsilon$-stationary points, one \emph{must} 
construct hard instances whose domain has dimension that grows with 
$1/\epsilon$.  The reason for 
this is simple: there exist algorithms with 
complexity that trades dependence on dimension $d$ in favor of better 
$1/\epsilon$ dependence. Indeed, \citet{Vavasis93} gives a grid-search
method that, for functions 
with Lipschitz gradient, finds an $\epsilon$-stationary point in 
$\max\{2^d,\epsilon^{-2d/(d+2)}\}$ gradient and function evaluations. 
Moreover, \citet{Hinder18} 
exhibits a cutting-plane method that, for functions with Lipschitz first  
and third derivatives, finds an  $\epsilon$-stationary point in $d \cdot  
\epsilon^{-4/3}\log \frac{1}{\epsilon}$ gradient and function evaluations. 

High-dimensional constructions are similarly unavoidable when developing 
lower bounds in convex optimization. There, the center-of-gravity
cutting plane 
method (cf.~\cite{Nesterov04}) finds an $\epsilon$-suboptimal point in 
$d\log\frac{1}{\epsilon}$  %
(sub)gradient evaluations, for any continuous 
convex function with bounded distance to optimality. Consequently, proofs of the dimension-free lower bound for 
convex optimization (as we cite in the previous section) all rely on 
constructions   
whose dimensionality grows polynomially in $1/\epsilon$. 

Our paper continues this well-established practice, and our lower bounds 
apply in the following order of quantifiers:
for all $\epsilon > 0$, there exists a dimension
$d \in \N$ such that for any $d' \ge d$ and algorithm
$\alg$, there is some
$f:\R^{d'} \to \R$ 
such that $\alg$ requires at least $T(\epsilon)$ oracle queries to find an 
$\epsilon$-stationary point of $f$. Our bounds on deterministic 
algorithms require dimension $d=1+2T(\epsilon)$, while our bounds on all 
randomized algorithms require $d = c \cdot T(\epsilon)^2 \log 
T(\epsilon)$ for a numerical constant $c<\infty$. In contrast, the results 
of~\citet{Vavasis93} 
and~\citet{CartisGoTo10,CartisGoTo12,CartisGoTo12b,CartisGoTo17} hold 
with $d=2$ independent of $\epsilon$. Inevitably, they do so at 
a 
cost; the lower bound~\cite{Vavasis93} is loose, while the lower 
bounds~\cite{CartisGoTo10,CartisGoTo12,CartisGoTo12b,CartisGoTo17} 
apply to only 
restricted algorithm classes that exclude the aforementioned 
grid-search and cutting-plane algorithms.

\subsection{Our contributions}\label{sec:our-contributions}

In this paper, we consider the class of all randomized algorithms that
access the function $f$ through an \emph{information oracle} that returns
the function value, gradient, Hessian and all higher-order derivatives of
$f$ at a queried point $x$.
Our main result (Theorem~\ref{thm:fullder-final} in
Section~\ref{sec:fullder-random}) is as follows. Let $p\in\N$ and $\DeltaF,
\Smp$, and $\epsilon>0$. Then, for any randomized algorithm $\alg$ based on
the oracle described above, there exists a function $f$ that has
$\Smp$-Lipschitz $p$th derivative, satisfies $f(x\ind{0})-f(x\opt)\le
\DeltaF$, and is such that, with high probability, $\alg$ requires at least
\begin{equation*}
  c_p \cdot \DeltaF \Sm{p}^{1/p} 
  \epsilon^{-(p+1)/p}
\end{equation*}
oracle queries to find an $\epsilon$-stationary point of $f$, where $c_p >
0$ is a constant decreasing at most polynomially in $p$. As explained in the
previous section, the domain of the constructed function $f$ has dimension
polynomial in $1/\epsilon$.

For every $p$, our lower bound matches (up to a constant) known upper
bounds, thereby characterizing the optimal complexity of finding
stationary points. For $p=1$, our results imply that gradient
descent~\cite{Nesterov04,Nesterov12b} is optimal among all methods (even
randomized, high-order methods) operating on functions with Lipschitz
continuous gradient and bounded initial sub-optimality. Therefore, to
strengthen the guarantees of gradient descent one \emph{must} introduce
additional assumptions, such as convexity of $f$ or Lipschitz continuity of
$\deriv{2}f$. Similarly, in the case $p=2$ we establish that cubic
regularization of Newton's method~\cite{NesterovPo06,CartisGoTo10} achieves
the optimal rate $\epsilon^{-3/2}$, and for general $p$ we show that $p$th
order Taylor-approximation methods~\cite{BirginGaMaSaTo17} are optimal.

These results say little about the potential of first-order methods on
functions with higher-order Lipschitz derivatives, where first-order methods
attain rates better than $\epsilon^{-2}$~\cite{CarmonDuHiSi17}.
In Part~II of this series~\cite{NclbPartII}, we address this issue and show
lower bounds for deterministic algorithms using only first-order
information. The lower bounds exhibit a fundamental gap between first- and
second-order methods, and nearly match the known upper
bounds~\cite{CarmonDuHiSi17}.

\subsection{Our approach and paper organization}

In Section~\ref{sec:prelims} we introduce the classes of functions and
algorithms we consider as well as our notion of complexity. Then, in
Section~\ref{sec:anatomy}, we present the generic technique we use to prove
lower bound for deterministic algorithms in both this paper and Part
II~\cite{NclbPartII}. While essentially present in previous work, our technique 
abstracts away and generalizes the central arguments in many lower
bounds~\cite{NemirovskiYu83, Nemirovski94, WoodworthSr16,
  ArjevaniShSh17}. The technique applies to higher-order methods and
provides lower bounds for general optimization goals, including finding
stationary points (our main focus), approximate minimizers, and second-order
stationary points. It is also independent of whether the functions under
consideration are convex, applying to any function class with appropriate
rotational invariance~\cite{NemirovskiYu83}. The key building blocks of the
technique are Nesterov's notion of a ``chain-like''
function~\cite{Nesterov04}, which is difficult for a certain subclass of
algorithms, and a ``resisting oracle''~\cite{NemirovskiYu83,Nesterov04}
reduction that turns a lower bound for this subclass into a lower bound for
all deterministic algorithms.

In Section~\ref{sec:fullder-deterministic} we apply this generic method to
produce lower bounds for \emph{deterministic} methods
(Theorem~\ref{thm:fullder-simple}). The deterministic results underpin our
analysis for randomized algorithms, which culminates in
Theorem~\ref{thm:fullder-final} in Section~\ref{sec:fullder-random}.
Following~\citet{WoodworthSr16}, we consider random rotations of our
deterministic construction, and show that for any algorithm such a randomly
rotated function is, with high probability, difficult. For completeness, in
Section~\ref{sec:fullder-distance} we provide lower bounds on finding
stationary points of functions where $\norms{x\ind{0} - x\opt}$ is bounded,
rather than the function value gap $f(x\ind{0})-f(x\opt)$; these bounds have
the same $\epsilon$ dependence as their bounded function value counterparts.

\paragraph{Notation}
Before continuing, we provide the conventions we adopt
throughout the paper. For a sequence of vectors, subscripts denote
coordinate index, while parenthesized superscripts denote element index, e.g. 
$x\ind{i}_j$ is the $j$th coordinate of the $i$th entry in the sequence
$x\ind{1}, x\ind{2}, \ldots$.  For any $p\ge1$ and $p$ times continuously
differentiable $f:\R^d \to \R$, we let $\deriv{p} f(x)$ denote the tensor of
$p$th order partial derivatives of $f$ at point $x$, so $\deriv{p} f(x)$ is
an order $p$ symmetric tensor with entries
\begin{equation*}
  \left[ \deriv{p} f(x) \right]_{i_1, \ldots, i_p}
  = \deriv{p}_{i_1, \ldots, i_p} f(x) 
  = \frac{\del^p f}{\del x_{i_1} \cdots \del x_{i_p}}(x)
  ~~ \mbox{for~} i_j \in \{1, \ldots, d\}.
\end{equation*} 
Equivalently, we
may write $\deriv{p} f(x)$ as a multilinear operator $\deriv{p} f(x):
(\R^d)^p \to \R$, 
\begin{equation*}
  \deriv{p} f(x)\left[v\ind{1}, \ldots, v\ind{p}\right]
  = \sum_{i_1 =1}^d \cdots  \sum_{i_p =1}^d 	
  v_{i_1}\ind{1} \cdots v_{i_p}\ind{p} \frac{\del^p f}{\del x_{i_1} \cdots \del x_{i_p}}(x)
  = \innerbig{\deriv{p} f(x)}{v\ind{1} \otimes \cdots \otimes v\ind{p}},
\end{equation*} 
where $\inner{\cdot}{\cdot}$ is the Euclidean inner product on tensors,
defined for order $k$ tensors $T$ and $M$ by $\inner{T}{M} = \sum_{i_1,
  \ldots, i_k} T_{i_1, \ldots, i_k}M_{i_1, \ldots, i_k}$, and $\otimes$
denotes the Kronecker product.
We let $\tensordim{k}{d}$ denote $d \times \cdots \times d$, $k$ times,
so that $T \in \R^{\tensordim{k}{d}}$ denotes an order $k$ tensor.

For a vector $v\in\R^d$ we let $\norm{v} \defeq \sqrt{\inner{v}{v}}$ denote 
the Euclidean norm of $v$.  
For a tensor $T \in
\R^{\tensordim{k}{d}}$, the Euclidean operator norm of $T$ is 
\begin{equation*}
  \opnorm{T} \defeq 
  \sup_{v\ind{1}, \ldots, v\ind{k}}
  \Big\{\inner{T}{v\ind{1} \otimes \cdots \otimes v\ind{k}}
  = \sum_{i_1, \ldots, i_k} T_{i_1,\ldots, i_k} v_{i_1}\ind{1}
  \cdots v_{i_k}\ind{k}
  \mid \norms{v\ind{i}} \le 1, i=1,\ldots,k \Big\}.
\end{equation*}
If $T$ is a symmetric order $k$ tensor, meaning that $T_{i_1, \ldots, i_k}$
is invariant to permutations of the indices (for example, $\deriv{k} f(x)$
is always symmetric), then \citet[Thm.~2.1]{ZhangLiQi12} show that
\begin{equation}
  \label{eq:prelims-eckhart-young-tensors}
  \opnorm{T} = \sup_{\norm{v} = 1} \big|\inner{T}{v^{\otimes k}}\big|,
  ~~~
  \mbox{where} ~~~
  v^{\otimes k} = \underbrace{v \otimes v \otimes \cdots \otimes v}_{
    k~{\rm times}}.
\end{equation}
For vectors, the
Euclidean and operator norms are identical. 

For any $n\in\N$, we let $[n] \defeq \{1, \ldots, n\}$ denote the set of
positive integers less than or equal to $n$. We let $\mc{C}^\infty$ denote
the set of infinitely differentiable functions. We denote the $i$th standard
basis vector by $e\ind{i}$, and let $I_d \in \R^{d\times d}$ denote the $d
\times d$ identity matrix; we drop the subscript $d$ when it is clear from
context.  For any set $\mc{S}$ and functions $g,h : \mc{S} \to
\openright{0}{\infty}$ we write $g \lesssim h$ or $g = O(h)$ if there exists
$c>0$ such that $g(s) \le c\cdot h(s)$ for every $s\in\mc{S}$. We
write $g = \Otil{h}$ if $g \lesssim h \log(h + 2)$.
\section{Preliminaries}
\label{sec:prelims}

We begin our development with definitions of the
classes of functions (\S~\ref{sec:prelims-funcs}), classes of algorithms (\S~\ref{sec:prelims-algs}), and notions of complexity  (\S~\ref{sec:prelims-complexity}) that
we study. 

\subsection{Function classes}\label{sec:prelims-funcs}

Measures of function regularity are crucial for the design and analysis of
optimization algorithms~\cite{Nesterov04,
	BoydVa04, NemirovskiYu83}.
We focus on two types of
regularity conditions: Lipschitzian properties of derivatives and bounds on
function value.

We first list a few equivalent definitions of Lipschitz
continuity.  A function $f:\R^d \to \R$ has $\Smp$-Lipschitz $p$th order
derivatives if it is $p$ times continuously differentiable, and for every $x
\in \R^d$ and direction $v \in \R^d, \norm{v} \le 1$, the directional
projection $f_{x,v}(t) \defeq f(x + t \cdot v)$ of $f$, defined for
$t \in \R$, satisfies
\begin{equation*}
  \left| f_{x,v}^{(p)}(t) - f_{x,v}^{(p)}(t') \right|
  \le \Smp \left|t - t'\right|
\end{equation*}
for all $t, t' \in \R$, where $f_{x,v}^{(p)}(\cdot)$ is the $p$th
derivative of $t \mapsto f_{x,v}(t)$. If $f$ is $p+1$ times
continuously differentiable, this is equivalent to requiring 
\begin{equation*}
  \left|f_{x,v}^{(p+1)}(0)\right| \le \Smp
  ~~~ \mbox{or} ~~~
  \opnorm{\deriv{p+1} f(x)} \le \Smp
\end{equation*}
for all $x, v \in \R^d$, $\norm{v}\le 1$. We occasionally refer to a
function with Lipschitz $p$th order derivatives as $p$th order smooth.

Complexity guarantees for finding stationary points of non-convex functions
$f$ typically depend on the function value bound $f(x\ind{0}) - \inf_x
f(x)$, where $x\ind{0}$ is a pre-specified point. Without loss of
generality, we take the pre-specified point to be $0$ for the remainder
of the paper. With that in mind, we define the following classes of
functions.

\begin{definition}
  \label{def:f-class}
  Let $p \ge 1$, $\DeltaF > 0$ and $\Sm{p} > 0$. Then the set
  \begin{equation*}
    \Fclass{p}
  \end{equation*}
  denotes the union, over $d \in \N$, of the collection of 
  $\mc{C}^\infty$ functions $f : \R^d \to \R$ with $\Smp$-Lipschitz $p$th
  derivative and $f(0)-\inf_x f(x) \le \DeltaF$. 
\end{definition}
\noindent
The function classes $\Fclass{p}$ include functions on $\R^d$ for all $d \in
\N$, following the established study of
 ``dimension free'' convergence 
 guarantees~\cite{NemirovskiYu83,Nesterov04}. As explained in 
 Section~\ref{sec:intro-importance-of-high-d}, 
 we construct explicit functions $f : \R^d \to \R$ that are
difficult to optimize, where the dimension $d$ is finite, but
our choice of $d$ grows inversely in the desired accuracy of the solution.

For our results, we also require the following important invariance notion,
proposed (in the context of optimization) by
\citet[Ch.~7.2]{NemirovskiYu83}.
\begin{definition}[Orthogonal invariance]
  A class of functions $\mc{F}$ is \emph{orthogonally invariant} if for
  every $f \in \mc{F}$, $f : \R^d \to \R$, and every matrix $U \in \R^{d'
    \times d}$ such that $U^{\top} U = I_d$, the function $f_U: \R^{d'} \to
  \R$ defined by $f_U(x) = f(U^{\top} x)$ belongs to $\mathcal{F}$.
\end{definition}
\noindent
Every function class we consider is orthogonally invariant, as $f(0)-\inf_x
f(x) = f_U(0) - \inf_x f_U(x)$ and $f_U$ has the same Lipschitz constants to
all orders as $f$, as their collections of associated directional
projections are identical.

\subsection{Algorithm classes}\label{sec:prelims-algs}

We also require careful definition of the classes of optimization
algorithms we consider.  For any dimension $d\in \N$, an \emph{algorithm}  
$\alg$ (also referred to as \emph{method}) maps
functions $f:\R^d\to\R$ to a sequence of \emph{iterates} in $\R^d$; that is,
$\alg$ is defined separately for every finite $d$.
We let
\begin{equation*}
  \alg[f] = \{x\ind{t}\}_{t=1}^\infty
\end{equation*}
denote the sequence $x\ind{t} \in \R^d$ of iterates that $\alg$ generates
when operating on $f$. 

To model the computational cost of an algorithm, we adopt the \emph{information-based complexity} framework,
which \citet{NemirovskiYu83} develop (see
also~\cite{TraubWaWa88,AgarwalBaRaWa12,BraunGuPo17}), and view every every iterate $x\ind{t}$ as a query to an \emph{information oracle}. Typically, one places restrictions on the information the oracle returns (\eg only the function value and gradient at the query point) and makes certain  assumptions on how the algorithm uses this information (\eg deterministically). Our approach is syntactically different but semantically identical: we build the oracle restriction, along with any other assumption, directly into the structure of the algorithm. To formalize this, we define
\begin{equation*}
  \deriv{(0,\ldots,p)} f(x) \defeq \{ f(x), \grad f(x), \hess f(x), \ldots,
  \deriv{p}f(x) \}
\end{equation*}
as shorthand for 
the response of a $p$th order oracle to a query at point $x$. When $p=\infty$ 
this corresponds to an oracle that reveals all derivatives at $x$.
Our algorithm classes follow.

\paragraph{Deterministic algorithms}
For any $p\ge 0$, a \emph{$p$th-order deterministic algorithm} $\alg$
operating on $f:\R^d\to\R$ is one producing iterates of the form
\begin{equation*}
  x\ind{i} = \alg\ind{i}
  \left(\deriv{(0,\ldots,p)}f(x\ind{1}), \ldots, \deriv{(0,\ldots,p)}f(x\ind{i-1})\right)
  ~~\text{for }i\in\N,
\end{equation*}
where $\alg\ind{i}$ is a measurable mapping to $\R^d$ (the dependence on
dimension $d$ is implicit). We denote the class
of $p$th-order deterministic algorithms by $\AlgDet\ind{p}$ and let
$\AlgDet \defeq \AlgDet\ind{\infty}$ 
denote the class of all deterministic algorithms based on derivative
information.

\newcommand{\PREG}[1]{{\mathsf{REG}_{#1}}}

As a concrete example, for any $p\ge1$ and $L>0$ consider the algorithm
$\PREG{p,L}\in\AlgDet\ind{p}$ that produces iterates by minimizing the sum
of a $p$th order Taylor expansion and an order $p+1$ proximal term:
\begin{equation}\label{eq:preg-def}
  x\ind{k + 1}
  \defeq \argmin_x
  \bigg\{f(x\ind{k}) + \sum_{q = 1}^p
  \<\deriv{q} f(x\ind{k}), x^{\otimes q}\>
  + \frac{L}{(p + 1)!} \norms{x - x\ind{k}}^{p+1} \bigg\}.
\end{equation}
For $p=1$, $\PREG{p,L}$ is gradient descent with step-size $1/L$, for $p=2$ 
it is cubic-regularized Newton's method~\cite{NesterovPo06}, and for 
general $p$ it is a simplified form of the scheme 
that~\citet{BirginGaMaSaTo17} propose. 

\paragraph{Randomized algorithms (and function-informed processes)}
A \emph{$p$th-order randomized algorithm} $\alg$ is a distribution on
$p$th-order deterministic algorithms. We can write any such algorithm as a
deterministic algorithm given access to a random uniform variable on $[0,
  1]$ (\ie infinitely many random bits).  Thus the algorithm operates on $f$
by drawing $\xi \sim \uniform[0, 1]$ (independently of $f$), then producing
iterates of the form
\begin{equation}
  \label{eq:prelims-randomized-alg}
  x\ind{i} = \alg\ind{i}\left(\xi, \deriv{(0,\ldots,p)} f(x\ind{1}), \ldots,
  \deriv{(0,\ldots,p)}f(x\ind{i-1})\right)
  ~~\text{for }i\in\N,
\end{equation}
where $\alg\ind{i}$ are measurable mappings into $\R^d$.  In this case,
$\alg[f]$ is a random sequence,  
and
we call a random process $\{x\ind{t}\}_{t\in\N}$ 
\emph{informed by $f$} if it has the same law as $\alg[f]$
for some randomized algorithm $\alg$.
We let $\AlgRand\ind{p}$ denote the class of $p$th-order
randomized algorithms and 
$\AlgRand \defeq \AlgRand\ind{\infty}$ 
denote the class of randomized algorithms that use derivative-based
information.

\paragraph{Zero-respecting sequences and algorithms}
While deterministic and randomized algorithms are the natural collections for
which we prove lower bounds, it is useful to define an additional
structurally restricted class. This class forms the backbone of our lower
bound strategy (Sec.~\ref{sec:anatomy}), as it is both `small' enough to 
uniformly underperform on a single function, and `large' enough to imply 
lower bounds on the natural algorithm classes. %

For $v \in\R^d$ we let $\support{v} \defeq \{ i \in [d] \mid v_i \neq 0 \}$
denote the support (non-zero indices) of $v$. We extend this to tensors as
follows. Let $T \in \R^{\tensordim{k}{d}}$ be an order $k$ tensor, and for
$i \in \{1,\ldots,d\}$ let $T_i \in \R^{\tensordim{k-1}{d}}$ be the
order $(k-1)$ tensor defined by $[T_i]_{j_1, \ldots, j_{k-1}} = T_{i, j_1,
  \ldots, j_{k-1}}$. With this notation, we define
\begin{equation*}
  \support{T} \defeq \{ i \in \{1 , \dots, d \} \mid T_i \neq 0 \}.
\end{equation*}
Then for $p \in \N$ and any $f:\R^d \to \R$, we say that the sequence
$x\ind{1}, x\ind{2}, \ldots$ is \emph{$p$th order zero-respecting
  with respect to $f$} if
\begin{equation}
  \label{eq:prelim-zr-def}
  \support{x\ind{t}} \subseteq \bigcup_{q \in [p]} \bigcup_{s < t}
  \support{\deriv{q}f(x\ind{s})}~~\mbox{for each } t \in \N.
\end{equation}
The definition~\eqref{eq:prelim-zr-def} says that $x\ind{t}_i = 0$ if all
partial derivatives involving the $i$th coordinate
of $f$ (up to the $p$th order) are zero.  For $p=1$, this
definition is equivalent to the requirement that for every $t$ and
$j\in[d]$, if $\grad_j f(x\ind{s}) = 0$ for $s < t$, then $x\ind{t}_j =
0$. The requirement~\eqref{eq:prelim-zr-def} implies that
$x\ind{1}=0$.

An algorithm $\alg \in \AlgRand$ is \emph{$p$th order zero-respecting} if
for any $f : \R^d\to \R$, the (potentially random) iterate sequence
$\alg[f]$ is $p$th order zero respecting w.r.t.\ $f$.  Informally, an
algorithm is zero-respecting if it never explores coordinates which appear
not to affect the function. When initialized at the origin, most common
first- and second-order optimization methods are zero-respecting, including
gradient descent (with and without Nesterov acceleration), conjugate
gradient~\cite{HagerZh06}, BFGS and
L-BFGS~\cite{LiuNo89,NocedalWr06},\footnote{if the initial Hessian
  approximation is a diagonal matrix, as is typical.} Newton's method (with
and without cubic regularization~\cite{NesterovPo06}) and trust-region
methods~\cite{ConnGoTo00}.  We denote the class of $p$th order
zero-respecting algorithms by $\AlgZR\ind{p}$, and let
$\AlgZR\defeq \AlgZR\ind{\infty}$.

In the literature on lower bounds for first-order convex optimization, it is
common to assume that methods only query points in the span of the gradients
they observe~\cite{Nesterov04,ArjevaniShSh16}.  Our notion of
zero-respecting algorithms generalizes this assumption to higher-order
methods, but even first-order zero-respecting algorithms are slightly more general. For example, coordinate
descent methods~\cite{Nesterov12} are zero-respecting, but they generally do
not remain in the span of the gradients.

\subsection{Complexity measures}\label{sec:prelims-complexity}

\newcommand{\law}{\mc{L}_x}
\newcommand{\PR}{\mathbf{P}}

With the definitions of function and algorithm class in hand, we turn to
formalizing our notion of complexity: what is the best performance an
algorithm in class $\mathcal{A}$ can achieve \emph{for all} functions in
class $\mc{F}$?  As we consider finding stationary points of $f$, the
natural performance measure is the number of iterations (oracle queries) 
required to find
a point $x$ such that $\norm{\nabla f(x)} \le \epsilon$. Thus for a
deterministic sequence  $\{x\ind{t}\}_{t\in\N}$ we define
\begin{equation*}
  \TimeEps{\{x\ind{t}\}_{t\in\N}}{f} \defeq
  \inf\left\{t\in\N \mid \normbig{\grad
    f(x\ind{t})} \le \epsilon \right\},
\end{equation*}
and refer to it as the \emph{complexity of $\{x\ind{t}\}_{t\in\N}$ on $f$}.
As we consider randomized algorithms as well, for a random process
$\{x\ind{t}\}_{t\in\N}$ with probability distribution $P$, meaning
for a set $A \subset (\R^d)^\N$ the probability that
$\{x\ind{t}\}_{t \in \N} \in A$ is $P(A)$, we define
\begin{equation}
  \label{eq:prelims-time-eps}
  \TimeEps{P}{f} \defeq
    \inf\left\{t\in\N \mid P\left(
    \normbig{\grad f(x\ind{s})} >  \epsilon
    ~ \mbox{for~all~} s \le t
    \right) \le \frac{1}{2}\right\}.
\end{equation}
The complexity $\TimeEps{P}{f}$ is also the median of the random variable  
$\TimeEps{\{x\ind{t}\}_{t\in\N}}{f}$ for $\{x\ind{t}\}_{t \in \N} \sim P$.
By Markov's inequality, definition~\eqref{eq:prelims-time-eps} provides a lower bound on
expectation-based alternatives, as
\begin{equation*}
  \inf\left\{t\in\N \mid \E_{P}\left[ \norms{\grad
      f(x\ind{t})}\right] \le \epsilon \right\} \ge \TimeEps[2\epsilon]{P}{f}
  ~~\mbox{and}~~
  \E_{P}\left[ \TimeEps{\{x\ind{t}\}_{t\in\N}}{f} \right]
  \ge \half \TimeEps{P}{f}.
\end{equation*}
(Here $\E_P$ denotes expectation taken according to the distribution $P$.)

To measure the performance
of algorithm $\alg$ on function $f$, we evaluate the iterates it produces
from $f$, and with mild abuse of notation, we define
\begin{equation*}
  \TimeEps{\alg}{f} \defeq
  \TimeEps{\alg[f]}{f}
\end{equation*}
as the complexity of $\alg$ on $f$.
With this setup, we define the \emph{complexity of algorithm class $\mathcal{A}$ on function class $\mathcal{F}$} as
\begin{equation}
  \label{eq:complexity-def}
  \CompEps{\mathcal{A}}{\mathcal{F}} \defeq 
  \inf_{\alg\in\mc{A}}\sup_{f\in\FclassBlank} \TimeEps{\alg}{f}.
\end{equation}

Many algorithms guarantee ``dimension independent''
convergence~\cite{Nesterov04} and thus provide upper bounds for the
quantity~\eqref{eq:complexity-def}. A careful tracing of constants in the
analysis of \citet{BirginGaMaSaTo17} implies that the generalized
regularization scheme $\PREG{p,L}$ defined by the
recursion~\eqref{eq:preg-def} guarantees
\begin{equation}\label{eq:pth-order-reg-ub}
  \CompEps{\AlgDet\ind{p}\cap\AlgZR\ind{p}}{\Fclass{p}}
  \le \sup_{f\in\Fclass{p}}\TimeEps{\PREG{p,\Smp}}{f}
   \lesssim \DeltaF
  \Sm{p}^{1/p} \epsilon^{-(1+p)/p}
\end{equation}
for all $p \in \N$.  In this paper we prove these rates are sharp to within
($p$-dependent) constant factors. 

While definition~\eqref{eq:complexity-def} is our primary notion of
complexity, our proofs provide bounds on smaller quantities
than~\eqref{eq:complexity-def} that also carry meaning. For zero-respecting
algorithms, we exhibit a single function $f$ and bound
$\inf_{\alg\in\AlgZR}\TimeEps{\alg}{f}$ from below, in effect interchanging
the $\inf$ and $\sup$ in~\eqref{eq:complexity-def}. This implies that all
zero-respecting algorithms share a common vulnerability. For randomized
algorithms, we exhibit a distribution $P$ supported on functions of
a fixed dimension $d$, and we lower bound the average $\inf_{\alg\in\AlgRand} \int
\TimeEps{\alg}{f} dP(f)$, bounding the
\emph{distributional complexity}~\cite{NemirovskiYu83,
  BraunGuPo17},
which is never greater than worst-case complexity (and is equal for
randomized and deterministic algorithms).
Even randomized algorithms
share a common vulnerability: functions drawn from $P$.

\section{Anatomy of a lower bound}
\label{sec:anatomy}

In this section we present a generic approach to proving lower bounds for 
optimization algorithms. The basic techniques we use are well-known and 
applied extensively in the literature on lower bounds for convex 
optimization~\cite{NemirovskiYu83,Nesterov04,WoodworthSr16,ArjevaniShSh17}.
 However, here we generalize and abstract away these techniques, showing 
how they apply to high-order methods, non-convex functions, and various 
optimization goals (\eg $\epsilon$-stationarity, $\epsilon$-optimality).

\subsection{Zero-chains}
\label{sec:anatomy-zero-chains}

Nesterov~\cite[Chapter 2.1.2]{Nesterov04} proves lower bounds for smooth
convex optimization problems using the ``chain-like'' quadratic function
\begin{equation}\label{eq:prelims-nesterov-chain}
  f(x) \defeq
  \half (x_1 - 1)^2
  + \half \sum_{i=1}^{d-1} (x_i - x_{i+1})^2,
\end{equation}
which he calls the ``worst function in the world.'' The important property
of $f$ is that for every $i\in[d]$, $\grad_i f(x)=0$ whenever
$x_{i-1}=x_{i}=x_{i+1}=0$ (with $x_0 \defeq 1$ and $x_{d+1} \defeq
0$). Thus, if we ``know'' only the first $t-1$ coordinates of $f$, \ie
are able to query only vectors $x$ such $x_{t}=x_{t+1}=\cdots=x_{d}=0$, then
any $x$ we query satisfies $\grad_s f(x) =0$ for $s>t$; we only ``discover''
a single new coordinate $t$. 
We generalize this chain
structure to higher-order derivatives as follows.

\begin{definition}\label{def:zero-chain}
  For $p\in\N$, a function $f : \R^d \rightarrow \R$ is a \emph{$p$th-order
    zero-chain} if for every $x \in \R^d$,
  \begin{equation*}
    \support{x} \subseteq 
    \{ 1, \dots, i - 1 \}
    ~~ \mbox{implies} ~~
    \bigcup_{q \in [p]}{\support{\deriv{q}f(x)}}  \subseteq  \{ 1, \dots, i \}.
  \end{equation*}
  We say $f$ is a \emph{zero-chain} if it is a $p$th-order zero-chain
  for every $p\in\N$.
\end{definition}
\noindent
In our terminology, Nesterov's function~\eqref{eq:prelims-nesterov-chain} is
a first-order zero-chain but not a second-order zero-chain, as
$\support{\deriv{2}f(0)}=[d]$. Informally, at a point for which
$x_{i-1}=x_{i}=\cdots=x_d=0$, a zero-chain appears constant in
$x_i, x_{i + 1}, \ldots, x_d$. Zero-chains structurally limit the rate
with which zero-respecting algorithms acquire information from
derivatives. We formalize this in the following observation,
whose proof is a straightforward induction; see 
Table~\ref{tab:zero-chain-illue} for an illustration.

\begin{observation}
  \label{obs:zero-chain}
  Let $f : \R^d \rightarrow \R$ be a $p$th order zero-chain and let
  $x\ind{1} = 0, x\ind{2}, \ldots$ be a pth order zero-respecting sequence with
  respect to $f$. Then $x\ind{t}_j=0$ for $j \ge t$ and all $t\le d$.
\end{observation}

\begin{proof}
	We show by induction on $k$ that $\support{x\ind{t}}\subseteq [t-1]$ for every $t\le k$; the case $k=d$ is the required result. The case $k=1$ holds since $x\ind{1}=0$. If the 
hypothesis holds for some $k<d$ then by Definition~\ref{def:zero-chain} we have $\cup_{q \in [p]}{\support{\deriv{q}f(x^{(t)})}}  \subseteq  \{ 1, \dots, t \}$ for every $t\le k$. Therefore, by 
the zero-respecting property~\eqref{eq:prelim-zr-def}, we have $ \support{x\ind{k+1}}\subseteq\cup_{q \in [p]} \cup_{t < k+1}  \support{\deriv{q}f(x\ind{t})} \subseteq [k]$, completing the 
induction.
\end{proof}

\newcommand{\nz}{*}
\definecolor{grayish}{rgb}{0.4,0.4,0.4} %
\newcommand{\z}{{\color{grayish}0}}

\mathprog{\vspace{-20pt}}

\begin{table}[h]
  \begin{center}
    \begin{tabular}{%
	>{\centering\arraybackslash}m{\mathprog{1.2cm}\arxiv{1.7cm}} %
	>{\centering\arraybackslash}m{\mathprog{1.5cm}\arxiv{2.0cm}} %
	>{\centering\arraybackslash}m{1.2cm} %
	>{\centering\arraybackslash}m{1cm} %
	>{\centering\arraybackslash}m{1cm} %
	>{\centering\arraybackslash}m{1cm} %
	>{\centering\arraybackslash}m{1cm} %
	>{\centering\arraybackslash}m{1cm} %
	>{\centering\arraybackslash}m{1cm} %
      }
      &&coordinate \\
      iteration & information & $j=1$ & $2$ & $3$ & $4$ & $\cdots$
      & $d-1$ & $d$ \\[3pt]
      \hline \vspace{3pt}
      \multirow{2}{*}{$t=0$} & $x\ind{0}$ & \z & \z & \z & \z & $\cdots$ &\z 
      &\z \\
      & $\grad f(x\ind{0})$  & $\nz$ & \z & \z & \z & $\cdots$ 
      &\z  
      &\z 
      \\[3pt]
      \hline \vspace{3pt}
      \multirow{2}{*}{$t=1$} & $x\ind{1}$ & $\nz$ & \z & \z & \z & $\cdots$ 
      &\z 
      &\z  \\
      & $\grad f(x\ind{1})$ & $\nz$ & $\nz$ & \z & \z & $\cdots$ 
      &\z  
      &\z 
      \\[3pt] 
      \hline \vspace{3pt}
      \multirow{2}{*}{$t=2$} & $x\ind{2}$ & $\nz$ & $\nz$ & \z & \z & $\cdots$ 
      &\z 
      &\z  \\
      & $\grad f(x\ind{2})$ & $\nz$ & $\nz$ & $\nz$ & \z & $\cdots$ 
      &\z  
      &\z 
      \\ \hline
      $\vdots$ & $\vdots$ & $\vdots$ & $\vdots$ & $\vdots$ & $\vdots$ &
      $\ddots$ & $\vdots$ & $\vdots$
      \\[3pt] 
      \hline \vspace{3pt}
      \multirow{2}{*}{$t=d-1$}
      & $x\ind{d-1}$              & $\nz$ & $\nz$ & $\nz$ & 
      $\nz$ & $\cdots$ 
      &$\nz$ 
      &\z  \\
      & $\grad f(x\ind{d-1})$ & $\nz$ & $\nz$ & $\nz$ & $\nz$ & 
      $\cdots$ 
      &$\nz$  
      &$\nz$
    \end{tabular}
  \end{center}
	\vspace{-8pt}
  \caption{Illustration of Observation~\ref{obs:zero-chain}: 
    a zero-respecting algorithm operating on a zero-chain. We indicate 
    the nonzero entries of the iterates and the gradients by $\nz$.} 
  \label{tab:zero-chain-illue}
\end{table}

\mathprog{\vspace{-28pt}}

\newcommand{\fbad}{f}
\newcommand{\fsup}{\tilde{f}} 
\subsection{A lower bound strategy}\label{sec:anatomy-strategy}

The preceding discussion shows that zero-respecting algorithms take many
iterations to ``discover'' all the coordinates of a zero-chain. In the 
following observation, we formalize how finding a suitable zero-chain 
provides a lower bound on the performance of zero-respecting algorithms.

\begin{observation}
  \label{obs:lower-bound-strategy}
Consider $\epsilon > 0$, a function class $\FclassBlank$, and $p,\T\in\N$.
If $\fbad : \R^{\T} \rightarrow \R$ satisfies
  \begin{enumerate}[i.]
\item  \label{item:fbad-zero-chain} $\fbad$ is a $p$th-order zero-chain,
\item \label{item:fbad-function-class}
  $\fbad$ belongs to the function class, \ie $\fbad\in\FclassBlank$, and 
\item \label{item:large-grad-prop}
  $\norm{\grad \fbad(x)} >\epsilon$ for every $x$ such that $x_{\T}=0$;\footnote{ We can readily adapt this property for lower bounds on other termination criteria, e.g. require $f(x)-\inf_y 
f(y) > \epsilon$ for every $x$ such that $x_\T=0$.}
\end{enumerate}
then $\CompEps{\AlgZR\ind p}{\FclassBlank} \ge \CompEps{\AlgZR\ind{p}}{\{f\}} > \T$.
\end{observation}

\begin{proof}
For $\alg\in\AlgZR\ind{p}$ and $\{x\ind{t}\}_{t\in\N}=\alg[\fbad]$ we have
by Observation~\ref{obs:zero-chain} that $x\ind{t}_T=0$ for all $t\le \T$ and
the large gradient property~\eqref{item:large-grad-prop} then implies $\norm{\grad
  \fbad(x\ind{t})} >\epsilon$ for all $t\le\T$. Therefore 
$\TimeEps{\alg}{\fbad} > \T$, and since this holds for any
$\alg\in\AlgZR\ind{p}$ we have
\begin{equation*}
\CompEps{\AlgZR\ind p}{\FclassBlank} = \inf_{\alg\in\AlgZR\ind{p}}\sup_{\fsup \in \FclassBlank}\TimeEps{\alg}{\fsup} \ge
\sup_{\fsup \in \FclassBlank} \inf_{\alg\in\AlgZR\ind{p}}\TimeEps{\alg}{\fsup}  \ge
 \inf_{\alg\in\AlgZR\ind{p}}\TimeEps{\alg}{\fbad} > T.
\end{equation*}
\end{proof}

If $f$ is a zero-chain, then so is the function $x \mapsto \mu f(x/\sigma)$
for any multiplier $\mu$ and scale parameter $\sigma$. This is useful for
our development, as we construct zero-chains $\{g_\T\}_{\T\in\N}$ such that $\norm{\grad g_T(x)}>c$ for every
$x$ with $x_\T=0$ and some constant $c > 0$. By setting
$\fbad(x) = \mu g_\T(x/\sigma)$, then choosing $\T$, $\mu$, and $\sigma$ to
satisfy conditions~\eqref{item:fbad-function-class}
and~\eqref{item:large-grad-prop}, we obtain a lower bound. As our choice of
$\T$ is also the final lower bound, it must grow to infinity as $\epsilon$
tends to zero. Thus, the hard functions we construct are fundamentally
high-dimensional, making this strategy suitable only for dimension-free
lower bounds.

\subsection{From deterministic to zero-respecting algorithms}
\label{sec:anatomy-reduction}

Zero-chains allow us to generate strong lower bounds for zero-respecting algorithms. The following reduction shows that these lower bounds are valid for deterministic algorithms as 
well.

\begin{restatable}{proposition}{propPrelimsDetZR}\label{prop:prelims-det-zr}
  Let $p\in \N\cup\{\infty\}$, $\mc{F}$ be an orthogonally invariant 
  function class and $\epsilon>0$. Then
  \begin{equation*}
    \CompEps{\AlgDet\ind{p}}{\mathcal{F}} 
    \ge
    \CompEps{\AlgZR\ind{p}}{\mathcal{F}}.
  \end{equation*}
\end{restatable}
\noindent
We also give a variant of Proposition~\ref{prop:prelims-det-zr} that is 
tailored to lower bounds constructed by means of 
Observation~\ref{obs:lower-bound-strategy} and allows explicit 
accounting of dimensionality. %
\begin{restatable}{proposition}{propPrelimsDetZRdim}
	\label{prop:prelims-det-zr-dim}
	Let $p\in \N\cup\{\infty\}$, $\FclassBlank$ be an orthogonally invariant 
	function class, $f\in\FclassBlank$ with domain of dimension $d$, and 
	$\epsilon>0$. If $\CompEps{\AlgZR\ind{p}}{\{f\}} \ge \T$, then
	\begin{equation*}
	\CompEps{\AlgDet\ind{p}}{\FclassBlank} 
	\ge
	\CompEps{\AlgDet\ind{p}}{\{f_U \mid U\in 
	\orthogonalgroup(d+\T,d)\}} 
	\ge
	\T,
	\end{equation*}
	where $f_U := f(U^\top z)$ and $\orthogonalgroup(d+\T,d)$ is the set of $(d+\T)\times d$ 
	orthogonal matrices, so that $\{f_U \mid U\in 
	\orthogonalgroup(d+\T,d)\}$ contains only function with domain of 
	dimension $d+T$.

\end{restatable}
\noindent
The proofs of Propositions~\ref{prop:prelims-det-zr} 
and~\ref{prop:prelims-det-zr-dim}, given in
Appendix~\ref{sec:prelims-proofs}, build on the classical notion of a
\emph{resisting oracle}~\cite{NemirovskiYu83,Nesterov04}, which we briefly
sketch here.  Let $\alg\in\AlgDet$, and let $f\in\mc{F}$, $f : \R^d \to \R$.
We adversarially select an orthogonal matrix $U \in
\R^{d'\times d}$ (for some finite $d'>d$)
 such that on the function $f_U \defeq f(U^{\top} z) \in \mathcal{F}$ the 
 algorithm $\alg$ behaves as if it was a zero-respecting algorithm.
In particular,  $U$ is sequentially constructed such that for the function
$f_U(z)$ the sequence $U^{\top} \alg[f_U]
\subset \R^d$ is zero-respecting with respect to $f$. Thus, there
exists an algorithm $\algzr \in \AlgDet \cap \AlgZR$ such that
$\algzr[f] = U^{\top} \alg[f_U]$, implying $\TimeEps{\alg}{f_U} =
\TimeEps{\algzr}{f}$. Therefore,
\begin{equation*}
\inf_{\alg \in \AlgDet}\sup_{f \in \mathcal{F}}{  \TimeEps{\alg}{f} } = 
\inf_{\alg \in \AlgDet}\sup_{f \in \mc{F}, U}{  \TimeEps{\alg}{f_U} } = 
\inf_{\alg \in \AlgDet}\sup_{f \in \mathcal{F}}{  \TimeEps{\algzr}{f} } \ge  
\inf_{\alg \in \AlgZR}\sup_{f \in \mathcal{F}}{  \TimeEps{\alg}{f} },
\end{equation*}
giving Proposition~\ref{prop:prelims-det-zr}; 
Proposition~\ref{prop:prelims-det-zr-dim}  follows similarly, and for it we 
may take $d'=d+\T$.

The adversarial rotation argument that yields
Propositions~\ref{prop:prelims-det-zr} and~\ref{prop:prelims-det-zr-dim} 
is more or less apparent in the proofs
of previous lower bounds in convex
optimization~\cite{NemirovskiYu83,WoodworthSr16,ArjevaniShSh17} for
deterministic algorithms. We believe it is 
instructive to separate the proof of lower bounds on
$\CompEps{\AlgZR}{\mathcal{F}}$ and the reduction from $\AlgDet$ to
$\AlgZR$, as the latter holds in great generality. Indeed,
Propositions~\ref{prop:prelims-det-zr} and~\ref{prop:prelims-det-zr-dim} 
hold for any complexity measure
$\TimeEps{\cdot}{\cdot}$ that satisfies
\begin{enumerate}
	\item Orthogonal invariance: for every $f:\R^d \to \R$, every $U\in \R^{d'
		\times d}$ such that $U^{\top} U = I_d$ and every sequence
	$\{z\ind{t}\}_{t \in \N} \subset \R^{d'}$, we have
	\begin{equation*}
	\TimeEps{\{z\ind{t}\}_{t \in \N}}{f(U^{\top} \cdot)}
	= \TimeEps{\{U^{\top} z\ind{t}\}_{t \in \N}}{f}.
	\end{equation*}
	\item ``Stopping time'' invariance: for any $T_0\in\N$, if
	$\TimeEps{\{x\ind{t}\}_{t \in \N}}{f} \le T_0$ then
	$\TimeEps{\{x\ind{t}\}_{t \in \N}}{f} = \TimeEps{\{\hat{x}\ind{t}\}_{t \in
			\N}}{f}$ for any sequence $\{\hat{x}\ind{t}\}_{t \in
		\N}$ such that $\hat{x}\ind{t} = x\ind{t}$ for $t \le T_0$. %
\end{enumerate}
These properties hold for the typical performance measures used in
optimization. Examples include time to $\epsilon$-optimality, in which case
$\TimeEps{\{x\ind{t}\}_{t \in \N}}{f} = \inf\{t\in\N \mid f(x\ind{t}) -
\inf_x f(x) \le \epsilon\}$, and the second-order stationarity desired in many
non-convex optimization 
problems~\cite{NesterovPo06,CarmonDuHiSi18,JinGeNeKaJo17}, 
where for
$\epsilon_1, \epsilon_2 > 0$ we define $\TimeEps{\{x\ind{t}\}_{t
		\in \N}}{f} = \inf\{t\in\N \mid \norms{ \grad f(x\ind{t})} \le
\epsilon_1 \text{ and } \hess f(x\ind{t}) \succeq -\epsilon_2 I \}$.

\subsection{Randomized algorithms}\label{sec:anatomy-randomized}

Proposition~\ref{prop:prelims-det-zr} and \ref{prop:prelims-det-zr-dim} do not apply to randomized
algorithms, as they require the adversary (maximizing choice of $f$) to
simulate the action of $\alg$ on $f$. To handle randomized algorithms, we strengthen the notion of a zero-chain as follows.
\begin{definition}\label{def:robust-zero-chain}
	A function $f : \R^d \rightarrow \R$ is a \emph{robust
		zero-chain} if for every $x \in \R^d$,
	\begin{equation*}
	|x_j| < 1/2, \ \forall j\ge i
	~~ \mbox{implies} ~~
	f(y) = f(y_1, \ldots, y_i, 0, \ldots, 0)
	~~\mbox{for all $y$ in a neighborhood of }x.
	\end{equation*}
\end{definition}
\noindent
A robust zero-chain is also an ``ordinary'' zero-chain. In 
Section~\ref{sec:fullder-random} we replace the adversarial rotation $U$ of  
\S~\ref{sec:anatomy-reduction} with an orthogonal matrix drawn uniformly at 
random, and consider the random function $f_U(x) = f(U^{\top}x)$, where $f$ 
is a robust zero-chain. 
We adapt a lemma by~\citet{WoodworthSr16}, and use it to show that for 
every 
$\alg\in\AlgRand$, $A[f_U]$ satisfies an 
approximate form of Observation~\ref{obs:zero-chain} (w.h.p.) whenever the 
iterates $\alg[f_U]$ have bounded norm. With further modification of $f_U$ to 
handle unbounded iterates, our zero-chain strategy yields a strong 
distributional complexity lower bound on $\AlgRand$. 

\section{Lower bounds for zero-respecting and deterministic algorithms}
\label{sec:fullder-deterministic}

For our first main results, we provide lower bounds on the complexity of all
deterministic algorithms for finding stationary points of smooth,
potentially non-convex functions. By 
Observation~\ref{obs:lower-bound-strategy}
and Proposition~\ref{prop:prelims-det-zr}, to prove a lower bound on
deterministic algorithms it is sufficient to construct a function that is
difficult for zero-respecting algorithms. For fixed $\T > 0$ , we define the 
(unscaled)
hard instance $\fhard : \R^d \to \R$ as
\begin{equation}
\label{eq:fullder-fhard-def}
\fhard(x) = -\compactfunc\left(1\right)
\gausscdf\left(x_{1}\right)+\sum_{i=2}^{\T}\left[
\compactfunc\left(-x_{i-1}\right)\gausscdf\left(-x_{i}\right)-
\compactfunc\left(x_{i-1}\right)\gausscdf\left(x_{i}\right)\right],
\end{equation}
where the component
functions are
\begin{equation*}
  \compactfunc(x)
  \defeq \begin{cases}
  0 & x \le 1/2\\
   \exp\left(1-\frac{1}{\left(2x-1\right)^{2}}\right)  & x>1/2
  \end{cases}
  ~~\mbox{and}~~
  \gausscdf(x)
  = \sqrt{e} \int_{-\infty}^{x} e^{-\half t^{2}}dt.
\end{equation*}

\begin{figure}
	\begin{minipage}[c]{0.5\textwidth}
		\centering
		\includegraphics[width=1\textwidth]{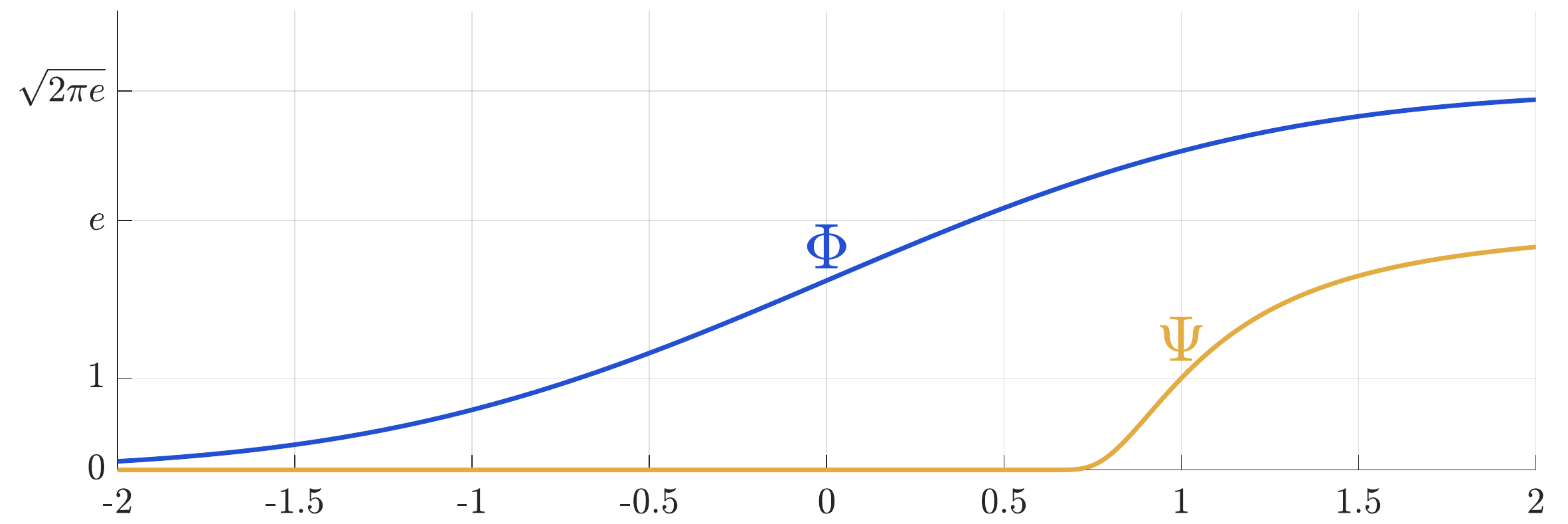} 
		
		\includegraphics[width=1\textwidth]{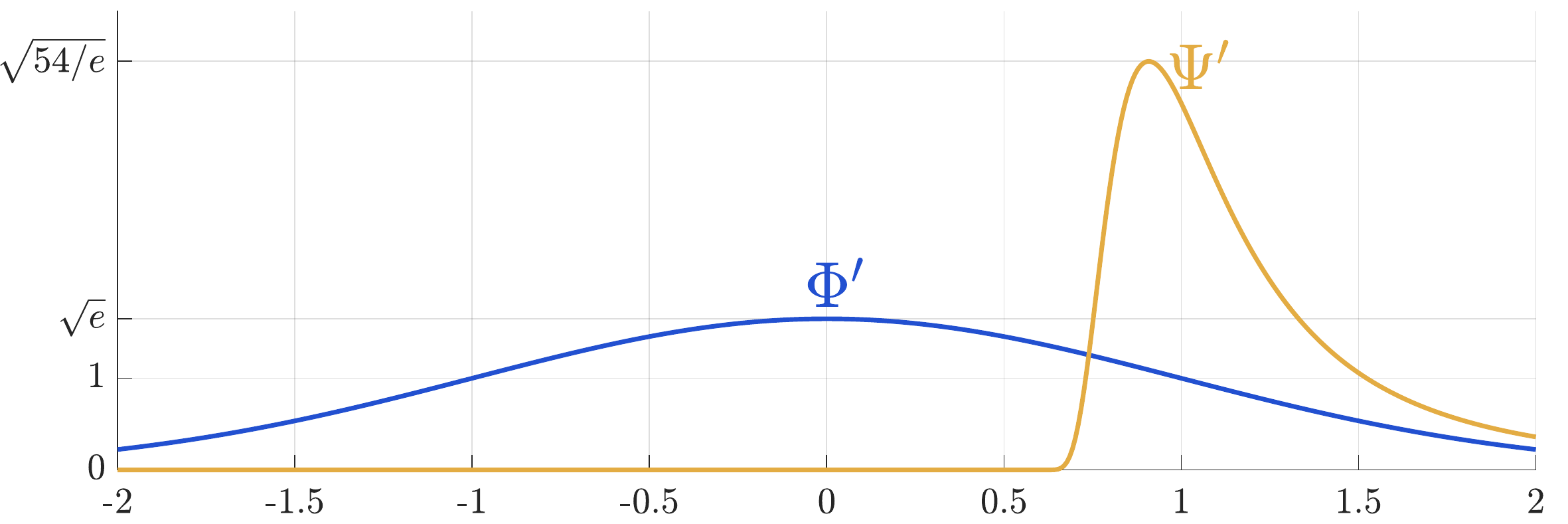}
		
	\end{minipage}
	\begin{minipage}[c]{0.5\textwidth}
		\centering
		\includegraphics[width=1\textwidth,  clip]{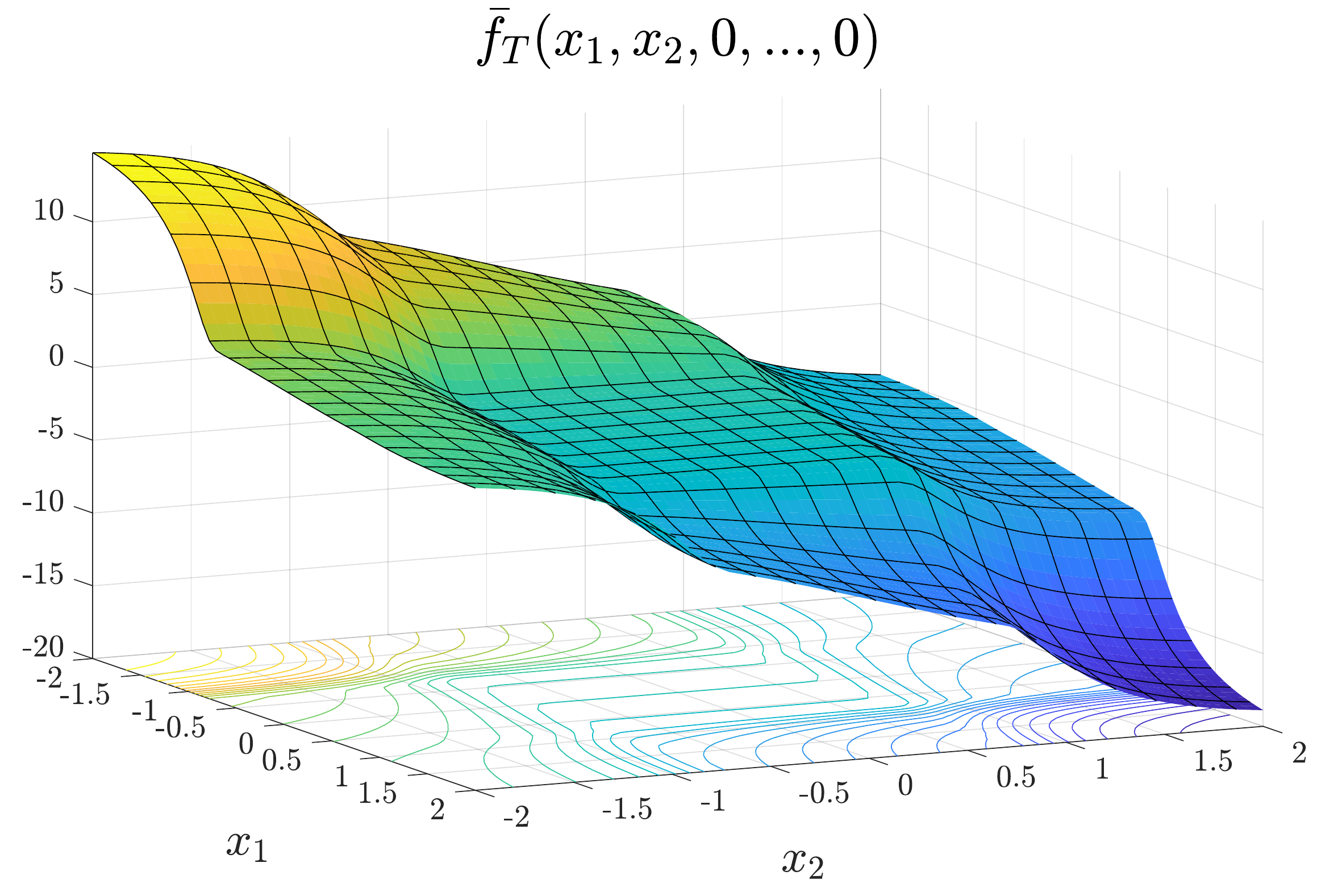} 
	\end{minipage}
	\vspace{6pt}
	\caption{Hard instance for full derivative information. Left: the functions  
		$\compactfunc$ and $\gausscdf$ (top) and their derivatives (bottom). 
		Right: Surface and contour plot of a two-dimensional cross-section of 
		the 
		hard instance $\fhard$. }\label{fig:fullder-construction}
\end{figure}

Our construction, illustrated in  Figure~\ref{fig:fullder-construction},  has 
two key properties. First is that $f$ is a zero-chain 
(Observation~\ref{obs:fhard-is-robust-zc} in the sequel). Second, as
we show in Lemma~\ref{lem:fullder-gradbound}, $\norms{\grad \fhard(x)}$ is
large unless $|x_i|\ge1$ for every $i\in[\T]$. These properties make it hard for
any zero-respecting method to find a stationary point of scaled versions
of $\fhard$, and coupled with
Proposition~\ref{prop:prelims-det-zr}, this gives a lower bound for
deterministic algorithms.

\subsection{Properties of the hard instance}

Before turning to the main theorem of this section, we catalogue 
the important properties of the functions $\compactfunc$,  
$\gausscdf$ and $\fhard$.
\begin{restatable}{lemma}{lemFullderProps}
	\label{lem:fullder-props}
  The functions $\compactfunc$ and $\gausscdf$ satisfy the following.
  \begin{enumerate}[i.]
  \item \label{item:fullder-psiphi-props-zero}
    For all $x \le \half$ and
    all $k \in \N$, $\compactfunc^{(k)}(x) = 0$.
  \item \label{item:fullder-psiphi-props-product}
    For all $x \ge 1$ and $|y| < 1$,
    $\compactfunc(x)\gausscdf'(y) > 1$.
  \item\label{item:fullder-psiphi-props-infinite}
    Both $\compactfunc$ and $\gausscdf$
    are infinitely differentiable,
    and for all $k \in \N$ we have
    \begin{equation*}
      \sup_x |\compactfunc^{(k)}(x)|
      \le \exp\left(\frac{5 k}{2}\log(4 k)\right)
      ~~\mbox{and}~~
      \sup_x |\gausscdf^{(k)}(x)|
      \le \exp\left(\frac{3k}{2} \log \frac{3k}{2} \right).
    \end{equation*}
  \item \label{item:fullder-psiphi-props-bounded} The functions and derivatives
    $\compactfunc, \compactfunc', \gausscdf$ and $\gausscdf'$ are
    non-negative and bounded, with
    \begin{equation*}
      0 \le \compactfunc < e,
      ~~ 0 \le \compactfunc' \le \sqrt{54/e},
      ~~ 0 < \gausscdf < \sqrt{2\pi e},
      ~~ \mbox{and} ~~
      0 < \gausscdf' \le \sqrt{e}.
    \end{equation*}
  \end{enumerate}
\end{restatable}
\noindent
We prove Lemma~\ref{lem:fullder-props} in
Appendix~\ref{sec:proof-fullder-props}. The remainder  
our development relies on $\compactfunc$ and $\gausscdf$ only 
through Lemma~\ref{lem:fullder-props}. Therefore, the precise choice of 
$\compactfunc, \gausscdf$ is not particularly special; any two functions 
with properties similar to Lemma~\ref{lem:fullder-props} will yield similar 
lower bounds.

The key consequence of
Lemma~\ref{lem:fullder-props}.\ref{item:fullder-psiphi-props-zero} is that
the function $f$ is a robust zero-chain (see Definition~\ref{def:robust-zero-chain}) and consequently also a zero-chain (Definition~\ref{def:zero-chain}):

\begin{observation}\label{obs:fhard-is-robust-zc}
  For any $j>1$, if $|x_{j-1}|,|x_j|<1/2$ then $\fhard(y) = 
  \fhard(y_1,\ldots,y_{j-1}, 0, y_{j+1}, \ldots, y_\T)$ for all $y$ in a 
  neighborhood of $x$.
\end{observation}

\noindent
Applying Observation~\ref{obs:fhard-is-robust-zc} for $j=i+1,\ldots,\T$ gives that $\fhard$ is a robust zero-chain by Definition~\ref{def:robust-zero-chain}. Taking derivatives of $\fhard(x_1,\ldots,x_i,0,\ldots,0)$ with respect to $x_j$, $j>i$, shows that $\fhard$ is also a zero-chain by Definition~\ref{def:zero-chain}. Thus, 
Observation~\ref{obs:zero-chain} shows that any
zero-respecting algorithm operating on $\fhard$ requires $T+1$ iterations to
find a point where $x_{T} \neq 0$.

Next, we establish the ``large gradient property'' that $\nabla
\fhard(x)$ must be large if any coordinate of $x$ is near zero.

\begin{lemma}
  \label{lem:fullder-gradbound}
  If $|x_i| < 1$ for any $i \le T$, then there exists $j \le i$ such that
  $|x_j| < 1$ and
  \begin{equation*}
    \norm{\grad \fhard(x)}
    \ge \left|\frac{\del }{\del x_j} \fhard(x)\right| > 1.
  \end{equation*}
\end{lemma}
\begin{proof}
  We take $j \le i$ to be the smallest $j$ for which $|x_j| < 1$, so that
  $|x_{j-1}| \ge 1$ (where we use the shorthand $x_0 \equiv 1$). Therefore,
 we have
  \begin{flalign}
    \frac{\del \fhard}{\del x_j}(x) & = -\compactfunc\left(-x_{j-1}\right)\gausscdf'\left(-x_{j}\right)
    -\compactfunc\left(x_{j-1}\right)\gausscdf'\left(x_{j}\right)
    -\compactfunc'\left(-x_{j}\right)\gausscdf\left(-x_{j+1}\right) 
    -\compactfunc'\left(x_{j}\right)\gausscdf\left(x_{j+1}\right)
    \label{eq:fullder-dfdj}  \nonumber \\ &
    \stackrel{(i)}{\le}
    -\compactfunc\left(-x_{j-1}\right)\gausscdf'\left(-x_{j}\right)
    -\compactfunc\left(x_{j-1}\right)\gausscdf'\left(x_{j}\right)
    \stackrel{(ii)}{=}
    -\compactfunc(|x_{j-1}|)\gausscdf'\left(x_{j}\sign(x_{j-1})\right)
    \stackrel{(iii)}{<} -1. \nonumber
  \end{flalign}
  In the chain of inequalities, inequality $(i)$ follows because
  $\compactfunc'(x)\gausscdf(y) \ge 0$ for every $x,y$; inequality $(ii)$
  follows because $\compactfunc(x) = 0$ for $x \le 1/2$, while equality
  $(iii)$ follows from
  Lemma~\ref{lem:fullder-props}.\ref{item:fullder-psiphi-props-product} and the
  pairing of $|x_j| < 1$ and $|x_{j-1}| \ge 1$.
\end{proof}

 Finally, we verify that $\fhard$ meets the smoothness and boundedness 
 requirements of the function classes we consider.

\begin{restatable}{lemma}{lemFullderBounded}
  \label{lem:fullder-bounded}
  The function $\fhard$ satisfies the following.
  \begin{enumerate}[i.]
  \item \label{item:fullder-funcbound}
    We have $\fhard(0) - \inf_x \fhard(x) \le 12\T$.
  \item \label{item:fullder-gradbound}
    For all $x \in \R^d$,
    $\norm{\grad \fhard (x)} \le 23\sqrt{\T}$.
  \item \label{item:fullder-lipschitz} For every $p \ge 1$,
    the $p$-th order derivatives of $\fhard$ are
    $\smC{p}$-Lipschitz continuous, where $\smC{p} \le
    \exp(\frac{5}{2} p \log p + c p)$ for a numerical constant
    $c<\infty$.
  \end{enumerate}
\end{restatable}
\noindent
The proof of Lemma~\ref{lem:fullder-bounded} is technical, so we
defer it to Appendix~\ref{sec:fullder-bounded-proof}. In the lemma,  
Properties~\ref{item:fullder-funcbound} and~\ref{item:fullder-lipschitz} allow 
us to guarantee that appropriately scaled versions of $\fhard$ are in 
$\Fclass{p}$. 
Property is~\ref{item:fullder-gradbound} is necessary for analysis of the 
randomized construction in Section~\ref{sec:fullder-random}.

\subsection{Lower bounds for zero-respecting and deterministic 
algorithms}\label{sec:fullder-deterministic-statement}

We can now state and prove a lower bound for finding stationary points of
$p$th order smooth functions using full derivative information and
zero-respecting algorithms (the class $\AlgZR$).
Proposition~\ref{prop:prelims-det-zr} transforms this bound into one on all
deterministic algorithms (the class $\AlgDet$).
\begin{theorem}\label{thm:fullder-simple}
  There exist numerical constants $0 < c_0, c_1 < \infty$
  such that the following lower bound holds. 
  Let $p \ge 1$, $p \in \N$, and let $\DeltaF$, $\Smp$, and
  $\epsilon$ be positive. Then
  \begin{equation*}
    \CompEps{\AlgDet}{\Fclass{p}} \ge \CompEps{\AlgZR}{\Fclass{p}}
    \ge c_0 \DeltaF \left(\frac{\Smp}{\smC{p}}\right)^{1/p}
    \epsilon^{-\frac{1+p}{p}}
  \end{equation*}
  where $\smC{p} \le e^{\frac{5}{2} p \log p + c_1 p}$. The lower bound 
  holds even if we restrict $\Fclass{p}$ to functions whose domain has
  dimension $1+{2c_0 \DeltaF (\Smp / \smC{p} )^{1/p}
  \epsilon^{-\frac{1+p}{p}}}$.
\end{theorem}

Before we prove the theorem, a few remarks are in order. First, our lower
bound matches the upper bound~\eqref{eq:pth-order-reg-ub} that
$p$th-order regularization schemes achieve~\cite{BirginGaMaSaTo17}, up to a
constant depending polynomially on $p$. Thus, although our lower bound
applies to algorithms given access to $\deriv{q} f(x)$ for all $q \in \N$,
only the first $p$ derivatives are necessary to achieve minimax optimal
scaling in $\DeltaF, \Sm{p}$, and $\epsilon$.

Second, inspection of the proof shows that we actually bound smaller
quantities than the complexity defined in
Eq.~\eqref{eq:complexity-def}. Indeed, we show that taking $T \gtrsim
\DeltaF (\Smp / \smC{p})^{1/p} \epsilon^{-\frac{1 + p}{p}}$ in the
construction~\eqref{eq:fullder-fhard-def} and appropriately scaling $\fhard$
yields a function $f : \R^\T \to \R$ that has $\Smp$-Lipschitz continuous
$p$th derivative, and for which \emph{any} zero-respecting algorithm
generates iterates such that $\norms{\grad f (x\ind{t})} 
> \epsilon$ for every $t \le \T$. That is,
\begin{equation*}
  \inf_{\alg \in \AlgZR} \TimeEps{\alg}{f}
  > T \gtrsim
  \DeltaF \Smp^{1/p} \epsilon^{-\frac{1 + p}{p}},
\end{equation*}
which is stronger than a lower bound on $\CompEps{\AlgZR}{\Fclass{p}}$.
Combined with the reduction in
Proposition~\ref{prop:prelims-det-zr-dim}, this implies that for any
deterministic algorithm $\alg\in\AlgDet$ there exists orthogonal
$U\in\R^{(2\T+1) \times \T}$ for which $f_U(x) = f(U^{\top} x)$ is difficult, 
\ie
$\TimeEps{\alg}{f(U^{\top} \cdot)} > \T$. 

Finally, the scaling of $\smC{p}$ with $p$ may appear strange, or perhaps
extraneous. We provide two viewpoints on this. First, one expects that the
smoothness constants $\Smp$ should grow quickly as $p$ grows; for
$\mc{C}^\infty$ functions such as $\phi(t) = e^{-t^2}$ or $\phi(t) =
\log(1+e^t)$, $\sup_t |\phi^{(p)}(t)|$ grows super-exponentially in
$p$. Indeed, $\smC{p}$ is the Lipschitz constant of the $p$th derivative of
$\fhard$. Second, the cases of main practical interest are $p \in \{1, 2\}$, where  
$\smC{p}^{1/p} \lesssim p^\frac{5}{2}$ can be considered a
numerical constant. This is because, for $p\ge3$, the only known methods 
with dimension-free rate of convergence $\epsilon^{-(p+1)/p}$~\cite{BirginGaMaSaTo17} 
require full access to third derivatives, which is generally impractical. 
Therefore, a realistic discussion of the complexity of finding stationary 
point with smoothness of order $p\ge3$ must include additional 
restrictions on the algorithm class.

\subsection{Proof of Theorem~\ref{thm:fullder-simple}}

To prove Theorem~\ref{thm:fullder-simple}, we set up the hard instance $f:
\R^\T \to \R$ for some integer $T$ by appropriately scaling $f$ defined
in Eq.~\eqref{eq:fullder-fhard-def},
\begin{equation*}
  f(x) \defeq \frac{\Smp \sigma^{p+1}}{\smC{p}} \fhard(x/\sigma)\,,
\end{equation*}
for some scale parameter $\sigma > 0$ to be determined, where $\smC{p} \le
e^{2.5 p \log p + c_{1}}$ is as in
Lemma~\ref{lem:fullder-bounded}.\ref{item:fullder-lipschitz}. We wish to 
show $f$ satisfies Observation~\ref{obs:lower-bound-strategy}. 
Observation~\ref{obs:fhard-is-robust-zc} implies 
Observation~\ref{obs:lower-bound-strategy}.\ref{item:fbad-zero-chain} 
 ($f$ is a zero-chain).
Therefore it remains to show parts~\ref{item:fbad-function-class} 
and~\ref{item:large-grad-prop} of 
Observation~\ref{obs:lower-bound-strategy}. Consider any $x \in \R^{\T}$ 
such that 
$x_{\T} = 0$. Applying Lemma~\ref{lem:fullder-gradbound} guarantees that 
$\norm{\grad
	\fhard(x/\sigma)} 
> 1$, and therefore
\begin{equation}
  \label{eq:fullder-scaled-gradbound}
  \norm{\grad f(x)}
  = \frac{\Smp \sigma^{p}}{\smC{p}}
  \norm{\grad \fhard(x/\sigma)} > \frac{\Smp \sigma^{p}}{\smC{p}}.
\end{equation}

It remains to choose $\T$ and $\sigma$ based on $\epsilon$ such that
$\norm{\grad f(x)} > \epsilon$ and $f \in \Fclass{p}$. By the lower
bound~\eqref{eq:fullder-scaled-gradbound}, the choice $\sigma =
(\smC{p}\epsilon/\Smp )^{1/p}$ guarantees $\norm{\grad f(x)} >
\epsilon$. We note that $\deriv{p+1}f(x)
= (\Smp/\smC{p})\deriv{p+1}f(x/\sigma)$ and therefore by
Lemma~\ref{lem:fullder-bounded}.\ref{item:fullder-lipschitz} we have that
the $p$-th order derivatives of $f$ are $\Smp$-Lipschitz continuous. Thus,
to ensure $f \in \Fclass{p}$ it suffices to show that $f(0) - \inf_x f(x)
\le \DeltaF$. By the first part of Lemma~\ref{lem:fullder-bounded} we have
\begin{equation*}
  f(0) - \inf_x f(x) = 
  \frac{\Smp \sigma^{p+1}}{\smC{p}}(\fhard(0) - \inf_x \fhard(x))
  \le \frac{12\Smp \sigma^{p+1}}{\smC{p}}T
  = \frac{12 \smC{p}^{1/p}\epsilon^{\frac{1+p}{p}}}{\Smp^{1/p}}T,
\end{equation*}
where in the last transition we substituted $\sigma =
(\smC{p}\epsilon/\Smp )^{1/p}$. We conclude that $f \in \Fclass{p}$ and
$T =
  \floor{\frac{\DeltaF\Smp^{1/p}}{12 \smC{p}^{1/p}}
    \epsilon^{-\frac{1+p}{p}}}$ so by 
    Lemma~\ref{obs:lower-bound-strategy},
$  \CompEps{\AlgZR}{\Fclass{p}} \ge \CompEps{\AlgZR}{\{f\}}
  \ge 1 + T
  \ge \frac{\DeltaF\Smp^{1/p}}{12 \smC{p}^{1/p} \epsilon^{\frac{1+p}{p}}},
$ with $\smC{p}$ bounded from above as in 
Lemma~\ref{lem:fullder-bounded}.\ref{item:fullder-lipschitz}. By  
Proposition~\ref{prop:prelims-det-zr-dim}, this bound transfers to 
$\CompEps{\AlgDet}{\Fclass{p}}$, where functions of dimension $2T+1$ 
suffice to establish it.

\section{Lower bounds for randomized algorithms}\label{sec:fullder-random}

With our lower bounds on the complexity of deterministic algorithms
established, we turn to the class of all randomized algorithms. We provide strong \emph{distributional complexity} lower bounds by exhibiting a distribution on functions such that a function drawn from it is ``difficult'' for \emph{any} randomized algorithm, with high probability. We do this via the composition
of a random orthogonal transformation with the function $\fhard$ defined
in~\eqref{eq:fullder-fhard-def}.  

The key steps in our deterministic bounds are (a) to show that any algorithm
can ``discover'' at most one coordinate per iteration and (b) finding an
approximate stationary point requires ``discovering'' $\T$ coordinates. In
the context of randomized algorithms, we must elaborate this development in
two ways.  First, in Section~\ref{sec:fullder-random-bounded} we provide a
``robust'' analogue of Observation~\ref{obs:zero-chain} (step (a) above): we
show that for a random orthogonal matrix $U$, any sequence of \emph{bounded}
iterates $\{x\ind{t}\}_{t\in\N}$ based on derivatives of $\fhard(U^{\top} \cdot)$
must (with high probability) satisfy that $\absinner{x\ind{t}}{ u\ind{j}}
\le \half$ for all $t$ and $j\ge t$, so that by
Lemma~\ref{lem:fullder-gradbound}, $\norm{\grad \fhard(U^{\top}x\ind{t})}$ 
must
be large (step (b)). Second, in Section~\ref{sec:fullder-random-unbounded}
we further augment our construction to force boundedness of
the iterates by composing $\fhard(U^{\top} \cdot)$ with a soft projection, so
that an algorithm cannot ``cheat'' with unbounded iterates. Finally, we
present our general lower bounds in Section~\ref{sec:fullder-random-lb}.

\subsection{Random rotations and bounded iterates}
\label{sec:fullder-random-bounded}

To transform our hard instance~\eqref{eq:fullder-fhard-def} into a hard
instance distribution, we introduce an orthogonal matrix $U\in\R^{d\times
  \T}$ (with columns $u\ind{1},\ldots,u\ind{\T}$), and define
\begin{equation}
  \label{eq:fullder-rand-rot}
  \fhardRot(x) \defeq \fhard(U^{\top} x)
  = \fhard(\inner{u\ind{1}}{x}, \ldots, \inner{u\ind{\T}}{x}),
\end{equation}
We assume throughout that $U$ is chosen uniformly at
random from the space of orthogonal matrices $\orthogonalgroup(d,\T) = \{V \in
\R^{d \times \T} \mid V^{\top} V = I_\T\}$; unless otherwise stated, the 
probabilistic 
statements we
give are respect to this uniform $U$ in addition to
any randomness in the algorithm that produces the iterates.
With this definition, we have the following extension of
Observation~\ref{obs:zero-chain} to randomized iterates, which we prove for 
$\fhard$ but is valid for any robust zero-chain 
(Definition~\ref{def:robust-zero-chain}). Recall that a
sequence is \emph{informed by $f$} if it has the same distribution as
$\alg[f]$ for some randomized algorithm $f$ (with
iteration~\eqref{eq:prelims-randomized-alg}).  
\begin{restatable}{lemma}{lemFullderRandSlow}
  \label{lem:fullder-rand-slow}
  Let $\delta > 0$ and $R \ge \sqrt{\T}$, and let $x\ind{1}, \ldots,
  x\ind{T}$ be informed by $\fhardRot$ and bounded, so that
  $\norms{x\ind{t}} \le R$ for each $T$.  If $d \ge 52\T R^2 \log
  \frac{2\T^2}{\delta}$ then with probability at least $1-\delta$,
  for all $t \le \T$ and each $j \in \{t, \ldots, \T\}$, we have
  \begin{equation*}
    |\inner{u\ind{j}}{x\ind{t}}| < 1/2.
  \end{equation*}
\end{restatable}

The result of Lemma~\ref{lem:fullder-rand-slow} is identical (to constant
factors) to an important result of \citet[Lemma~7]{WoodworthSr16}, but we
must be careful with the sequential conditioning of randomness between the
iterates $x\ind{t}$, the random orthogonal $U$, and how much information
the sequentially computed derivatives may leak. Because of this additional
care, we require a modification of their original proof,\footnote{ 
	In a recent note \citet{WoodworthSr17} independently provide a revision 
	of their proof that is similar, but not 
identical, to the 
one we propose here. 
} which we 
provide in Section~\ref{sec:fullder-rand-slow-proof}, giving a rough outline 
here. For
a fixed $t< \T$, assume that $|\inner{u\ind{j}}{x\ind{s}}| < 1/2$ holds for
every pair $s \le t$ and $j\in\{s,\ldots,\T\}$; we argue that this (roughly)
implies that $|\inner{u\ind{j}}{x\ind{t+1}}| < 1/2$ for every
$j\in\{t+1,\ldots,\T\}$ with high probability, completing the
induction. When the assumption that $|\inner{u\ind{j}}{x\ind{s}}| < 1/2$
holds, the robust zero-chain property of $\fhard$ (Definition~\ref{def:robust-zero-chain} and Observation~\ref{obs:fhard-is-robust-zc}) implies that for every $s\le t$ we have
\begin{equation*}
  \fhardRot(y) =
  \fhard(\inner{u\ind{1}}{y}, \ldots,
  \inner{u\ind{s}}{y}, 0, \ldots, 0)
\end{equation*}
for all $y$ in a neighborhood of $x\ind{s}$. That is, we can compute all the
derivatives of $\fhardRot$ at $x\ind{s}$ from $x\ind{s}$ and  $u\ind{1},\ldots,u\ind{s}$, as $\fhard$ is known. Therefore, given
$u\ind{1}, x\ind{1}, \ldots, u\ind{t}, x\ind{t}$ it is possible to
reconstruct all the information the algorithm has collected up to
iteration $t$. This means that beyond possibly revealing
$u\ind{1},\ldots,u\ind{t}$, these derivatives contain no additional
information on $u\ind{t+1},\ldots,u\ind{\T}$. Consequently, any component of
$x\ind{t+1}$ outside the span of
$u\ind{1},x\ind{1},\ldots,u\ind{t},x\ind{t}$ is a complete ``shot in the
dark.''

To give ``shot in the dark'' a more precise meaning, let $\hat{u}\ind{j}$ be
the projection of $u\ind{j}$ to the orthogonal complement of
$\mathrm{span}\{u\ind{1},x\ind{1},\ldots,u\ind{t},x\ind{t}\}$. We show that
conditioned on $u\ind{1},\ldots,u\ind{\T}$, and the induction hypothesis,
$\hat{u}\ind{j}$ has a rotationally symmetric distribution in that subspace,
and that it is independent of $x\ind{t+1}$. Therefore, by concentration of measure arguments on the sphere~\cite{Ball97}, we have
$\absinner{\hat{u}\ind{j}}{x\ind{t+1}} \lesssim \norms{x\ind{t + 1}} /
\sqrt{d} \le R / \sqrt{d}$ for any individual $j \ge t + 1$, with high
probability. Using an appropriate induction hypothesis, this is sufficient
to guarantee that for every $t+1 \le j \le \T$,
$|\inner{u\ind{j}}{x\ind{t+1}}| \lesssim R\sqrt{(\T\log\T)/d}$, which 
is bounded by $1/2$ for sufficiently large $d$.

\subsection{Handling unbounded iterates}\label{sec:fullder-random-unbounded}

In the deterministic case, the adversary (choosing the hard function $f$)
can choose the rotation matrix $U$ to be exactly orthogonal to all past
iterates; this is impossible for randomized algorithms.  The
construction~\eqref{eq:fullder-rand-rot} thus fails for unbounded random
iterates, since as long as $x\ind{t}$ and $u\ind{j}$ are not exactly orthogonal, their inner product will exceed $1/2$ for sufficiently large
$\norms{x\ind{t}}$, thus breaching the ``dead zone'' of $\compactfunc$ and
providing the algorithm with information on $u\ind{j}$. To prevent this, we
force the algorithm to only access $\fhardRot$ at points with bounded norm,
by first passing the iterates through a smooth mapping from $\R^d$ to a ball
around the origin. We denote our final hard instance construction by
$\fhardBound: \R^d \to \R$, and define it as
\begin{equation}\label{eq:fullder-fhardBound-def}
  \fhardBound(x) = \fhardRot(\rho(x)) + \frac{1}{10}\norm{x}^2,
  ~\mbox{where}~\rho(x) = \frac{x}{\sqrt{1+\norm{x}^2/R^2}}
  ~\mbox{and}~R=230\sqrt{T}\,.
\end{equation} 

The quadratic term in $\fhardBound$ guarantees that all points beyond a
certain norm have a large gradient, which prevents the algorithm from
trivially making the gradient small by increasing the norm of the
iterates. The following lemma captures the hardness of $\fhardBound$ for
randomized algorithms.

\begin{restatable}{lemma}{lemFullderRandHard}\label{lem:fullder-rand-hard}
  Let $\delta > 0$, and let $x\ind{1}, \ldots, x\ind{\T}$ be informed by
  $\fhardBound$. If $d \ge 52\cdot230^2\cdot \T^2 \log \frac{2\T^2}{\delta}$
  then, with probability at least $1-\delta$,
  \begin{equation*}
    \normbig{\grad \fhardBound
      (x\ind{t})} > 1/2
    ~~ \mbox{for all}~ t \le \T.
  \end{equation*}
\end{restatable}

\begin{proof}
  For $t\le \T$, set $y\ind{t} \defeq \rho(x\ind{t})$. For every $p\ge 0$
  and $t \in \N$, the quantity $\deriv{p} \fhardBound(x\ind{t})$ is
  measurable with respect $x\ind{t}$ and
  $\{\deriv{i}\fhardRot(y\ind{t})\}_{i=0}^p$ (the chain rule shows it can be
  computed from these variables without additional dependence on $U$, as
  $\rho$ is fixed). Therefore, the process $y\ind{1}, \ldots, y\ind{\T}$ is
  informed by $\fhardRot$ (recall defining
  iteration~\eqref{eq:prelims-randomized-alg}). Since $\norms{y\ind{t}} =
  \norms{\rho(x\ind{t})}\le R$ for every $t$, we may apply
  Lemma~\ref{lem:fullder-rand-slow} with $R=230\sqrt{T}$ to obtain that with
  probability at least $1-\delta$,
  \begin{equation*}
    \absinner{u\ind{\T}}{y\ind{t}} <
    1/2
    ~~ \mbox{for every } t \le \T.
  \end{equation*}
  Therefore, by Lemma~\ref{lem:fullder-gradbound} with $i = T$, for
  each $t$ there exists $j \le T$ such that
  \begin{equation}
    \label{eq:fullder-unbounded-uj}
    \left|\left\<u\ind{j}, y\ind{t}\right\>\right| < 1 ~\mbox{and}~ 
    \absinnerbig{u\ind{j}}{\grad \fhardRot( y\ind{t})} > 1.
  \end{equation}
  
  To show that $\norms{\grad \fhardBound( x\ind{t} )}$ is also large, we
  consider separately the cases $\norms{ x\ind{t} } \le R/2$ and
  $\norms{x\ind{t}} \ge R/2$. For the first case, we
  use $\frac{\del \rho}{\del x}(x) = \frac{I - \rho(x)
    \rho(x)^{\top}/R^2}{\sqrt{1+\norm{x}^2/R^2}}$ to write
  \begin{align*}
    \lefteqn{\innerbig{u\ind{j}}{\grad \fhardBound( x\ind{t})}  =
    \innerbig{u\ind{j}}{\frac{\del \rho}{\del x}(x\ind{t})
      \grad \fhardRot( y\ind{t})} + \frac{1}{5}\innerbig{u\ind{j}}{x\ind{t}}}
    \\ & \qquad  = 
    \frac{\inner{u\ind{j}}{\grad \fhardRot( y\ind{t})} -
      \inner{u\ind{j}}{y\ind{t}} \inner{y\ind{t}}{\grad \fhardRot( y\ind{t})}
      / R^2}{
      \sqrt{1+\norms{x\ind{t}}^2 / R^2}}
    +\frac{1}{5}\inner{u\ind{j}}{y\ind{t}}\sqrt{1+\norms{x\ind{t}}^2 / R^2}.
  \end{align*}
  Therefore, for $\norms{ y\ind{t}
  } \le \norms{ x\ind{t} } \le R/2$ we have
  \begin{equation*}
    \absinnerbig{u\ind{j}}{\grad \fhardBound( x\ind{t})} \ge 
    \frac{2}{\sqrt{5}} \absinnerbig{u\ind{j}}{\grad \fhardRot( y\ind{t})}
    - 
    \absinnerbig{u\ind{j}}{y\ind{t}}
    \left( \frac{\norms{\grad \fhardRot( y\ind{t})}}{2R} +
    \frac{1}{2\sqrt{5}} \right).
  \end{equation*}
  By Lemma~\ref{lem:fullder-bounded}.\ref{item:fullder-gradbound} we have
  $\norms{\grad \fhardRot( y\ind{t})} \le 23\sqrt{\T} = R/10$, which
  combined with~\eqref{eq:fullder-unbounded-uj} and the above display yields
  $\norms{\grad \fhardBound(x\ind{\T})} \ge \absinner{u\ind{j}}{\grad
    \fhardBound( x\ind{\T})} \ge \frac{2}{\sqrt{5}} - \frac{1}{20} -
  \frac{1}{2\sqrt{5}} > \frac{1}{2}$.
  
  In the second case, $\norm{x\ind{t}} \ge R/2$, we have for
  any $x$ satisfying $\norm{x} \ge R/2$ and $y = \rho(x)$ that
  \begin{equation}
    \label{eqn:when-x-big-so-is-fhard}
    \norm{\grad \fhardBound(x)} \ge
    \frac{1}{5} \norm{x}
    - \opnorm{\frac{\del \rho}{\del x}(x)}
    \norm{\grad \fhardRot(y)} \ge \frac{R}{10} -
    \frac{2}{\sqrt{5}}\frac{R}{10} > \sqrt{T} \ge 1,
  \end{equation}
  where we used $\opnorms{\frac{\del \rho}{\del x}(x)} \le
  \frac{1}{\sqrt{1+\norm{x}^2/R^2}}\le 2/\sqrt{5}$
  and that $\norms{\grad \fhardRot(y)} \le 23\sqrt{\T}=R/10$.
\end{proof}

As our lower bounds repose on appropriately scaling the function
$\fhardBound$, it remains to verify that $\fhardBound$ satisfies the few
boundedness properties we require. We do so in the following lemma.
\begin{restatable}{lemma}{lemFullderRandBoundedProps}
  \label{lem:fullder-rand-bounded-props}
  The function $\fhardBound$ satisfies the following.
  \begin{enumerate}[i.]
  \item \label{item:random-bounds-f-gap}
    We have $\fhardBound(0) - \inf_x \fhardBound(x) \le 12\T$.
  \item \label{item:random-bounds-lipschitz} For every $p \ge 1$, the $p$th
    order derivatives of $\fhardBound$ are $\smChat{p}$-Lipschitz
    continuous, where $\smChat{p} \le
    \exp(c p \log p + c )$ for a numerical constant $c<\infty$.
  \end{enumerate}
\end{restatable}
\noindent
We defer the (computationally involved) proof of this lemma to
Section~\ref{sec:fullder-rand-bounded-props-proof}.

\subsection{Final lower bounds}\label{sec:fullder-random-lb}

With Lemmas~\ref{lem:fullder-rand-hard}
and~\ref{lem:fullder-rand-bounded-props} in hand, we can state our lower
bound for all algorithms, randomized or otherwise, given access to all
derivatives of a $\mc{C}^\infty$ function.  Note that our construction also
implies an identical lower bound for (slightly) more general algorithms that
use any \emph{local oracle}~\cite{NemirovskiYu83,BraunGuPo17}, meaning 
that the information the oracle
returns about
a function $f$ when queried at a point $x$ is identical to that it returns
when a function $g$ is queried at $x$ whenever $f(z) = g(z)$ for all $z$ in a 
neighborhood of  $x$.
\begin{restatable}{theorem}{thmFullderFinal}\label{thm:fullder-final}
  There exist numerical constants $0 < c_0, c_1 < \infty$
  such that the following lower bound holds. Let $p \ge 1, p \in \N$,
  and let $\DeltaF$, $\Smp$, and $\epsilon$ be positive.
  Then
  \begin{equation*}
    \CompEps{\AlgRand}{\Fclass{p}}
    \ge c_0 \cdot \DeltaF \left(\frac{\Smp}{\smChat{p}}\right)^{1/p}
    \epsilon^{-\frac{1+p}{p}},
  \end{equation*}
  where $\smChat{p} \le e^{c_1 p \log p + c_1}$.
  The lower bound 
  holds even if we restrict $\Fclass{p}$ to functions where the domain has 
  dimension $1 + c_2 q\Big(\DeltaF \left({\Smp}/{\smC{p}}\right)^{1/p} 
  	\epsilon^{-\frac{1+p}{p}}\Big)$ with $c_2$ a numerical constant 
  	and $q(x) = x^2 \log(2x)$.
\end{restatable}

We return to the proof of Theorem~\ref{thm:fullder-final} in
Sec.~\ref{sec:fullder-final-proof}, following the same outline as that
of Theorem~\ref{thm:fullder-simple}, and provide some commentary here.  An
inspection of the proof to come shows that we actually demonstrate a
stronger result than that claimed in the theorem. For any $\delta\in(0,1)$
let $d \ge \ceil{52 \cdot (230)^2 \cdot \T^2 \log(2 \T^2/\delta)}$ where $\T =
\lfloor{c_0 \DeltaF ({\Smp}/{\smChat{p}})^{1/p}
  \epsilon^{-\frac{1+p}{p}}}\rfloor$ as in the claimed lower bound. In the
proof we construct a probability measure $\mu$ on functions in $\Fclass{p}$, of
fixed dimension $d$, such that
\begin{equation}
  \label{eq:fullder-rand-prob-lower-bound}
  \inf_{\alg \in \AlgRand}
  \int \P_\alg\left(
  \normbig{\nabla f(x\ind{t})} > \epsilon
  ~ \mbox{for~all~} t \le \T \mid f \right) d\mu(f)
  > 1-\delta,
\end{equation}
where the randomness in $\P_\alg$ depends only on $\alg$.  Therefore, by
definition~\eqref{eq:prelims-time-eps}, for \emph{any} $\alg\in\AlgRand$ a
function $f$ drawn from $\mu$ satisfies
\begin{equation}
  \label{eq:fullder-rand-T-lower-bound}
  \TimeEps{\alg}{f} > \T
  ~\mbox{with probability greater than }1-2\delta,
\end{equation} 
implying
Theorem~\ref{thm:fullder-final} for any $\delta \ge 1/2$. Thus, we exhibit a
randomized procedure for finding hard instances for any randomized
algorithm that requires no knowledge of the algorithm itself.

Theorem~\ref{thm:fullder-final} is stronger than
Theorem~\ref{thm:fullder-simple} in that it applies to the broad class of
all randomized algorithms. Our probabilistic analysis requires that the  
functions
constructed to prove Theorem~\ref{thm:fullder-final} have dimension scaling proportional to $\T^2 \log(\T)$ where $\T$ is the lower bound on the number of iterations. Contrast this to Theorem~\ref{thm:fullder-simple}, which
only requires dimension $2\T + 1$. A similar gap exists in complexity results for convex 
optimization~\cite{WoodworthSr16,WoodworthSr17}. At present, it unclear if these gaps are fundamental or a consequence of our specific constructions.

\subsection{Proof of Theorem~\ref{thm:fullder-final}}
\label{sec:fullder-final-proof}

We set up our hard instance distribution $f_U: \R^d \to \R$, indexed by a
uniformly distributed orthogonal matrix $U\in\orthogonalgroup(d,\T)$, by
appropriately scaling $\fhardBound$ defined
in~\eqref{eq:fullder-fhardBound-def},
\begin{equation*}
  f_U(x) \defeq \frac{\Smp \sigma^{p+1}}{\smChat{p}} \fhardBound(x/\sigma),
\end{equation*}
where the integer $\T$ and scale parameter $\sigma > 0$ are to be
determined, $d=\lceil{52\cdot (230)^2 T^2 \log(4T^2)}\rceil$, and the
quantity $\smChat{p} \le \exp(c_1 p \log p + c_1)$ for a numerical constant $c_1$
is defined in
Lemma~\ref{lem:fullder-rand-bounded-props}.\ref{item:random-bounds-lipschitz}.

Fix $\alg\in\AlgRand$ and let $x\ind{1},x\ind{2},\ldots,x\ind{T}$ be the
iterates produced by $\alg$ applied on $f_U$. Since $f$ and $\fhardBound$
differ only by scaling, the iterates
$x\ind{1}/\sigma,x\ind{2}/\sigma,\ldots,x\ind{T}/\sigma$ are informed by
$\fhardBound$ (recall Sec.~\ref{sec:prelims-algs}), and therefore we may
apply Lemma~\ref{lem:fullder-rand-hard} with $\delta = 1/2$ and our large
enough choice of dimension $d$ to conclude that
\begin{equation*}
  \P_{\alg,U}\left( \normbig{\grad \fhardBound\left( x\ind{t}/\sigma \right)}
  > \frac{1}{2}
  ~ \mbox{for~all}~ t \le \T\right) > \frac{1}{2},
\end{equation*}
where the probability is taken over both the random
orthogonal $U$ and any randomness in $\alg$.
As $\alg$ is arbitrary, 
taking $\sigma = (2\smChat{p}\epsilon/\Smp )^{1/p}$, this inequality becomes
the desired strong inequality~\eqref{eq:fullder-rand-prob-lower-bound} with $\delta=1/2$ and $\mu$ induced by the distribution of $U$. Thus, by~\eqref{eq:fullder-rand-T-lower-bound}, for every $\alg\in\AlgRand$ there exists $U_\alg\in\orthogonalgroup(d,\T)$ such that $\TimeEps{\alg}{f_{U_\alg}} \ge 1+T$, so
\begin{equation*}
	\inf_{\alg\in\AlgDet}\sup_{U\in\orthogonalgroup(d,\T)}\TimeEps{\alg}{f_U} \ge 1+T.
\end{equation*}

It remains to choose $T$ to guarantee that $f_U$ belongs to the relevant
function class (bounded and smooth) for every orthogonal $U$. By
Lemma~\ref{lem:fullder-rand-bounded-props}.\ref{item:random-bounds-lipschitz},
$f_U$ has $\Smp$-Lipschitz continuous $p$th order derivatives. By
Lemma~\ref{lem:fullder-rand-bounded-props}.\ref{item:random-bounds-f-gap},
we have
\begin{equation*}
  f_U(0) - \inf_x f_U(x)
  \le \frac{\Smp \sigma^{p+1}}{\smChat{p}}
  \left(\fhard(0) - \inf_x \fhard(x)\right)
  \le \frac{12\Smp \sigma^{p+1}}{\smChat{p}} T =
  \frac{24 (2\smChat{p})^{1/p} \epsilon^{\frac{p+1}{p}}}{\Smp^{1/p}} T,
\end{equation*}
where in the last transition we have substituted $\sigma =
(2\smC{p}\epsilon/\Smp)^{1/p}$.
Setting
$\T = \lfloor{\frac{\DeltaF}{48} ({\Smp}/{\smChat{p}})^{1/p} \epsilon^{-\frac{1+p}{p}}}\rfloor$
gives $f_U(0) - \inf_x f_U(x) \le \Delta$, and
$f_U \in \Fclass{p}$, yielding the theorem.
\section{Distance-based lower bounds}\label{sec:fullder-distance}

We have so far considered finding approximate stationary points of smooth
functions with bounded sub-optimality at the origin, \ie $f(0)-\inf_x f(x)
\le \DeltaF$. In convex optimization, it is common to consider instead
functions with bounded distance between the origin and a global minimum. 
We may consider a similar restriction for non-convex functions; for $p\ge 1$ 
and positive $\Smp,
D$, let
\begin{equation*}
\FclassD{p}
\end{equation*}
be the class of $\mc{C}^\infty$ functions with $\Smp$-Lipschitz $p$th order
derivatives satisfying 
\begin{equation}
  \label{eq:fullder-distance-def}
  \sup_x \left\{\norm{x} \mid x \in \argmin f \right\} \le D,
\end{equation}
that is, all global minima have bounded distance to the origin.

In this section we give a lower bound on the complexity of this function class 
that
has the same $\epsilon$ dependence as our bound for the class
$\Fclass{p}$. This is in sharp contrast to convex optimization, where
distance-bounded functions enjoy significantly better $\epsilon$ dependence
than their value-bounded counterparts (see
Section~\ref*{sec:convex} in the companion~\cite{NclbPartII}). 
Qualitatively, the reason for this
difference is that the lack of convexity allows us to ``hide'' global minima
close to the origin that are difficult to find for any algorithm with local
function access~\cite{NemirovskiYu83}.

We postpone the construction and proof to
Appendix~\ref{sec:proof-general-distance}, and move directly to the final bound.

\begin{restatable}{theorem}{thmFullderDistLB}\label{thm:fullder-final-dist}
  There exist numerical constants $0 < c_0, c_1 < \infty$ such that the
  following lower bound holds. For any $p\ge 1$, let $D, \Smp$, and
  $\epsilon$ be positive. Then
  \begin{equation*}
    \CompEps{\AlgRand}{\FclassD{p}} \ge 
    c_0 \cdot D^{1+p} \left( \frac{\Smp}{\smC{p}'}\right)^{\frac{1+p}{p}}
    \epsilon^{-\frac{1+p}{p}},
  \end{equation*}
  where $\smC{p}' \le e^{c_1 p \log p + c_1}$. 
  The lower bound 
  holds even if we restrict $\FclassD{p}$ to functions with domain of 
  dimension $1 + c_2 q\Big(D^{1+p} \left( 
  {\Smp}/{\smC{p}'}\right)^{\frac{1+p}{p}}
  \epsilon^{-\frac{1+p}{p}}\Big)$, for a some numerical constant 
  $c_2<\infty$ 
  and $q(x) = x^2 \log(2x)$.
\end{restatable}
\noindent
We remark that a lower-dimensional construction suffices for proving the 
lower bound for deterministic algorithm, similarly to 
Theorem~\ref{thm:fullder-simple}.

While we do not have a matching upper bound for 
Theorem~\ref{thm:fullder-final-dist}, we can match its $\epsilon$ dependence 
in the smaller function class
\begin{equation*}
\FclassDTwo{1}{p} = \FclassD{1} \cap \FclassD{p},
\end{equation*} 
due to the fact that for any $f:\R^d\to\R$ with $\SmGrad$-Lipschitz 
continuous gradient and global minimizer $x\opt$, we have $f(x)-f(x\opt)\le 
\half\SmGrad\norm{x-x\opt}^2$ for all $x\in\R^d$~\cite[cf.][Eq.\  
(9.13)]{BoydVa04}. Hence $\FclassDTwo{1}{p} \subset \Fclass{p}$, with 
$\DeltaF\defeq \half\SmGrad D^2$, and consequently by the 
bound~\eqref{eq:pth-order-reg-ub} we have
\begin{equation*}
\CompEps{\AlgDet\ind{p}\cap\AlgZR\ind{p}}{\FclassDTwo{1}{p}} \lesssim
 D^2 \SmGrad  \Smp^{1/p} \epsilon^{-\frac{p+1}{p}}.
\end{equation*}
\section{Conclusion}

This work provides the first algorithm independent and tight lower bounds 
on the dimension-free complexity of finding stationary points. As a 
consequence, we have characterized the optimal rates of convergence to 
$\epsilon$-stationarity, under the assumption of high dimension and an 
oracle that provides all derivatives. Yet, given the importance of 
high-dimensional problems, the picture is incomplete: high-order 
algorithms---even second-order method---are often impractical in large 
scale settings. We address this in the companion~\cite{NclbPartII}, which 
provides sharper lower bounds for the more restricted class of first-order 
methods. In~\cite{NclbPartII} we also provide a full conclusion for this 
paper sequence, discussing in depth the implications and questions that 
arise from our results.

\arxiv{
\section*{Acknowledgments}
OH was supported by the PACCAR INC fellowship. YC and JCD were partially
supported by the SAIL-Toyota Center for AI Research, NSF-CAREER award
1553086, and a Sloan Foundation Fellowship in Mathematics. YC was partially
supported by the Stanford Graduate Fellowship and the Numerical Technologies
Fellowship.
}

\arxiv{\bibliographystyle{abbrvnat}}
\mathprog{\bibliographystyle{abbrvnat}} 

\newpage
\appendix
\section{Proof of Propositions~\ref{prop:prelims-det-zr} 
and~\ref{prop:prelims-det-zr-dim}}\label{sec:prelims-proofs}

The core of the proofs of Propositions~\ref{prop:prelims-det-zr} 
and~\ref{prop:prelims-det-zr-dim} is the following construction.

\begin{lemma}\label{lem:prelims-app-det-zr-core}
	Let $p\in\N\cup\{\infty\}$, $T_0\in\N$ and $\alg\in\AlgDet\ind{p}$. 
	There exists an 
	algorithm $\algzr\in\AlgZR\ind{p}$ with the following property. For 
	every 
	$f:\R^d\to\R$ there exists an orthogonal matrix 
	$U\in\R^{(d+T_0)\times d}$ such that, for every $\epsilon > 0$,
	\begin{equation*}
	  \TimeEps{\alg}{f_U} > T_0
	  ~~\mbox{or}~~
	  \TimeEps{\alg}{f_U} = \TimeEps{\algzr}{f},
	\end{equation*}
	 where $f_U(x) \defeq f(U^\top x)$. 
\end{lemma}

\begin{proof}
	We explicitly construct $\algzr$ with the following slightly stronger 
	property. For every every $f:\R^d\to\R$ in $\FclassBlank$, there exists an
	orthogonal $U\in\R^{(d+T_0)\times d}$, $U^{\top} U = I_d$, such that
	$f_U(x) \defeq f(U^{\top} x)$ satisfies that the first $T_0$ iterates in
	sequences $\algzr[f]$ and $U^{\top} \alg[f_U]$ are identical.  (Recall the
	notation $\alg[f] = \{a\ind{t}\}_{t \in \N}$ where $a\ind{t}$ are the
	iterates of $\alg$ on $f$, and we use the obvious shorthand $U^{\top}
	\{a\ind{t}\}_{t\in\N} = \{U^{\top} a\ind{t}\}_{t\in\N}$.)  
	
	Before explaining the 
	construction
	of $\algzr$, let us see how its defining property implies the lemma. If 
	$\TimeEps{\alg}{f_U} > T_0$, we are done. Otherwise, 
	$\TimeEps{\alg}{f_U} \le T_0$ and we have
	\begin{equation}\label{eq:prelims-timeps-AB}
	\TimeEps{\alg}{f_U}
	\defeq \TimeEps{\alg[f_U]}{f_U}
	\stackrel{(i)}{=} \TimeEps{U^{\top} \alg[f_U]}{f}
	\stackrel{(ii)}{=} \TimeEps{\algzr}{f},
	\end{equation}
	as required. 
	The equality $(i)$ follows because $\norm{U g} = \norm{g}$ for all
	orthogonal $U$, so for any sequence $\{a\ind{t}\}_{t\in\N}$
	\begin{flalign*}
	\TimeEps{\{a\ind{t}\}_{t \in \N}}{f_U} & = 
	\inf\left\{ t\in\N \mid \norms{\grad f_U(a\ind{t})} \le \epsilon \right\} 
	\\ & =
	\inf\left\{ t\in\N \mid \norms{ \grad f(U^{\top} a\ind{t})} \le \epsilon
	\right\} =
	\TimeEps{\{U^{\top} a\ind{t}\}_{t \in \N}}{f}
	\end{flalign*}
	and in equality~$(i)$ we let $\{a\ind{t}\}_{t \in \N} = \alg[f_U]$. The
	equality $(ii)$ holds because $\TimeEps{\cdot}{\cdot}$ is a ``stopping
	time'': if $\TimeEps{U^{\top}\alg[f_U]}{f} \le T_0$ then the first $T_0$
	iterates of $U^{\top}\alg[f_U]$ determine $\TimeEps{U^{\top} 
	\alg[f_U]}{f}$, 
	and
	these $T_0$ iterates are identical to the first $T_0$ iterates of
	$\algzr[f]$ by assumption. 
	
	  It remains to construct the zero-respecting algorithm $\algzr$ with
	iterates matching those of $\alg$ under appropriate rotation.  We do this
	by describing its operation inductively on any given $f:\R^d\to \R$, 
	which
	we denote $\{z\ind{t}\}_{t\in\N} = \algzr[f]$. Letting $d' = d+T_0$, the
	state of the algorithm $\algzr$ at iteration $t$ is determined by a
	\emph{support} $S_t \subseteq [d]$ and orthonormal vectors 
	$\{u\ind{i}\}_{i
		\in S_t} \subset \R^{d'}$ identified with this support.  The support
	condition~\eqref{eq:prelim-zr-def} defines the set $S_t$,
	\begin{equation*}
	S_t = \bigcup_{q \in [p]} \bigcup_{s < t}
	\support{\deriv{q}f(z\ind{s})},
	\end{equation*}
	so that $\emptyset = S_1 \subseteq S_2 \subseteq \cdots$ and the 
	collection
	$\{u\ind{i}\}_{i \in S_t}$ grows with $t$.  We let $U \in \R^{d'\times d}$
	be the orthogonal matrix whose $i$th column is $u\ind{i}$---even 
	though $U$
	may not be completely determined throughout the runtime of $\algzr$, 
	our
	partial knowledge of it will suffice to simulate the operation of $\alg$ on
	$f_U(a) = f(U^{\top} a)$. Letting $\{a\ind{t}\}_{t \in \N} = \alg[f_U]$,
	our requirements $\algzr[f] = U^{\top} \alg[f_U]$ and $\algzr\in \AlgZR$ 
	are
	equivalent to
	\begin{equation}\label{eq:prelims-induciton}
	z\ind{t} = U^{\top} a\ind{t}\mbox{ and }\supports{z\ind{t}} \subseteq S_t
	\end{equation}
	for every $t \le T_0$ (we set $z\ind{i}=0$ for every $i>T_0$ without loss 
	of
	generality). 
	
	Let us proceed with the inductive argument. The iterate $a\ind{1} \in
	\R^{d'}$ is an arbitrary (but deterministic) vector in $\R^{d'}$. We thus
	satisfy~\eqref{eq:prelims-induciton} at $t=1$ by requiring that
	$\inner{u\ind{j}}{a\ind{1}}=0$ for every $j\in[d]$, whence the first iterate
	of $\algzr$ satisfies $z\ind{1} = 0 \in \R^d$. Assume now the equality 
	and
	containment~\eqref{eq:prelims-induciton} holds for every $s < t$, where 
	$t
	\le T_0$ (implying that $\algzr$ has emulated the iterates $a\ind{2},
	\ldots, a\ind{t-1}$ of $\alg$); we show how $\algzr$ can emulate 
	$a\ind{t}$,
	the $t$'th iterate of $\alg$, and from it can construct $z\ind{t}$ that
	satisfies~\eqref{eq:prelims-induciton}. To obtain $a\ind{t}$, note that for
	every $q\le p$, and every $s < t$, 
	the derivatives $\deriv{q} f_U(a\ind{s})$ are
	a function of $\deriv{q} f(z\ind{s})$ and
	orthonormal the vectors $\{u\ind{i}\}_{i \in S_{s+1}}$,
	because $\supports{\deriv{q} f(z\ind{s})}
	\subseteq S_{s+1}$ and therefore
	the chain rule implies
	\begin{equation*}
	\left[\deriv{q}f_U(a\ind{s})\right]_{j_1, ..., j_q} = 
	\sum_{i_1, \ldots, i_q \in S_{s+1}}
	\left[\deriv{q} f(z\ind{s})\right]_{i_1, ..., i_q}
	u\ind{i_1}_{j_1} \cdots u\ind{i_q}_{j_q}.
	\end{equation*}
	Since $\alg\in\AlgDet\ind{p}$ is deterministic, $a\ind{t}$ is a function of
	$\deriv{q} f(z\ind{s})$ for $q\in[p]$ and $s\in[t-1]$, and thus $\algzr$ 
	can
	simulate and compute it. To satisfy the support condition
	$\supports{z\ind{t}} \subseteq S_t$ we require that
	$\inner{u\ind{j}}{a\ind{t}}=0$ for every $j\not\in S_t$. This also means
	that to compute $z\ind{t} = U^{\top} a\ind{t}$ we require only the 
	columns 
	of $U$
	indexed by the support $S_t$.
	
	Finally, we need to show that after computing $S_{t+1}$ we can find the
	vectors $\{u\ind{i}\}_{i \in S_{t+1} \setminus S_t}$ 
	satisfying
	$\inner{u\ind{j}}{a\ind{s}}=0$ for every $s\le t$ and $j\in
	S_{t+1}\setminus S_t$, and additionally that $U$ be orthogonal. Thus, we
	need to choose $\{u\ind{i}\}_{i \in S_{t+1}\setminus S_t}$ in the 
	orthogonal
	complement of $\mathrm{span}\left\{ a\ind{1}, ..., a\ind{t}, \{u\ind{i}\}_{i
		\in S_t}\right\}$. This orthogonal complement has dimension at least
	$d'-t-|S_t| = |S_t^c| + T_0 - t \ge |S_t^c|$. Since $|S_{t+1}\setminus 
	S_t|
	\le |S_t^c|$, there exist orthonormal vectors $\{u\ind{i}\}_{i \in
		S_{t+1}\setminus S_t}$ that meet the requirements. This completes the
	induction.
	
	Finally, note that the arguments above hold unchanged for $p=\infty$.
\end{proof}

With Lemma~\ref{lem:prelims-app-det-zr-core} in hand, the propositions 
follow easily.

\propPrelimsDetZR*

\begin{proof}
  We may assume that $\CompEps{\AlgDet\ind{p}}{\FclassBlank}  < T_0$ 
  for 
  some integer $T_0 < \infty$, as 
  otherwise we have $\CompEps{\AlgDet\ind{p}}{\FclassBlank}=\infty$ and 
  the result 
  holds trivially. 
  For any $\alg\in\AlgDet\ind{p}$ and the value 
  $T_0$, we invoke Lemma~\ref{lem:prelims-app-det-zr-core} to construct 
  $\algzr\in\AlgZR\ind{p}$ such that $\TimeEps{\alg}{f_U} \ge 
  \min\{T_0, \TimeEps{\algzr}{f}\}$ for every $f\in\FclassBlank$ and 
  some 
  orthogonal 
  matrix 
  $U$ that depends on $f$ and $\alg$. 
  Consequently, we have
  \begin{flalign*}
    \CompEps{\AlgDet\ind{p}}{\FclassBlank} & =
    \inf_{\alg\in\AlgDet\ind{p}}\sup_{f\in\FclassBlank} \TimeEps{\alg}{f}
    \stackrel{(i)}{\ge}
    \inf_{\alg\in\AlgDet\ind{p}}\sup_{f\in\FclassBlank} \TimeEps{\alg}{f_U}
    \stackrel{(ii)}{\ge}
    \min\Big\{T_0,\inf_{\alg\in\AlgDet\ind{p}}\sup_{f\in\FclassBlank} 
    \TimeEps{\algzr}{f}\Big\}
    \\
    & 
    \stackrel{(iii)}{\ge}
    \min\Big\{T_0,\inf_{\mathsf{B} 
    \in\AlgZR\ind{p}}\sup_{f\in\FclassBlank}\TimeEps{\mathsf{B}}{f}\Big\}
    = \min\Big\{T_0,\CompEps{\AlgZR\ind{p}}{\FclassBlank}\Big\},
  \end{flalign*}
  where inequality $(i)$ uses that $f_U \in \mc{F}$ because 
  $\mc{F}$ is orthogonally invariant, step $(ii)$
  uses $\TimeEps{\alg}{f_U} \ge 
  \min\{T_0, \TimeEps{\algzr}{f}\}$ and step $(iii)$ is due to 
  $\algzr \in \AlgZR\ind{p}$ by construction. As we chose $T_0$ for which 
  $\CompEps{\AlgDet\ind{p}}{\FclassBlank}  < T_0$, the chain of 
  inequalities implies 
  $ \CompEps{\AlgDet\ind{p}}{\FclassBlank} \ge 
  \CompEps{\AlgZR\ind{p}}{\FclassBlank}$, concluding the proof.
\end{proof}

\propPrelimsDetZRdim*
\begin{proof}
	For any $\alg\in\AlgDet\ind{p}$, we invoke 
	Lemma~\ref{lem:prelims-app-det-zr-core} with $T_0 = T$ to obtain 
	$\algzr\in\AlgZR\ind{p}$ and orthogonal matrix $U'$ (dependent on $f$ 
	and $\alg$) for which
	\begin{equation*}
	\TimeEps{\alg}{f_{U'}} \ge \min\{T, \TimeEps{\algzr}{f}\} = T,
	\end{equation*}
	where the last equality is due to 
	$\inf_{\mathsf{B}\in\AlgZR\ind{p}}\TimeEps{\mathsf{B}}{f} = 
	\CompEps{\AlgZR\ind{p}}{\{f\}} \ge T$. Since $f_{U'} \in\{f_U \mid U\in 
	\orthogonalgroup(d+\T,d)\}$, we have
	\begin{equation*}
	\sup_{f'\in\{f_U \mid U\in 
		\orthogonalgroup(d+\T,d)\}}\TimeEps{\alg}{f'} \ge T,
	\end{equation*}
	and taking the infimum over $\alg\in\AlgDet\ind{p}$ concludes the 
	proof.
\end{proof}

\section{Technical Results}
\label{sec:fullder-proofs}

\subsection{Proof of Lemma~\ref{lem:fullder-props}}
\label{sec:proof-fullder-props}

\lemFullderProps*

Each of the statements in the lemma is immediate except for
part~\ref{item:fullder-psiphi-props-infinite}. To see this part, we require a few
further calculations. We begin by providing bounds on the derivatives of
$\gausscdf(x) = e^\half \int_{-\infty}^x e^{-\half t^2} dt$. To avoid
annoyances with scaling factors, we define $\gausspdf(t) = e^{-\half
  t^2}$.

\begin{lemma}
  \label{lemma:bound-gaussian-derivatives}
  For all $k \in \N$, there exist constants $c_i^{(k)}$
  satisfying $|c_i^{(k)}| \le  (2\max\{i,1\})^k$, and
  \begin{equation*}
    \gausspdf^{(k)}(t) = \bigg(\sum_{i = 0}^k c_i^{(k)} t^i\bigg) \gausspdf(t).
  \end{equation*}
\end{lemma}

\begin{proof}
  We prove the result by induction. We have $\gausspdf'(t) = -t e^{-\half
    t^2}$, so that the base case of the induction is satisfied. Now, assume
  for our induction that
  \begin{equation*}
    \gausspdf^{(k)}(t)
    = \sum_{i = 0}^k c_i^{(k)} t^i e^{-\half t^2}
    = \sum_{i = 0}^k c_i^{(k)} t^i \gausspdf(t).
  \end{equation*}
  where $|c_i^{(k)}| \le 2^k (\max\{i, 1\})^k$.
  Then taking derivatives, we have
  \begin{equation*}
    \gausspdf^{(k + 1)}(t)
    = \sum_{i = 1}^k
    \left[ i \cdot c_i^{(k)} t^{i - 1} 
      - c_i^{(k)} t^{i + 1} \right] \gausspdf(t)
    - c_0^{(k)} t \gausspdf(t)
    = \sum_{i = 0}^{k + 1} c_i^{(k + 1)} t^i \gausspdf(t)
  \end{equation*}
  where $c_i^{(k + 1)} = (i + 1) c_{i + 1}^{(k)} -c_{i-1}^{(k)}$
  (and we treat $c_{k + 1}^{(k)} = 0$)
  and $|c_{k + 1}^{(k + 1)}| = 1$. With the induction hypothesis that
  $c_i^{(k)} \le  (2\max\{i,1\})^k$, we obtain
  \begin{equation*}
    |c_i^{(k + 1)}|
    \le 2^k (i + 1) (i + 1)^k + 2^k  (\max\{i,1\})^k
    \le 2^{k+1} (i + 1)^{k + 1}.
  \end{equation*}
  This gives the result.
\end{proof}

\noindent
With this result, we find that for any $k\ge1$,
\begin{equation*}
  \gausscdf^{(k)}(x)
  = \sqrt{e} \bigg(\sum_{i = 0}^{k - 1}
  c_i^{(k-1)} x^i \bigg) \gausspdf(x).
\end{equation*}
The function $\log(x^i \gausspdf(x)) = i \log x - \half x^2$ is maximized at
$x = \sqrt{i}$, so that $x^i \gausspdf(x) \le \exp(\frac{i}{2} \log \frac{i}{e})$. We thus obtain the numerically verifiable upper bound
\begin{flalign*}
  |\gausscdf^{(k)}(x)|  & \le 
   \sqrt{e}
  \sum_{i = 0}^{k - 1}
  \left(2\max\{i,1\}\right)^{k-1} \exp\left(\frac{i}{2} \log \frac{i}{e}
  \right) \le
  \exp\left(1.5 k \log (1.5k)\right).
\end{flalign*}

Now, we turn to considering the function $\compactfunc(x)$.
We assume w.l.o.g.\ that $x > \half$, as
otherwise $\compactfunc^{(k)}(x) = 0$ for all $k$. Recall  $\compactfunc(x) = \exp\left(1-\frac{1}{(2x - 1)^2}\right)$ for $x > \half$.
We have the following lemma regarding its derivatives.

\begin{lemma}
  \label{lemma:bound-compact-derivatives}
  For all $k \in \N$, there exist constants $c_i^{(k)}$ satisfying
  $|c_i^{(k)}| \le 6^k (2i + k)^k$ such that
  \begin{equation*}
    \compactfunc^{(k)}(x)
    = \bigg(\sum_{i = 1}^k \frac{c_i^{(k)}}{(2x - 1)^{k + 2i}}
    \bigg) \compactfunc(x).
  \end{equation*}
\end{lemma}
\noindent
\begin{proof}
  We provide the proof by induction over $k$. For $k = 1$, we have that
  \begin{equation*}
    \compactfunc'(x) = \frac{4}{(2x - 1)^3} \exp\left(1-
    \frac{1}{(2 x - 1)^2} \right)
    = \frac{4}{(2 x - 1)^3} \compactfunc(x),
  \end{equation*}
  which yields the base case of the induction.
  Now, assume that for some $k$, we have
  \begin{equation*}
    \compactfunc^{(k)}(x)
    = \left(\sum_{i = 1}^k \frac{c_i^{(k)}}{(2x - 1)^{k + 2i}}
    \right) \compactfunc(x).
  \end{equation*}
  Then
  \begin{align*}
    \compactfunc^{(k + 1)}(x)
    & = \left(-\sum_{i = 1}^k 
    \frac{2 (k + 2i) c_i^{(k)}}{(2x - 1)^{k + 1 + 2i}}
    + \sum_{i = 1}^k 
    \frac{4 c_i^{(k)}}{(2 x - 1)^{k + 3 + 2i}}
    \right) \compactfunc(x) \\
    & = \left(\sum_{i = 1}^{k + 1}
    \frac{4 c_{i-1}^{(k)}
      - 2 (k + 2i) c_i^{(k)}}{
      (2x - 1)^{k + 1 + 2i}}\right) \compactfunc(x),
  \end{align*}
  where $c_{k + 1}^{(k)} = 0$ and $c_{0}^{(k)} = 0$. Defining $c_{1}^{1} = 4$ and
  $c_i^{(k + 1)} = 4 c_{i - 1}^{(k)} - 2 (k + 2i) c_i^{(k)}$ for $i > 1$,
  then, under the inductive hypothesis that
  $|c_i^{(k)}| \le 6^k (2i + k)^k$, we have
  \begin{equation*}
    |c_i^{(k + 1)}|
    \le 4 \cdot 6^k (k - 2 + 2i)^k + 2 \cdot 6^k (k + 2i) (k + 2i)^k
    \le 6^{k + 1} (k + 2i)^{k + 1}
    \le 6^{k + 1} (k + 1 + 2i)^{k + 1}
  \end{equation*}
  which gives the result.
\end{proof}

As in the derivation immediately following
Lemma~\ref{lemma:bound-gaussian-derivatives},
by replacing $t = \frac{1}{2x - 1}$, we have that
$t^{k + 2i} e^{- t^2}$ is maximized by
$t = \sqrt{(k + 2i)/2}$, so that
\begin{equation*}
  \frac{1}{(2 x - 1)^{k + 2i}} \compactfunc(x)
  \le \exp\left(1 + \frac{k + 2i}{2} \log \frac{k + 2i}{2e} \right),
\end{equation*}
which yields the numerically verifiable upper bound
\begin{equation*}
  |\compactfunc^{(k)}(x)|
  \le \sum_{i = 1}^k
  \exp\left(1 + k\log(6k+12i) + \frac{k + 2i}{2} \log \frac{k + 2i}{2e}\right)
  \le \exp\left(2.5 k \log(4 k)\right).
\end{equation*}

\subsection{Proof of Lemma~\ref{lem:fullder-bounded}}\label{sec:fullder-bounded-proof}

\lemFullderBounded*

\begin{proof}
  Part~\ref{item:fullder-funcbound} follows because $\fhard(0) < 0$ and,
  since $0 \le \compactfunc(x) \le e$ and $0 \le \gausscdf(x) \le \sqrt{2 \pi
    e}$,
  \begin{equation*}
    \fhard(x) \ge 
    -\compactfunc\left(1\right)\gausscdf\left(x_{1}\right)-\sum_{i=2}^{\T}
    \compactfunc\left(x_{i-1}\right)\gausscdf\left(x_{i}\right)
    >
    -T \cdot e \cdot \sqrt{2 \pi e} \ge - 12T.
  \end{equation*}
  
  Part~\ref{item:fullder-gradbound} follows additionally from $\compactfunc(x) = 0$ on $x < 1/2$,  $0 \le
  \compactfunc'(x) \le \sqrt{54e^{-1}}$ and $0 \le \gausscdf'(x) \le
  \sqrt{e}$, which when substituted into 
  $$\frac{\del \fhard}{\del x_j}(x)  = -\compactfunc\left(-x_{j-1}\right)\gausscdf'\left(-x_{j}\right)
    -\compactfunc\left(x_{j-1}\right)\gausscdf'\left(x_{j}\right)
    -\compactfunc'\left(-x_{j}\right)\gausscdf\left(-x_{j+1}\right) 
    -\compactfunc'\left(x_{j}\right)\gausscdf\left(x_{j+1}\right)$$
  yields
  \begin{equation*}
    \left\vert \frac{\del \fhard}{\del x_j}(x)\right\vert \le e\cdot \sqrt{e} + \sqrt{54e^{-1}} \cdot \sqrt{2 \pi e}
    \le 23
  \end{equation*}
  for every $x$ and $j$. Consequently, $\norm{\grad \fhard (x)} \le
  \sqrt{T}
  \le 23 \sqrt{T}$.
  
  To establish part~\ref{item:fullder-lipschitz}, fix a point $x\in\R^{\T}$
  and a unit vector $v\in\mathbb{R}^{\T}$.  Define the real function $h_{x,v}
  : \R \to \R$ by the directional projection of $\fhard$, $h_{x,v}(\theta)
  \defeq \fhard(x+\theta v)$. The function $\theta \mapsto h_{x,v}(\theta)$
  is infinitely differentiable for every $x$ and $v$. Therefore, $\fhard$
  has $\smC{p}$-Lipschitz $p$-th order derivatives if and only if
  $|h_{x,v}^{(p+1)}(0)| \le \smC{p}$ for every $x$, $v$. Using the
  shorthand notation $\del_{i_1}\cdots \del_{i_k}$ for $\frac{\del^k}{\del
    x_{i_1} \cdots \del x_{i_k}}$, we have
  \begin{equation*}
    h_{x,v}^{\left(p+1\right)}\left(0\right)
    =\sum_{i_{1}, \ldots, i_{p+1}=1}^{\T}
    \del_{i_{1}}\cdots\del_{i_{p+1}}\fhard\left(x\right)v_{i_{1}}\cdots v_{i_{p+1}}\,.
  \end{equation*}
  Examining $\fhard$, we see that $\del_{i_{1}}\cdots\del_{i_{p+1}}\fhard$
  is non-zero if and only if $\left|i_{j}-i_{k}\right|\le1$ for every
  $j,k\in\left[p+1\right]$. Consequently, we can rearrange the above
  summation as 
  \begin{equation*}
    h_{x,v}^{\left(p+1\right)}\left(0\right) = 
    \sum_{\delta_{1},\delta_{2},\ldots,\delta_{p}\in\left\{ 0,1\right\}^{p}
      \cup \left\{ 0,-1\right\}^{p}}
    \sum_{i=1}^{\T}\del_{i+\delta_{1}}\cdots\del_{i+\delta_{p}}\del_{i}
    \fhard\left(x\right)v_{i+\delta_{1}}\cdots v_{i+\delta_{p}}v_{i},
  \end{equation*}
  where we take $v_0 \defeq 0$ and $v_{\T+1}\defeq 0$. 
Brief calculation show that
  \begin{flalign*}
    \sup_{x \in \R^\T}  &
    \max_{i \in [\T]} \max_{\delta \in \{0, 1\}^p
      \cup \{0, -1\}^p}
    \left|\del_{i+\delta_{1}}\cdots\del_{i+\delta_{p}}\del_{i}\fhard(x)\right|  
    \le \max_{k\in[p+1]}  \left\{ 2\sup_{x\in\R} \left| \compactfunc^{(k)}(x)\right| \sup_{x'\in\R} \left| \gausscdf^{(p+1-k)}(x')\right| \right\}
    \\ & \le 
    2\sqrt{2\pi e}\cdot e^{2.5(p+1)\log(4 (p+1))} \le 
    \exp\left(2.5p\log p + 4p + 9\right). 
  \end{flalign*}
 where the second inequality uses Lemma~\ref{lem:fullder-props}.\ref{item:fullder-psiphi-props-infinite}, and $\gausscdf(x') \le \sqrt{2\pi e}$ for the case $k = p+1$.
  Defining
  $\smC{p} = 2^{p + 1} e^{2.5p\log p + 4p + 9} \le e^{2.5p\log p + 5p + 10}$, 
  we thus have
  \begin{equation*}
    \left|h_{x,v}^{\left(p+1\right)}\left(0\right)\right|\le
    \sum_{\delta \in\left\{ 0,1\right\} ^{p}\cup\left\{ 0,-1\right\} ^{p}}
    2^{-\left(p+1\right)}\smC{p}
    \left|\sum_{i=1}^{\T}v_{i+\delta_{1}}\cdots v_{i+\delta_{p}}v_{i}\right|
    \le\left(2^{p+1}-1\right)2^{-\left(p+1\right)}\smC{p}\le\smC{p},
  \end{equation*}
  where we have used $|\sum_{i=1}^{\T}v_{i+\delta_{1}}\cdots
  v_{i+\delta_{p}}v_{i}|\le1$ for every $\delta \in \{0, 1\}^p \cup \{0,
  -1\}^p$.  To see this last claim is true, recall that $v$ is a unit vector and note that
  \begin{equation*}
    \sum_{i=1}^{\T}v_{i+\delta_{1}}\cdots v_{i + \delta_{p}} v_{i}
    = \sum_{i=1}^{\T} v_{i}^{p + 1- \sum_{j=1}^{p} \delta_{j}}
    v_{i\pm1}^{\sum_{j=1}^{p}\delta_{j}}.
  \end{equation*}
  If $\delta = 0$ then $|\sum_{i=1}^{\T}v_{i+\delta_{1}}\cdots
  v_{i+\delta_{p}}v_{i}| = 
  | \sum_{i=1}^{\T}v_{i}^{p+1} |
  \le\sum_{i=1}^{\T}v_{i}^{2}=1$. 
  Otherwise, letting $1\le\sum_{j=1}^{p}|\delta_{j}|=n\le p$, the
  Cauchy-Swartz inequality implies
  \begin{equation*}
    \left|\sum_{i=1}^{\T}v_{i+\delta_{1}}\cdots v_{i+\delta_{p}}v_{i}\right| = 
    \left|\sum_{i=1}^{\T}v_{i}^{p+1-n}v_{i+s}^{n}\right|
    \le\sqrt{\sum_{i=1}^{\T}v_{i}^{2\left(p+1-n\right)}}
    \sqrt{\sum_{i=1}^{\T}v_{i+s}^{2n}}\le\sum_{i=1}^{\T}v_{i}^{2}=1,
  \end{equation*}
where $s = -1$ or $1$.
  This gives the result.
\end{proof}
\subsection{Proof of Lemma~\ref{lem:fullder-rand-slow}}
\label{sec:fullder-rand-slow-proof}

The proof of Lemma~\ref{lem:fullder-rand-slow} uses a number of auxiliary 
arguments, marked as Lemmas~\ref{lem:fullder-rand-slow-pu},%
~\ref{lem:fullder-rand-slow-deterministic} 
and~\ref{lem:fullder-rand-slow-uniform}. Readers looking to gain a 
high-level view of the proof of Lemma~\ref{lem:fullder-rand-slow} can 
safely skip the proofs of these sub-lemmas.
In the following, recall that $U\in \R^{d\times T}$ is drawn from the 
uniform distribution over $d\times T$ orthogonal matrices (satisfying $U^T 
U = I$, as $d>T$), that the columns of $U$ are denoted $u\ind{1}, 
\ldots, u\ind{T}$, and that $\fhardRot(x) = \fhard(U^\top x)$.

\lemFullderRandSlow*

  For $t \in \N$, let $P_t \in \R^{d\times d}$ denote the projection
  operator to the span of $x\ind{1}, u\ind{1}, \ldots, x\ind{t}, u\ind{t}$,
  and let $P_t^\perp = I - P_t$ denote its orthogonal complement. We define
  the event $G_t$ as
  \begin{equation}
    \label{eqn:G-def}
    G_{t}=\left\{ \max_{j\in\{t,\ldots,\T\}}
    \left|\innerbig{u\ind{j}}{P_{t-1}^\perp x\ind{t}}\right|
    \le \alpha\norm{P_{t-1}^\perp x\ind{t}} \right\}
    \text{ where }\alpha=\frac{1}{5R\sqrt{T}}.
  \end{equation}
  For every $t$, define
  \begin{equation*}
    G_{\le t}=\cap_{i\le t}G_{i}\text{ and }G_{<t}=\cap_{i<t}G_{i}\,.
  \end{equation*}
  The following linear-algebraic result justifies
  the definition~\eqref{eqn:G-def} of $G_t$.

	\begin{customlemma}{\ref{lem:fullder-rand-slow}a}
    \label{lem:fullder-rand-slow-pu}
    For all $t\le \T$, $G_{\le t}$ implies $|\inner{u\ind{j}}{x\ind{s}}| <
    1/2$ for every $s \in \{1,\ldots,t\}$ and every $j \in \{s, \ldots,
    \T\}$.
    \end{customlemma}

  \begin{proof}
  First, notice that since $G_{\le t}$ implies $G_{\le s}$ for every $s \le
  t$, it suffices to show that $G_{\le t}$ implies
  $|\inner{u\ind{j}}{x\ind{t}}| < 1/2$ for every $j \in \{t, \ldots,
  \T\}$. We will in fact prove a stronger statement:
  \begin{equation}\label{eq:fullder-rand-slow-pu}
    \text{
      For every $t$, $G_{< t}$ implies $\norm{ P_{t-1}u\ind{j}}^{2} \le 2\alpha^{2}\left(t-1\right)$ for every $j \in \{t, \ldots, \T\}$,
    }
  \end{equation}
  where we recall that $P_t \in \R^{d\times d}$ is the projection operator
  to the span of $x\ind{1}, u\ind{1}, \ldots, x\ind{t}, u\ind{t}$,
  $P_t^\perp = I_d - P_t$ and $\alpha = 1/(5R\sqrt{T})$. Before
  proving~\eqref{eq:fullder-rand-slow-pu}, let us show that it implies our
  result. Fixing $j \in \{t, \ldots, \T\}$, we have
  \begin{equation*}
    \absinnerbig{u\ind{j}}{x\ind{t}} \le \absinnerbig{u\ind{j}}{P_{t-1}^\perp x\ind{t}} + \absinnerbig{u\ind{j}}{P_{t-1}x\ind{t}}.
  \end{equation*}
  Since $G_t$ holds, its definition~\eqref{eqn:G-def} implies
  $|\inner{u\ind{j}}{P_{t-1}^\perp x\ind{t}}| \le \alpha \norm{P_{t-1}^\perp
    x\ind{t}} \le \alpha \norm{x\ind{t}}$. Moreover, by Cauchy-Schwarz and
  the implication~\eqref{eq:fullder-rand-slow-pu}, we have
  $\absinner{u\ind{j}}{P_{t-1}x\ind{t}} \le
  \norm{P_{t-1}u\ind{j}}\norm{x\ind{t}} \le
  \sqrt{2\alpha^2(t-1)}\norm{x\ind{t}}$. Combining the two bounds, we obtain
  the result of the lemma,
  \begin{equation*}
    \absinnerbig{u\ind{j}}{x\ind{t}} \le \norm{x\ind{t}}(\alpha + \sqrt{2\alpha^2(t-1)}) < \frac{5}{2}\sqrt{t}R\alpha \le \frac{1}{2},
  \end{equation*}
  where we have used $\norm{x\ind{t}}\le R$ and $\alpha = 1/(5R\sqrt{T})$.

  We prove bound~\eqref{eq:fullder-rand-slow-pu} by induction. The basis of
  the induction, $t=1$, is trivial, as $P_{0}=0$. We shall
  assume~\eqref{eq:fullder-rand-slow-pu} holds for $s\in \{1, \ldots, t-1\}$
  and show that it consequently holds for $s=t$ as well.  We may apply the
  Graham-Schmidt procedure on the sequence $x\ind{1}, u\ind{1}, \ldots,
  x\ind{t-1}, u\ind{t-1}$ to write
  \begin{equation}
    \label{eq:fullder-rand-slow-gs}
    \norm{P_{t-1}u\ind{j}}^{2} = 
    \sum_{i=1}^{t-1} \absinnerbig{\frac{P_{i-1}^\perp x\ind{i}}{\norm{P_{i-1}^\perp x\ind{i}}}}{u\ind{j}}^2 + 
    \sum_{i=1}^{t-1} \absinnerbig{\frac{\hat{P}_{i-1}^\perp u\ind{i}}{\norm{\hat{P}_{i-1}^\perp u\ind{i}}}}{u\ind{j}}^2
  \end{equation}
  where $\hat{P}_{k}$ is the projection to the span of $\{x\ind{1},
  u\ind{1}, \ldots, x\ind{k}, u\ind{k}, x\ind{k+1}\} $,
  \begin{equation*}
    \hat{P}_{k}=P_{k}+\frac{1}{\norm{ P_{k}^{\perp}x\ind{k+1}} ^{2}}\left(P_{k}^{\perp}x\ind{k+1}\right)
    \left(P_{k}^{\perp}x\ind{k+1}\right)^{\top}.
  \end{equation*}
  Then for every $j > i$ we have
  \begin{equation*}
    \innerbig{\hat{P}_{i-1}^\perp u\ind{i}}{u\ind{j}}=
    - \innerbig{\hat{P}_{i-1} u\ind{i}}{u\ind{j}} =  
    - \innerbig{{P}_{i-1} u\ind{i}}{u\ind{j}}
    -\frac{ \innerbig{u\ind{i}}{P_{i-1}^\perp x\ind{i}} 
      \innerbig{u\ind{j}}{P_{i-1}^\perp x\ind{i}}} {\norm{P_{i-1}^\perp x\ind{i}}^2},
  \end{equation*}
  where the equalities hold by $\innerbig{u\ind{i}}{u\ind{j}} = 0$,
  $\hat{P}_{i-1}^\perp = I - \hat{P}_{i-1}$, and the definition of
  $\hat{P}_{i-1}$.
  
  The $P_i$ matrices are projections,
  so ${P}_{i-1}^2 = {P}_{i-1}$, and Cauchy-Swartz and the induction hypothesis
  imply
  \begin{equation*}
    \absinnerbig{{P}_{i-1} u\ind{i}}{u\ind{j}} = \absinnerbig{{P}_{i-1} u\ind{i}}{{P}_{i-1} u\ind{j}} \le \norm{ P_{i-1}u\ind{i}} \norm{ P_{i-1}u\ind{j}} \le 2\alpha^{2}\cdot\left(i-1\right).
  \end{equation*}	
  Moreover, the event $G_{i}$ implies $ \left| \inner{u\ind{i}}{P_{i-1}^\perp x\ind{i}} 
  \inner{u\ind{j}}{P_{i-1}^\perp x\ind{i}} \right| \le \alpha^{2}\norm{P_{i-1}^\perp x\ind{i}}^2$,
  so
  \begin{subequations}
    \begin{equation}
      \label{eqn:proj-ui-uj}
      \absinnerbig{\hat{P}_{i-1}^\perp u\ind{i}}{u\ind{j}} \le \absinnerbig{{P}_{i-1} u\ind{i}}{u\ind{j}} + \left|\frac{ \innerbig{u\ind{i}}{P_{i-1}^\perp x\ind{i}} 
        \innerbig{u\ind{j}}{P_{i-1}^\perp x\ind{i}}} {\norm{P_{i-1}^\perp x\ind{i}}^2}\right| \le \alpha^{2}\left(2i-1\right) \le \frac{\alpha}{2},
    \end{equation}
    where the last transition uses
    $\alpha=\frac{1}{5R\sqrt{\T}}\le\frac{1}{4i}$ because $R\ge\sqrt{\T}\ge
    i$. We also have the lower bound
    \begin{equation}
      \label{eqn:proj-u-u}
      \norm{ \hat{P}_{i-1}^{\perp}u\ind{i}} ^{2} = \absinnerbig{\hat{P}_{i-1}^\perp u\ind{i}}{u\ind{i}} =1-\norm{ P_{i-1}u\ind{i}} ^{2} -
      \frac{ \left(\innerbig{u\ind{i}}{P_{i-1}^\perp x\ind{i}}\right)^{2}}{\norm{ P_{i-1}^{\perp}x\ind{i}} ^{2}}\ge1-\alpha^{2}\left(2i-1\right)\ge\frac{1}{2},
    \end{equation}
  \end{subequations}
  where the first equality uses $(P_{i-1}^\perp)^2 = P_{i-1}^\perp$,
  the second the definition of $\hat{P}_{i-1}$, and the inequality uses
  $\inner{u\ind{j}}{P_{i-1}^\perp x\ind{i}} \le \alpha \norms{P_{i-1}^\perp
    x\ind{i}}$ and $\norms{ P_{i-1}u\ind{j}}^{2} \le
  2\alpha^{2}\left(i-1\right)$.

  Combining the observations~\eqref{eqn:proj-ui-uj}
  and~\eqref{eqn:proj-u-u}, we can bound each summand in the second
  summation in~\eqref{eq:fullder-rand-slow-gs}. Since the summands in the
  first summation are bounded by $\alpha^2$ by definition~\eqref{eqn:G-def}
  of $G_i$, we obtain
  \begin{equation*}
    \normbig{P_{t-1}u\ind{j}}^{2}
    \le \sum_{i=1}^{t-1}\alpha^{2}
    +\sum_{i=1}^{t-1}\frac{\left(\alpha/2\right)^{2}}{1/2}
    =\alpha^{2}\left(t-1+\frac{t-1}{2}\right)\le2\alpha^{2}\left(t-1\right),
  \end{equation*}
  which completes the induction. 
\end{proof}

  \noindent
  By Lemma~\ref{lem:fullder-rand-slow-pu} the event $G_{\le T}$ implies our
  result, so using
   $\P(G_{\le T}^c) \le \sum_{t = 1}^\T \P(G_t^c \mid G_{< t})$,
  it suffices to show that
  \begin{equation}
    \P\left(G_{\le \T}^c\right)
    \le \sum_{t=1}^{\T} \P( G_t^c \mid G_{<t} ) \le \delta.
    \label{eqn:smallprob-Gs}
  \end{equation}
  Let us therefore consider $\P\left(G_{t}^{c} \mid G_{<t}\right)$.  By the
  union bound and fact that $\norm{ P_{t-1}^{\perp}u\ind{j}}
  \le1$ for every $t$ and $j$,
  \begin{align}
    \P (G_{t}^{c} \mid G_{<t})
    & \le 
    \sum_{j \in \{t,\ldots,T\}} \P\left(
    \absinnerbig{u\ind{j}}{\unitvec{P_{t-1}^\perp x\ind{t}}} > \alpha
    \mid G_{<t} \right) \nonumber \\ &
    = \sum_{j \in \{t,\ldots,T\}} \E_{\xi, U_{(<t)}} 
    \P\left( \absinnerbig{u\ind{j}}{\unitvec{P_{t-1}^\perp x\ind{t}}} > \alpha
    \mid \xi, U_{(<t)}, G_{<t} \right) \nonumber \\ &
    \le \sum_{j \in \{t,\ldots,T\}} \E_{\xi, U_{(<t)}} 
    \P\left( \absinnerbig{\unitvec{P_{t-1}^\perp u\ind{j}}} {\unitvec{P_{t-1}^\perp x\ind{t}}} > \alpha
    \mid \xi, U_{(<t)}, G_{<t} \right),
    \label{eq:fullder-rand-slow-Gc-ub}
  \end{align}
  where $U_{(<t)}$ is shorthand for $u\ind{1},\ldots,u\ind{t-1}$ and $\xi$
  is the random variable generating $x\ind{1},\ldots,x\ind{\T}$.
  
  In the following lemma, we state formally that conditioned on $G_{<i}$,
  the iterate $x\ind{i}$ depends on $U$ only through its first $(i-1)$ columns.

\begin{customlemma}{\ref{lem:fullder-rand-slow}b}
    \label{lem:fullder-rand-slow-deterministic}
  For every $i\le \T$, there exist measurable functions $\alg\ind{i}_+$ and $\alg\ind{i}_{-}$ such that
    \begin{equation}\label{eq:fullder-rand-slow-xi-det}
      x\ind{i}=\alg\ind{i}_+\left(\xi,U_{(<i)}\right)
      \indic{G_{<i}} +
      \alg\ind{i}_{-}\left(\xi,U\right)\indic{G_{<i}^{c}}.
    \end{equation}
\end{customlemma}

    \begin{proof}
    Since the iterates are informed by $\fhardRot$, we may write each one as
    (recall definition~\eqref{eq:prelims-randomized-alg})
    \begin{equation*}
      x\ind{i}=\alg\ind{i}\left(\xi, \deriv{(0,\ldots,p)}
      \fhardRot(x\ind{1}), \ldots,
      \deriv{(0,\ldots,p)}\fhardRot(x\ind{i-1})\right)
      = \alg\ind{i}_{-}\left(\xi,U\right),
    \end{equation*}
    for measurable functions $\alg\ind{i},\alg\ind{i}_{-}$, where we recall
    the shorthand $\deriv{(0,\ldots,p)} h(x)$ for the derivatives of $h$ at $x$
    to order $p$. Crucially, by Lemma~\ref{lem:fullder-rand-slow-pu},
    $G_{<i}$ implies $\absinner{u\ind{j}}{x\ind{s}} < \frac{1}{2}$ for every
    $s<i$ and every $j\ge s$. As $\fhard$ is a fixed robust zero-chain (Definition~\ref{def:robust-zero-chain}), for any $s<i$, the derivatives of
    $\fhardRot$ at $x\ind{s}$ can therefore be expressed as functions of
    $x\ind{s}$ and $u\ind{1},\ldots,u\ind{s-1}$, and---applying this
    argument recursively---we see that $x\ind{i}$ is of the
    form~\eqref{eq:fullder-rand-slow-xi-det} for every $i\le \T$.
  \end{proof}

  Consequently (as $G_{<t}$ implies $G_{<i}$ for every $i\le t$),
  conditioned on $\xi, U_{(<t)}$ and $G_{<t}$, the iterates $x\ind{1},
  \ldots, x\ind{t}$ are deterministic, and so is
  $P_{t-1}^{\perp}x\ind{t}$. If $P_{t-1}^{\perp}x\ind{t}=0$ then
  $G_t$ holds and $\P (G_{t}^{c}\mid G_{<t})=0$, so we may assume without
  loss of generality that $P_{t-1}^{\perp}x\ind{t}\neq0$. We may therefore
  regard $\unitvecsmall{P_{t-1}^\perp x\ind{t}}$
  in~\eqref{eq:fullder-rand-slow-Gc-ub} as a deterministic unit
  vector in the subspace $P_{t-1}^\perp$ projects to. We
  now characterize the
  conditional distribution of $\unitvecsmall{P_{t-1}^\perp u\ind{j}}$.

\begin{customlemma}{\ref{lem:fullder-rand-slow}c}
    \label{lem:fullder-rand-slow-uniform}
    Let $t\le \T$, and $j\in\{t, \ldots, \T\}$.  Then conditioned on $\xi,
    U_{(<t)}$ and $G_{<t}$, the vector $\unitvec{P_{t-1}^\perp u\ind{j}}$ is
    uniformly distributed on the unit sphere in the subspace to which
    $P_{t-1}^\perp$ projects.
\end{customlemma}

  \begin{proof}
  	This lemma is subtle. The vectors
  	$u\ind{j}$, $j\ge t$, conditioned on $U_{(<t)}$, are certainly uniformly
  	distributed on the unit sphere in the subspace orthogonal to
  	$U_{(<t)}$. However, the additional conditioning on $G_{<t}$ requires
  	careful handling. 
  Throughout the proof we fix $t\le \T$ and $j\in\{t,\ldots,\T\}$. We begin
  by noting that by~\eqref{eq:fullder-rand-slow-pu}, $G_{<t}$ implies
  \begin{equation*}
    \norm{P_{t-1}^\perp u\ind{j} }^2 
    = 1 - \norm{P_{t-1} u\ind{j} }^2
    \ge 1 - 2\alpha^2(t-1) > 0.
  \end{equation*}
  Therefore, when $G_{<t}$ holds we have $P_{t-1}^\perp u\ind{j} \neq 0$ so
  $\unitvecsmall{P_{t-1}^\perp u\ind{j}}$ is well-defined.
  
  To establish our result, we will show that the density of $U_{(\ge t)} =
  [u\ind{t}, \ldots, u\ind{\T}]$ conditioned on $\xi, U_{( < t)}, G_{<t}$ is
  invariant to rotations that preserve the span of
  $x\ind{1},u\ind{1},\ldots,x\ind{t-1},u\ind{t-1}$. More formally, let
  $p_{\ge t}$ denote the density of $U_{(\ge t)}$ conditional on $\xi,U_{(< t)}$ 
  and $G_{<t}$. We wish to show that
  \begin{equation}\label{eq:fullder-rand-slow-rot-inv}
    p_{\ge t} \left(U_{(\ge t)} \mid \xi,U_{(<t)},G_{<t}\right) = 
    p_{\ge t} \left(Z U_{(\ge t)} \mid \xi,U_{(<t)},G_{<t}\right)
  \end{equation}
  for every rotation $Z\in\R^{d\times d}$, $Z^{\top} Z = I_d$, 
  satisfying
  \begin{equation*}
    Zv=v=Z^{\top}v
    ~~ \mbox{for~all} ~~
    v\in \left\{ x\ind{1},u\ind{1},\ldots,x\ind{t-1},u\ind{t-1}\right\}.
  \end{equation*}
  Throughout, we let $Z$ denote such a rotation. 
  Letting $p_{\xi,U}$ and $p_{U}$ denote the densities of
  $(\xi, U)$ and $U$, respectively, we have
  \begin{equation*}
    p_{\ge t} \left(U_{(\ge t)} \mid \xi, U_{(<t)},G_{<t}\right) 
    = \frac{\P\left(G_{<t}  \mid  \xi,U\right)  p_{\xi, U}\left(\xi, U \right)}{\P\left(G_{<t}  \mid  \xi,U_{(<t)}\right)  p_{\xi,U_{(<t)}}\left( \xi,U_{(<t)} \right) } 
    = \frac{\P\left(G_{<t}  \mid  \xi,U\right) p_{U}\left(U\right)}
    {\P\left(G_{<t}  \mid  \xi,U_{(<t)}\right) p_{U_{(<t)}}\left(U_{(<t)}\right)}
  \end{equation*}
  where the first equality holds by the definition of conditional
  probability and second by the independence of $\xi$ and $U$.  We have $Z
  U_{(<t)} = U_{(<t)} $ and therefore, by the invariance of $U$ to
  rotations, $p_U([U_{(<t)}, Z U_{(\ge t)}]) = p_U(Z U) = p_U(U)$. Hence,
  replacing $U$ with $ZU$ in the above display yields
  \begin{equation*}
    p_{\ge t} \left( Z U_{(\ge t)} \mid \xi,U_{(<t)},G_{<t}\right)
    = \frac{\P\left(G_{<t}  \mid  \xi, Z U\right) p_{U}\left(U\right)}
    {\P\left(G_{<t}  \mid  \xi,U_{(<t)}\right) p_{U_{(<t)}}\left(U_{(<t)}\right)}.
  \end{equation*}
  Therefore if we prove $\P(G_{<t}\mid
  \xi,U) = \P(G_{<t} \mid \xi,Z U)$---as we proceed to do---then
  we can conclude the equality~\eqref{eq:fullder-rand-slow-rot-inv} holds.
  
  First, note that $\P\left(G_{<t}\mid\xi,U\right)$ is supported on
  $\{0,1\}$ for every $\xi,U$, as they completely determine
  $x\ind{1},\ldots,x\ind{T}$. It therefore suffices to show that $\P(G_{<t}
  \mid \xi,U)=1$ if and only if $\P\left(G_{<t} \mid \xi,Z U\right)=1$. Set
  $U'=Z U$, observing that ${u'}\ind{i}=Z u\ind{i}=u\ind{i}$ for any $i<t$,
  and let ${x'}\ind{1},\ldots,{x'}\ind{\T}$ be the sequence generated from
  $\xi$ and $U'$. We will prove by induction on $i$ that
  $\P(G_{<t}\mid\xi,U)=1$ implies
  $\P(G_{<i}\mid\xi,U')=1$ for every $i\le t$. The basis of the
  induction is trivial as $G_{<1}$ always holds. Suppose now that
  $\P(G_{<i}\mid\xi,U')=1$ for $i<t$, and therefore
  ${x'}\ind{1},\ldots,{x'}\ind{i}$ can be written as functions of
  $\xi$ and ${u'}\ind{1},\ldots,{u'}\ind{i-1}=u\ind{1},\ldots,u\ind{i-1}$ by
  Lemma~\ref{lem:fullder-rand-slow-deterministic}.  Consequently,
  ${x'}\ind{l}=x\ind{l}$ for any $l\le i$ and also
  $P_{i-1}'^{\perp}{x'}\ind{i}=P_{i-1}^{\perp}x\ind{i}$.  Therefore, for any
  $l\ge i$,
  \begin{equation*}
    \absinnerbig{{u'}\ind{l}}{\unitvec{P_{i-1}'^{\perp}{x'}\ind{i}}}
    \stackrel{(i)}{=} 
    \absinnerbig{u\ind{l}}{Z^{\top} \unitvec{P_{i-1}^{\perp}{x}\ind{i}}}
    \stackrel{(ii)}{=} 
    \absinnerbig{u\ind{l}}{\unitvec{P_{i-1}^{\perp}{x}\ind{i}}}
    \stackrel{(iii)}{\le}  \alpha ,
  \end{equation*}
  where in $(i)$ we substituted ${u'}\ind{l} = Z u\ind{l}$ and
  $P_{i-1}'^{\perp}{x'}\ind{i}=P_{i-1}^{\perp}x\ind{i}$, $(ii)$ is because
  $P_{i-1}^{\perp}x\ind{i}=x\ind{i}-P_{i-1}x\ind{i}$ is in the span of 
  \mathprog{vectors} 
  $\left\{ x\ind{1},u\ind{1},\ldots,x\ind{i-1},u\ind{i-1},x\ind{i}\right\}$
  and therefore not modified by $Z^{\top}$, and $(iii)$ is by our assumption
  that $G_{<t}$ holds, and so $G_{i}$ holds. Therefore
  $\P\left(G_{i}\mid\xi,U'\right)=1$ and
  $\P\left(G_{<i+1}\mid\xi,U'\right)=1$, concluding the induction.  An
  analogous argument shows that $\P\left(G_{<t}\mid\xi,U'\right)=1$ implies
  $\P\left(G_{<t}\mid\xi,U\right)=\P\left(G_{<t}\mid\xi,Z^{\top}U'\right)=1$
  and thus $\P\left(G_{<t}\mid\xi,U\right)=\P\left(G_{<t}\mid\xi,Z U\right)$
  as required.
  
  Marginalizing the density~\eqref{eq:fullder-rand-slow-rot-inv} to obtain a
  density for $u\ind{j}$ and recalling that $P_{t-1}^\perp$ is measurable
  $\xi, U_{( < t)}, G_{<t}$, we conclude that, conditioned on
  $\xi, U_{( < t)}, G_{<t}$ the random variable $\unitvec{P_{t-1}^\perp
    u\ind{j} }$ has the same density as $\unitvec{ P_{t-1}^\perp Z u\ind{j}
  }$. However, $P_{t-1}^\perp Z = Z P_{t-1}^\perp$ by assumption on $Z$, and
  therefore
  \begin{equation*}
    \unitvec{ P_{t-1}^\perp Z u\ind{j} } = Z \unitvec{ P_{t-1}^\perp u\ind{j} }.
  \end{equation*} 
  We conclude that the conditional distribution of the unit vector
  $\unitvec{ P_{t-1}^\perp u\ind{j} }$ is invariant to rotations in the
  subspace to which $P_{t-1}^\perp$ projects.
\end{proof}
  
  Summarizing the discussion above, the conditional probability
  in~\eqref{eq:fullder-rand-slow-Gc-ub} measures the inner product
  of two unit vectors in a subspace of $\R^d$ of dimension
  $d'=\tr\left(P_{t-1}^{\perp}\right)\ge d-2\left(t-1\right)$, with one of
  the vectors deterministic and the other uniformly distributed. We may
  write this as
  \begin{equation*}
    \P\left( \absinnerbig{\unitvec{P_{t-1}^\perp u\ind{j}}} {\unitvec{P_{t-1}^\perp x\ind{t}}} > \alpha
    \mid \xi, U_{(<t)}, G_{<t} \right) = \P( |v_1| > \alpha),
  \end{equation*}
  where $v$ is uniformly distributed on the unit sphere in $\R^{d'}$. By a
  standard concentration of measure bound on the
  sphere~\cite[Lecture~8]{Ball97},
  \begin{equation*}
    \P( |v_1| > \alpha) \le 2e^{-d'\alpha^2/2}
    \le 2e^{-\frac{\alpha^{2}}{2}\left(d-2t\right)}.
  \end{equation*}	
  Substituting this bound back into the
  probability~\eqref{eq:fullder-rand-slow-Gc-ub} gives
  \begin{equation*}
    \P\left(G_{t}^{c}\mid G_{<t}\right)\le 2\left(\T-t+1\right)e^{-\frac{\alpha^{2}}{2}\left(d-2t\right)}
    \le 2 \T e^{-\frac{\alpha^{2}}{2}\left(d-2\T\right)}.
  \end{equation*}
  Substituting this in turn into the bound~\eqref{eqn:smallprob-Gs}, we have
  $\P(G_{\le T}^c) \le \sum_{t = 1}^\T \P(G_t^c \mid G_{< t}) \le 2 \T^2
  e^{-\frac{\alpha^2}{2} (d - 2\T)}$.  Setting $d\ge52\T
  R^{2}\log\frac{2\T^{2}}{\delta}\ge\frac{2}{\alpha^{2}}\log\frac{2\T^{2}}{\delta}+2\T$
  establishes $\P(G_{\le \T}^c)\le \delta$, concluding
  Lemma~\ref{lem:fullder-rand-slow}. \qed

\subsection{Proof of Lemma~\ref{lem:fullder-rand-bounded-props}}\label{sec:fullder-rand-bounded-props-proof}

\lemFullderRandBoundedProps*

\newcommand{\derivtil}[1]{\wt{\nabla}^{{#1}}}

\begin{proof}
  Part~\ref{item:random-bounds-f-gap} holds because
  $\fhardBound\left(0\right)=\fhard\left(0\right)$ and
  $\fhardBound\left(x\right)\ge\fhardRot\left(\rho(x)\right)$ for every $x$, so
  \begin{equation*}
    \inf_{x\in\R^{d}}\fhardBound\left(x\right)
    \ge \inf_{x\in\R^{d}}\fhardRot\left(\rho(x)\right)
    = 
    \inf_{\norm{x} \le R} \fhard\left(x\right)
    \ge \inf_{x\in\R^{d}}\fhard\left(x\right),
  \end{equation*}
  and therefore by
  Lemma~\ref{lem:fullder-bounded}.\ref{item:fullder-funcbound}, we have
  $\fhardBound(0) - \inf_{x}\fhardBound(x) \le \fhard(0)-\inf_{x}\fhard(x)
  \le 12T$.

  Establishing part~\ref{item:random-bounds-lipschitz} requires
  substantially more work.  Since smoothness with respect to Euclidean
  distances is invariant under orthogonal transformations, we take $U$ to be the first $\T$ columns of the $d$-dimensional identity matrix, denoted $U=I_{d,\T}$. Recall the scaling
  $\rho(x) = R x / \sqrt{R^2 + \norm{x}^2}$ with ``radius'' $R = 230
  \sqrt{\T}$ and the definition
  $\fhardBound(x) = \fhard(U^{\top}\rho(x)) + \frac{1}{10} \norm{x}^2$.
  The quadratic $\frac{1}{10} \norm{x}^2$ term in $\fhardBound$ 
  has $\frac{1}{5}$-Lipschitz first derivative and $0$-Lipschitz higher
  order derivatives (as they are all constant or zero),
  and we take $U = I_{d,\T}$ without loss of generality,
  so we consider the function
  \begin{equation*}
    \fhardBoundI(x) \defeq
    \fhard(\rho(x)) =
    \fhard\left(\left[\rho\left(x\right)\right]_{1},
    \ldots, \left[\rho\left(x\right)\right]_{\T}\right).
  \end{equation*}

  We now compute the partial derivatives of $\fhardBoundI$. 
  Defining $y = \rho(x)$, let $\derivtil{k}_{j_1, ..., j_k} \defeq \frac{\del^k}{\del y_{j_1}\cdots \del y_{j_k}}$ denote derivatives with respect to $y$. In addition,  
  define $\mc{P}_k$ to be the
  set of all partitions of $[k] = \{1, \ldots, k\}$, \ie $(S_1, \ldots,
  S_L) \in \mc{P}_k$ if and only if the $S_i$ are disjoint and $\cup_l S_l =
  [k]$.  Using the chain rule, we have for any $k$ and set of indices
  $i_1, \ldots, i_k \le T$ that
  \begin{equation}
    \label{eqn:crazy-partial-derivatives}
    \deriv{k}_{i_1, ..., i_k}
    \fhardBoundI(x) =
    \sum_{(S_1, \ldots, S_L) \in \mc{P}_k}
    \sum_{j_1, ..., j_L =1}^\T
    \bigg(\prod_{l = 1}^L
    \deriv{|S_l|}_{i_{{S_l}}} \rho_{j_l}(x)\bigg)
    \derivtil{L}_{j_1, ..., j_L} \fhard(y),
    ~~ y = \rho(x),
  \end{equation}
  where we have used the shorthand $\deriv{|S|}_{i_{S}}$ to denote
  the partial derivatives with respect to each of $x_{i_j}$ for $j \in S$.
  We use the equality~\eqref{eqn:crazy-partial-derivatives} to
  argue that (recall the identity~\eqref{eq:prelims-eckhart-young-tensors})
  \begin{equation*}
    \opnorm{\deriv{p+1} \fhardBoundI(x)}
    = \sup_{\norm{v} = 1}
    \inner{\deriv{p+1} \fhardBoundI(x)}{v^{\otimes (p+1)}} \defeq \smChat{p} - 
    \frac{1}{5}\indic{p=1} \le e^{c p\log p + c},
  \end{equation*}
  for some numerical constant\footnote{ 
  	To simplify notation we allow $c$ to change from equation to equation throughout the proof, always representing a finite numerical constant independent of $d$, $\T$, $k$ or $p$.}, 
   $0< c < \infty$ and every $p\ge1$. As explained in Section~\ref{sec:prelims-funcs}, this implies $\fhardBound$ has $e^{c p\log p + c}$-Lipschitz $p$th order derivative, giving part \ref{item:random-bounds-lipschitz} of the lemma.

  To do this, we begin by considering the partitioned
  sum~\eqref{eqn:crazy-partial-derivatives}. Let $v \in \R^d$ be an
  arbitrary direction with $\norm{v} = 1$. Then for $j \in [d]$ and $k
  \in \N$ we define the quantity
  \begin{equation*}
    \wt{v}_j^k
    = \wt{v}_j^k(x) \defeq \<\deriv{k} \rho_j(x), v^{\otimes k}\>,
  \end{equation*}
  algebraic manipulations and rearrangement of the
  sum~\eqref{eqn:crazy-partial-derivatives} yield
  \begin{align*}
    \<\deriv{k} \fhardBoundI(x), v^{\otimes k}\>
    & = \sum_{(S_1, \ldots, S_L) \in \mc{P}_k}
    \sum_{i_1, \ldots, i_k = 1}^d
    v_{i_1} v_{i_2} \cdots v_{i_k}
    \sum_{j_1, ..., j_L =1}^\T
    \bigg(\prod_{l = 1}^L
    \deriv{|S_l|}_{i_{{S_l}}} \rho_{j_l}(x)\bigg)
    \derivtil{L}_{j_1, ..., j_L} \fhard(y) \\
    & = \sum_{(S_1, \ldots, S_L) \in \mc{P}_k}
    \sum_{j_1, ..., j_L =1}^\T
    \wt{v}_{j_1}^{|S_1|}
    \cdots \wt{v}_{j_L}^{|S_L|}
    \derivtil{L}_{j_1, ..., j_L} \fhard(y) \\
    & = \sum_{(S_1, \ldots, S_L) \in \mc{P}_k}
    \innerbig{\derivtil{L} \fhard(y)}{\wt{v}^{|S_1|}
    \otimes \cdots \otimes \wt{v}^{|S_L|}}.
  \end{align*}
  We claim that there exists a numerical constant
  $c < \infty$ such that for all $k \in \N$,
  \begin{equation}
    \label{eqn:bound-tilde-v}
    \sup_x \norms{\wt{v}^k(x)} \le
    \exp(c k \log k + c) R^{1 - k}.
  \end{equation}
  Before proving inequality~\eqref{eqn:bound-tilde-v}, we
  show how it implies the desired lemma.
  By the preceding display, we have
  \begin{equation*}
    \absinner{\deriv{p+1}\fhardBoundI(x)}{v^{\otimes (p+1)}}
    \le 
    \sum_{(S_1, \ldots, S_L) \in \mc{P}_{p+1}}
    \opnorm{\derivtil{L} \fhard(y)}
    \prod_{l = 1}^L \norms{\wt{v}^{|S_l|}}.
  \end{equation*}
  Lemma~\ref{lem:fullder-bounded} shows that there
  exists a numerical constant $c < \infty$ such that
 \begin{equation*}
 	 \opnorm{\nabla^{(L)}
    \fhard(y)} \le \smC{L-1} \le \exp(c L \log L + c)~\mbox{for all }L \ge 2.
 \end{equation*}
  When the number of partitions $L = 1$, we 
  have $|S_1| = p+1 \ge 2$, and so
  Lemma~\ref{lem:fullder-bounded}.\ref{item:fullder-gradbound}
  yields
  \begin{equation*}
    \opnorm{\grad \fhard(y)}
    \norms{\wt{v}^{|S_1|}}
    = \norm{\grad \fhard(y)}
    \norms{\wt{v}^{|S_1|}}
    \le 23 \sqrt{T} \cdot R^{-p} \exp(c p \log p + c)
    \le \exp(c p \log p + c),
  \end{equation*}
  where we have used $R = 230 \sqrt{\T}$.
  Using $|S_1| + \cdots + |S_L| = p+1$ and the fact that $q(x) = (x+1)\log(x+1)$ satisfies $q(x)+q(y) \le q(x+y)$ for every $x,y>0$, we have
  \begin{equation*}
    \opnorm{\derivtil{L} \fhard(y)}
   \prod_{l = 1}^L \norms{\wt{v}^{|S_l|}} \le \exp(c p \log p + c)
  \end{equation*}
  for some $c < \infty$ and every $(S_1, \ldots, S_L) \in \mc{P}_{p+1}$. Bounds on Bell numbers~\cite[Thm.~2.1]{BerendTa10} give that there are at
  most $\exp(k \log k)$ partitions in $\mc{P}_k$, which combined with the bound above gives desired result.

  Let us return to the derivation of inequality~\eqref{eqn:bound-tilde-v}.
  We begin by recalling Fa\`{a} di Bruno's formula for the chain rule. Let
  $f, g : \R \to \R$ be appropriately smooth functions. Then
  \begin{equation}
    \label{eqn:faa-di-bruno}
    \frac{d^k}{dt^k}
    f(g(t))
    = \sum_{P \in \mc{P}_k}
    f^{(|P|)}(g(t)) \cdot \prod_{S \in P} g^{(|S|)}(t),
  \end{equation}
  where $|P|$ denotes the number of disjoint elements of partition $P \in
  \mc{P}_k$.   
  Define the function $\wb{\rho}(\xi) = \xi / \sqrt{1 + \norms{\xi}^2}$,
  and let $\lambda(\xi) = \sqrt{1 + \norms{\xi}^2}$ so that $\wb{\rho}(\xi)
  = \nabla \lambda(\xi)$ and $\rho(\xi) = R \wb{\rho}(\xi / R)$.
  Let $\wb{v}_j^k(\xi) = \<\deriv{k} \wb{\rho}_j(\xi), v^{\otimes k}\>$,
  so that
  \begin{equation}
    \wb{v}^k(\xi)
    = \nabla
    \<\deriv{k} \lambda(\xi), v^{\otimes k}\>
    ~~ \mbox{and} ~~
    \wt{v}^k = R^{1 - k} \wb{v}^k(x / R).
    \label{eqn:definitions-of-vs}
  \end{equation}
  With this in mind, we consider the quantity
  $\<\deriv{k} \lambda(\xi), v^{\otimes k}\>$. Defining
  temporarily the functions $\alpha(r) = \sqrt{1 + 2r}$ and
  $\beta(t) = \half \norms{\xi + t v}^2$, and their
  composition $h(t) = \alpha(\beta(t))$, we evidently have
  \begin{equation*}
    h^{(k)}(0) = \<\deriv{k} \lambda(\xi), v^{\otimes k}\>
    = \sum_{P \in \mc{P}_k}
    \alpha^{(|P|)}(\beta(0))
    \cdot \prod_{S \in P} \beta^{(|S|)}(0),
  \end{equation*}
  where the second equality used Fa\'{a} di Bruno's
  formula~\eqref{eqn:faa-di-bruno}.
  Now, we note the following immediate facts:
  \begin{equation*}
	\alpha^{(l)}(r)
	= (-1)^l \frac{(2l - 1)!!}{(1 + 2r)^{l - 1/2}}
	~~ \mbox{and} ~~
    \beta^{(l)}(t) = \begin{cases}
    	\inner{v}{\xi} + t \norm{v}^2 & l = 1 \\
    	 \norm{v}^2 & l = 2 \\ 
    	 0 & l > 2.
    \end{cases}
  \end{equation*}
  Thus, if we let $\mc{P}_{k,2}$ denote the partitions of
  $[k]$ consisting only of subsets with one or two elements, we have
  \begin{equation*}
    h^{(k)}(0)
    = \sum_{P \in \mc{P}_{k,2}}
    (-1)^{|P|} \frac{(2|P| - 1)!!}{(1 + \norm{\xi}^2)^{|P| - 1/2}}
    \<\xi, v\>^{\countset_1(P)}
    \norm{v}^{2\countset_2(P)}
  \end{equation*}
  where $\countset_i(P)$ denotes the number of sets in $P$ with
  precisely $i$ elements. Noting that $\norm{v} = 1$, we may
  rewrite this as
  \begin{equation*}
    \<\deriv{k} \lambda(\xi), v^{\otimes k}\>
    = \sum_{l = 1}^k
    \sum_{P \in \mc{P}_{k,2},
      \countset_1(P) = l}
    (-1)^{|P|}
    \frac{(2|P| - 1)!!}{(1 + \norm{\xi}^2)^{|P| - 1/2}} \<\xi, v\>^l.
  \end{equation*}
  Taking derivatives we obtain
  \begin{equation*}
    \wt{v}^k
    = \nabla
    \<\deriv{k} \lambda(\xi), v^{\otimes k}\>
    = \bigg(\sum_{l = 1}^k
    a_l(\xi) \<\xi, v\>^{l - 1}\bigg) v
    + \bigg(\sum_{l = 1}^k b_l(\xi) \<\xi, v\>^l \bigg) \xi
  \end{equation*}
  where
  \begin{equation*}
    a_l(\xi) = l \cdot \sum_{P \in \mc{P}_{k,2},
      \countset_1(P) = l}
    \frac{(-1)^{|P|} (2|P| - 1)!!}{(1 + \norm{\xi}^2)^{|P| - 1/2}}
    ~~ \mbox{and} ~~
    b_l(\xi) =
    \sum_{P \in \mc{P}_{k,2}, \countset_1(P) = l}
    \frac{(-1)^{|P| + 1}
      (2|P| + 1)!!}{(1 + \norm{\xi}^2)^{|P| + 1/2}}.
  \end{equation*}
  We would like to bound $a_l(\xi) \<\xi, v\>^{l-1}$ and $b_l(\xi) \<\xi,
  v\>^l \xi$.  Note that $|P| \ge \countset_1(P)$ for every $P\in\mc{P}_{k}$, so $|P|\ge l$ in the sums above. Moreover, bounds for Bell
  numbers~\cite[Thm.~2.1]{BerendTa10} show that there are at most $\exp(k
  \log k)$ partitions of $[k]$, and $(2k - 1)!! \le \exp(k \log k)$ as well.
  As a consequence, we obtain
  \begin{equation*}
    \sup_\xi
    |a_l(\xi) \<\xi, v\>^{l - 1}|
    \le \exp(c l \log l)
    \sup_\xi \frac{|\<\xi, v\>|^{l - 1}}{
      (1 + \norm{\xi}^2)^{(l - 1)/2}}
    < \exp(c l \log l),
  \end{equation*}
  where we have used $\absinner{\xi}{v}\le \norm{\xi}$ due to $\norm{v}=1$. We similarly bound 
  $\sup_\xi |b_l(\xi)| |\<\xi, v\>|^l \norm{\xi}$.
  Returning to expression~\eqref{eqn:definitions-of-vs}, we have
  \begin{equation*}
    \sup_x \norms{\wt{v}^k(x)}
    \le \exp\left(c k \log k + c\right)
    R^{1 - k},
  \end{equation*}
  for a numerical constant $c<\infty$.
  This is the desired bound~\eqref{eqn:bound-tilde-v}, completing the proof.
\end{proof}

\section{Proof of Theorem~\ref{thm:fullder-final-dist}}
\label{sec:proof-general-distance}

\thmFullderDistLB*

We divide the proof of the theorem into two parts, as in our previous
results, first providing a few building blocks, then giving the
theorem.  The basic idea is to introduce a negative ``bump'' that is
challenging to find, but which is close to the origin. %

To make this precise, 
let $e\ind{j}$ denote the $j$th standard basis
vector. Then we define the bump function $\hhard:\R^\T \to \R$
by
\begin{flalign}
  \label{eq:fullder-distance-hhard-def}
  \hhard(x) &
  = \compactfunc\left( 1 - \frac{25}{2}\norm{x - \frac{4}{5}e\ind{\T}}^2\right) 
  = 
  \begin{cases}
    0 & \norm{x -\frac{4}{5}e\ind{\T}} \ge \frac{1}{5}\\
    \exp\left( 1-\frac{1}{
      \left(1-25\norm{x -\frac{4}{5}e\ind{\T}}^2\right)^{2}}\right)
    & \mbox{otherwise.}
  \end{cases}
\end{flalign}
As Figure~\ref{fig:fullder-distance} shows, $\hhard$ features a unit-height
peak centered at $\frac{4}{5} e\ind{\T}$, and it is identically zero when the
distance from that peak exceeds $\frac{1}{5}$. The volume of the peak vanishes
exponentially with $\T$, making it hard to find by querying
$\hhard$ locally. We list the properties of $\hhard$ necessary for our
analysis.

\begin{figure}
	\centering
	\includegraphics[width=0.75\textwidth]{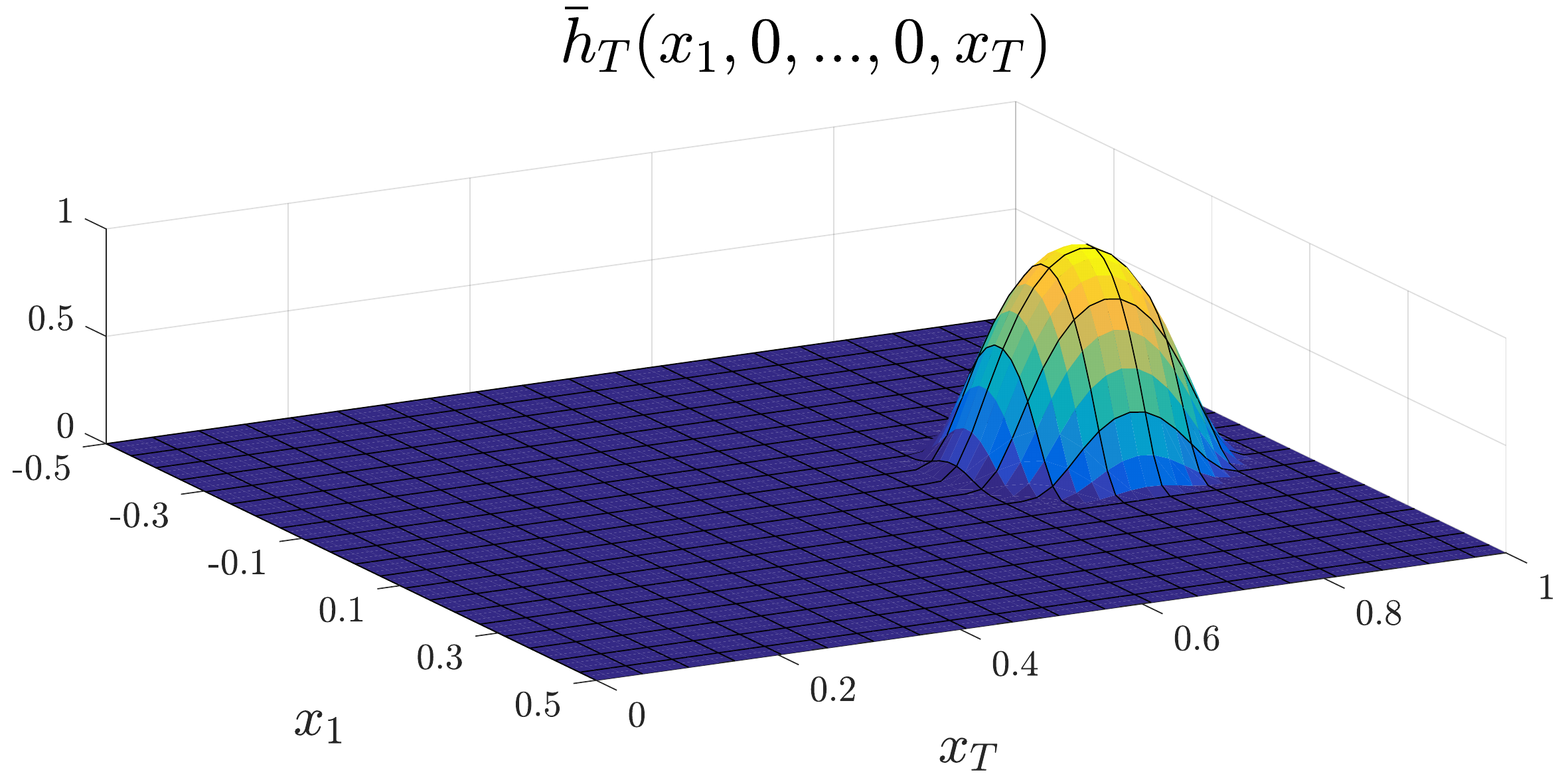} 
	\caption{Two-dimensional cross-section of the bump function $\hhard$.}\label{fig:fullder-distance}
\end{figure}

\begin{restatable}{lemma}{lemFullderHhardProps}\label{lem:fullder-hhard-props}
  The function $\hhard$ satisfies the following.
  \begin{enumerate}[i.]
  \item \label{item:fullder-hhard-bump} $\hhard\left(0.8 e\ind{\T}\right) =
    1$ and $\hhard(x)\in[0,1]$ for all $x\in\R^\T$.
  \item \label{item:fullder-hhard-radius} 
    $\hhard(x)=0$ on the set $\{x\in\R^d\,|\,x_{\T} \le \frac{3}{5} 
    \textnormal{ or }\norm{x}\ge 1\}$.
  \item \label{item:fullder-hhard-lipschitz} For $p \ge 1$, the $p$th
    order derivatives of $\hhard$ are $\smCvar{p}$-Lipschitz continuous,
    where $\smCvar{p} < e^{3 p\log p + c p}$ for some numerical constant $c 
    <
    \infty$.
  \end{enumerate}
\end{restatable}
\noindent
We prove the lemma in Section~\ref{sec:fullder-hhard-props-proof}; the proof
is similar to that of Lemma~\ref{lem:fullder-rand-bounded-props}. With these
properties in hand, we can prove Theorem~\ref{thm:fullder-final-dist}.

\subsection{Proof of Lemma~\ref{lem:fullder-hhard-props}}
\label{sec:fullder-hhard-props-proof}

Properties~\ref{item:fullder-hhard-bump} and~\ref{item:fullder-hhard-radius}
are evident from the definition~\eqref{eq:fullder-distance-hhard-def} of
$\hhard$. To show property~\ref{item:fullder-hhard-lipschitz}, consider
$\hhardVar(x) =
\hhard(\frac{x+0.8e\ind{T}}{5})=\compactfunc(1-\half\|x\|^2)$, which is a
translation and scaling of $\hhard$, so if we show $\hhardVar$ has
$(\smCvar{p}/5^{p+1})$-Lipschitz $p$th order derivatives, for every $p\ge1$,
we obtain the required results. For any $x,v\in\R^\T$ with $\norm{v}\le1$ we
define the directional projection $\hhardVar_{x,v}(t) = \hhardVar(x+t\cdot
v)$. The required smoothness bound is equivalent to
\begin{equation*}
  \left| \hhardVar_{x,v}^{(p+1)}(0) \right| \le \smCvar{p}/5^{p+1} \le e^{c p\log p + c}
\end{equation*} 
for every $x,v\in\R^d$ with $\norm{v}\le1$, every $p\ge1$ and some numerical
constant $c<\infty$ (which we allow to change from equation to equation,
caring only that it is finite and independent of $\T$ and $p$).

As in the proof of Lemma~\ref{lem:fullder-rand-bounded-props}, we write
$\hhardVar_{x,v}(t) = \compactfunc(\beta(t))$ where $\beta(t) =
1-\half\norm{x+tv}^2$, and use Fa\'{a} di Bruno's
formula~\eqref{eqn:faa-di-bruno} to write, for any $k\ge 1$,
\begin{equation*}
  \hhardVar_{x,v}^{(k)}(0) 
  = \sum_{P \in \mc{P}_k}
  \compactfunc^{(|P|)}(\beta(0))
  \cdot \prod_{S \in P} \beta^{(|S|)}(0),
\end{equation*}
where $\mc{P}_k$ is the set of partitions of $[k]$ and $|P|$ denotes the
number of set in partition $P$. Noting that $\beta'(0) = -\inner{x}{v}$,
$\beta''(0) = -\norm{v}^2$ and $\beta^{(n)}(0)=0$ for any $n>2$, we have
\begin{equation*}
  \hhardVar_{x,v}^{(k)}(0) 
  = \sum_{P \in \mc{P}_{k,2}}
  (-1)^{|P|} \compactfunc^{(|P|)}\left(1-\half\norm{x}^2\right)
  \inner{x}{v}^{\countset_1(P)}
  \norm{v}^{2\countset_2(P)}
\end{equation*}
where $\mc{P}_{k,2}$ denote the partitions of $[k]$ consisting only of
subsets with one or two elements and $\countset_i(P)$ denotes the number of
sets in $P$ with precisely $i$ elements.

Noting that $\compactfunc^{(k)}(1-\half\norm{x}^2) = 0$ for any
$k\ge 0$ and $\norm{x} > 1$, we may assume $\norm{x}\le 1$. Since
$\norm{v}\le 1$, we may bound $| \hhardVar_{x,v}^{(p+1)}(0)|$ by
\begin{equation*}
  \left|\hhardVar_{x,v}^{(p+1)}(0) \right|
  \le \left|  \mc{P}_{p+1,2} \right| \cdot
  \max_{k\in[p+1]}\sup_{x\in\R} | \compactfunc^{(k)}(x)| 
  \stackrel{(i)}{\le} e^{\frac{p+1}{2}\log(p+1)}
  \cdot e^{\frac{5(p+1)}{2}\log(\frac{5}{2}(p+1))}
  \le e^{3 p \log p + c p}
\end{equation*}
for some absolute constant $c<\infty$, where inequality $(i)$ follows from
Lemma~\ref{lem:fullder-props}.\ref{item:fullder-psiphi-props-bounded} and
that the number of matchings in the complete graph (or the $k$th telephone
number~\cite[Lem.~2]{ChowlaHeMo51}) has bound $|\mc{P}_{k,2}| \le
e^{\frac{k}{2}\log k}$. This gives the result.

\subsection{Proof of Theorem~\ref{thm:fullder-final-dist}}
\label{sec:fullder-proof-final-dist}

For some $\T\in \N$ and $\sigma > 0$ to be specified, and
$d=\ceil{52\cdot230^2 \cdot T^2 \log(4 \T^2)}$, consider the function $f_U 
:\R^d
\to \R$ indexed by orthogonal matrix $U\in \R^{d\times \T}$ and defined as
\begin{equation*}
  f_U(x) = \frac{\Smp\sigma^{p+1}}{\smC{p}'}\fhardBound(x/\sigma) -  
  \frac{\Smp D^{p+1}}{\smC{p}'}\hhard(U^{\top} x/D),
\end{equation*}
where $\fhardBound(x) = \fhardRot(\rho(x)) + \frac{1}{10} \norm{x}^2$ is the 
randomized hard instance construction~\eqref{eq:fullder-fhardBound-def} with
$\rho(x) = x / \sqrt{1 + \norm{x/R}^2}$, $\hhard$ is the bump
function~\eqref{eq:fullder-distance-hhard-def} and $\smC{p}' = \smChat{p} +
\smCvar{p}$, for $\smChat{p}$ and $\smCvar{p}$ as in
Lemmas~\ref{lem:fullder-rand-bounded-props}.\ref{item:random-bounds-lipschitz}
and \ref{lem:fullder-hhard-props}.\ref{item:fullder-hhard-lipschitz},
respectively. By the lemmas, $f_U$ has $\Smp$-Lipschitz $p$th order
derivatives and $\smC{p}' \le e^{c_1 p\log p + c_1}$ for some $c_1<\infty$. 
We assume that $\sigma \le D$;
our subsequent choice of $\sigma$ will obey this constraint. 

Following our general proof strategy, we first demonstrate that $f_U \in
\FclassD{p}$, for which all we need do is guarantee that the global
minimizers of $f_U$ have norm at most $D$. By the
constructions~\eqref{eq:fullder-fhardBound-def}
and~\eqref{eq:fullder-fhard-def} of $\fhardBound$ and $\fhardRot$,
Lemma~\ref{lem:fullder-hhard-props}.\ref{item:fullder-hhard-bump} implies
\begin{flalign*}
  f_U\left((0.8D)u\ind{\T}\right)
  &  = \frac{\Smp\sigma^{p+1}}{\smC{p}'}\fhard(\rho(e\ind{\T}))
  +  \frac{\Smp\sigma^{p+1}}{10\smC{p}'}\norm{\frac{4Du\ind{\T}}{5\sigma}}^2
  - \frac{\Smp D^{p+1}}{\smC{p}'}\hhard(0.8 e\ind{\T}) \\
  & = \frac{\Smp \sigma^{p + 1}}{\smC{p}'} \fhard(0)
  + \frac{8 \Sm{p} \sigma^{p - 1} D^2}{125 \smC{p}'}
  + \frac{-\Sm{p} D^{p + 1}}{\smC{p}'}
  < -\frac{117}{125} \frac{\Sm{p} D^{p + 1}}{\smC{p}'}
  + \frac{\Smp \sigma^{p + 1}}{\smC{p}'} \fhard(0)
\end{flalign*}
with the final inequality using our assumption $\sigma \le D$. On the other
hand, for any $x$ such that $\hhard(U^{\top} x/D) = 0$, we have by
Lemma~\ref{lem:fullder-rand-bounded-props}.\ref{item:random-bounds-f-gap}
(along with $\fhardBound(0)=0)$ that
\begin{equation*}
  f_U(x) \ge \frac{\Smp\sigma^{p+1}}{\smC{p}'} \inf_x \fhardBound(x) \ge 
  - 12\frac{\Smp\sigma^{p+1}}{\smC{p}'}T
  + \frac{\Smp \sigma^{p + 1}}{\smC{p}'} \fhard(0).
\end{equation*}
Combining the two displays above, we conclude that if
\begin{equation*}
  12 \frac{\Smp \sigma^{p + 1}}{\smC{p}'} \T
  \le \frac{117}{125}\frac{\Smp D^{p + 1}}{\smC{p}'},
\end{equation*}
then all global minima $x\opt$ of $f_U$ must satisfy $\hhard(U^{\top} x\opt / 
D)
> 0$.  Inspecting the definition~\eqref{eq:fullder-distance-def} of
$\hhard$, this implies $\norm{x^\star/D - 0.8u\ind{T}} < \frac{1}{5}$, and
therefore $\norm{x^\star} \le D$. Thus, by setting
\begin{equation}
  \T = \floor{\frac{D^{p+1}}{13\sigma^{p+1}}},
  \label{eqn:T-choice-to-distance}
\end{equation}
we guarantee that $f_U\in \FclassD{p}$ as long as $\sigma \le D$.

It remains to show that, for an appropriately chosen $\sigma$, any 
randomized algorithm requires (with high probability) more than $\T$ 
iterations 
to find $x$ such that $\norms{\grad f_U(x)} <
\epsilon$. We claim that when $\sigma\le D$, for any
$x\in \R^d$,
\begin{equation}
  \label{eq:fullder-distance-rhoequiv}
  \absinner{u\ind{\T}}{\rho (x/\sigma)} < \half
  ~~\mbox{implies} ~~
  \hhard(U^{\top} y/D) = 0
  ~~ \mbox{for}~ y ~ \mbox{in a neighborhood of } x.
\end{equation}
We defer the proof of claim~\eqref{eq:fullder-distance-rhoequiv} to the end
of this section.

Now, let $U \in \R^{d \times \T}$ be an orthogonal matrix chosen uniformly
at random from $\orthogonalgroup(d,T)$.
Let $x\ind{1},\ldots, x\ind{t}$ be a sequence of iterates generated by
algorithm $\alg\in\AlgRand$ applied on $f_U$. We argue that $|\<u\ind{T},
\rho(x\ind{t} / \sigma)\>| < 1/2$ for all $t \le \T$, with high probability. To do 
so, we briefly revisit the proof of Lemma~\ref{lem:fullder-rand-slow} 
(Sec.~\ref{sec:fullder-rand-slow-proof}) where we replace $\fhardRot$ with 
$f_U$ and $x\ind{t}$ with $\rho(x\ind{t}/\sigma)$. By 
Lemma~\ref{lem:fullder-rand-slow-pu} we have that for every $t\le\T$ the 
event $G_{\le t}$ implies  $|\<u\ind{\T},\rho(x\ind{s} / \sigma)\>| < 1/2$ for all 
$s\le t$, and therefore by the claim~\eqref{eq:fullder-distance-rhoequiv} we 
have that Lemma~\ref{lem:fullder-rand-slow-deterministic} holds (as we may 
replace the terms $\hhard(U^{\top} x\ind{s}/D)$, $s<t$, with 0 whenever 
$G_{<t}$ holds). The rest of the proof of 
Lemma~\ref{lem:fullder-rand-slow-pu} proceeds unchanged and gives us that 
with probability
greater than $1/2$ (over any randomness in $\alg$ and the uniform choice of
$U$),
\begin{equation*}
\absinner{u\ind{T}}{\rho (x\ind{t}/\sigma)} < \half
~~ \mbox{for all} ~
t \le \T.
\end{equation*}
By claim~\eqref{eq:fullder-distance-rhoequiv}, this implies $\grad
\hhard(U^{\top} x\ind{t}/D)=0$, and by Lemma~\ref{lem:fullder-rand-hard},  
$\norms{\grad \fhardBound(x\ind{t}/\sigma)} > 1/2$. Thus, after
scaling,
\begin{equation*}
  \norm{\grad f_U(x\ind{t})} > \frac{\Smp \sigma^p}{2\smC{p}'}
\end{equation*}
for all $t\le \T$, with probability greater that $1/2$. 
As in the proof of Theorem~\ref{thm:fullder-final}, By taking $\sigma =
(2\smC{p}'\epsilon/\Smp)^{1/p}$ we guarantee
\begin{equation*}
\inf_{\alg\in\AlgDet}\sup_{U\in\orthogonalgroup(d,\T)}\TimeEps{\alg}{f_U} 
\ge 1+T.
\end{equation*}
where $\T = \floor{D^{p + 1} / 13 \sigma^{p + 1}}$ is defined in
Eq.~\eqref{eqn:T-choice-to-distance}. Thus, as $f_U \in \FclassD{p}$ for our 
choice of $\T$, we  immediately obtain
\begin{equation*}
  \CompEps{\AlgRand}{\FclassD{p}}
  \ge \T + 1
  \ge 
  \frac{D^{1+p}}{52}\left( \frac{\Smp}{\smC{p}'}\right)^{\frac{1+p}{p}}
  \epsilon^{-\frac{1+p}{p}},%
\end{equation*}
as long as our initial assumption $\sigma \le D$ holds.  
When $\sigma > D$, we have
that $\frac{2 \smC{p}'}{\Smp} \epsilon > D^p$,
or $1 > D^{p + 1} (\frac{\Smp}{2 \smC{p}'})^\frac{1 + p}{p}
\epsilon^{-\frac{1 + p}{p}}$, so that the bound
is vacuous in this case regardless: every method must take at least
$1$ step.

Finally, we return to demonstrate
claim~\eqref{eq:fullder-distance-rhoequiv}.  Note that
$\absinner{u\ind{\T}}{\rho (x/\sigma)} < 1/2$ is equivalent to
$\absinner{u\ind{\T}}{x} < \frac{\sigma}{2} \sqrt{1+\norms{\frac{x}{\sigma
      R}}^2}$, and consider separately the cases $\norm{x/\sigma} \le R/2$
and $\norm{x/\sigma} > R/2 = 115\sqrt{T}$. In the first case, we have
$\absinner{u\ind{\T}}{x} < (\sqrt{5}/4)\sigma < (3/5)D$, by our assumption
$\sigma \le D$. Therefore, by
Lemma~\ref{lem:fullder-hhard-props}.\ref{item:fullder-hhard-radius} we have
that $\hhard(U^{\top} y/D) = 0$ for $y$ near $x$. In the second case, we have
$\norm{x} > (4R/\sqrt{5})\absinner{u\ind{\T}}{x} > 230
\absinner{u\ind{\T}}{x}$. If in addition $\absinner{u\ind{\T}}{x} < (3/5)D$
then our conclusion follows as before. Otherwise, $\norm{x}/D >
230\cdot(3/5) > 1$, so again the conclusion follows by
Lemma~\ref{lem:fullder-hhard-props}.\ref{item:fullder-hhard-radius}.

\end{document}